\documentclass[10pt, a4paper, twoside]{amsart}
\usepackage{amsfonts}
\usepackage{mathrsfs}
\usepackage{amsmath}
\usepackage{amssymb}
\usepackage{fancyhdr}
\usepackage{graphicx}
\usepackage{color}

\numberwithin{equation}{section}

\allowdisplaybreaks

\newcommand\RR{\mathbb{R}}

\newcommand\ZZ{\mathbb{Z}}

\newcommand\Id{\mathrm{Id}}

\newcommand\Omegazh{{}^0\Omega^{1/2}}

\newcommand\diag{\mathrm{diag}}
\newcommand\diagz{\mathrm{diag}_0}

\newcommand\FL{\mathrm{FL}}
\newcommand\FR{\mathrm{FR}}
\newcommand\FF{\mathrm{FF}}
\newcommand\tL{\tilde \Lambda}
\newcommand\tHL{\tilde H^L}
\newcommand\tHR{\tilde H^R}
\newcommand\tpL{\tilde p_L}

\newcommand\FBR{\mathrm{FBR}}
\newcommand\tFBR{\widetilde{\mathrm{FBR}}}
\newcommand\tgamma{\tilde \gamma}
\newcommand\gammatf{\gamma^{2,\prime}_f}
\newcommand\gammatb{\gamma^{2,\prime}_b}
\newcommand\gammatfb{\gamma^{2,\prime}_{fb}}
\newcommand\dist{\mathrm{dist}}
\newcommand\sgn{\mathrm{sgn}}

\newcommand\blambda{{\boldsymbol\lambda}}

\newcommand\bzeta{{\boldsymbol\zeta
}}

\newtheorem{theorem}{Theorem}
\newtheorem{lemma}[theorem]{Lemma}
\newtheorem{proposition}[theorem]{Proposition}
\newtheorem{corollary}[theorem]{Corollary}
\newtheorem{definition}[theorem]{Definition}

\theoremstyle{remark}
\newtheorem{remark}[theorem]{Remark}

\begin{document}

\title{\textbf{Resolvent and Spectral Measure on Non-Trapping Asymptotically Hyperbolic Manifolds I: 
Resolvent Construction at High Energy}}

% \footnotetext {{}{\emph{2010 Mathematics Subject Classification}: 58J40, 58J50.}} \footnotetext {{}\emph{Key words and phrases}:  Asymptotically hyperbolic manifolds, semiclassical analysis, Fourier integral operator, intersecting Lagrangian distribution, resolvent.}} \setcounter{footnote}{0}


\thanks{This research was supported in part by Australian Research Council Discovery Grants DP1095448 and DP120102019}

\author{Xi Chen and Andrew Hassell}

\begin{abstract} This is the first in a series of papers in which we investigate the resolvent and spectral measure on non-trapping asymptotically hyperbolic manifolds with applications to the restriction theorem, spectral multiplier results and Strichartz estimates. In this first paper, we construct the high energy resolvent on general non-trapping asymptotically hyperbolic manifolds, using  semiclassical Lagrangian distributions and semiclassical intersecting Lagrangian distributions, together with the $0$-calculus of Mazzeo-Melrose.

Our results generalize recent work of Melrose, S\'{a} Barreto and Vasy \cite{Melrose-Sa Barreto-Vasy}, which applies to metrics close to the exact hyperbolic metric.  
We note that there is an independent work by Y. Wang \cite{Wang} which also constructs the high-energy resolvent.  \end{abstract}

\maketitle

\section{Introduction}

Given a Riemannian manifold $(M, g)$, the (positive) Laplacian $\Delta$ is defined by  $$\int (\Delta u) v \,  \text{vol}_g = \int \langle \nabla u,  \nabla v \rangle_g  \,  \text{vol}_g, \quad u, v \in C_c^\infty(M). $$ It is essentially self-adjoint on $C_c^\infty(M)$ provided $M$ is complete. 
For a complex number $\sigma \notin \text{spec}\,(\Delta)$, the resolvent $R_\Delta(\sigma)$ at $\sigma$  inverts the Laplacian in the sense $$(\Delta - \sigma) \circ R_\Delta(\sigma) = Id.$$

In this article, we work on an $n + 1$-dimensional manifold $M$ that is the interior $X^\circ$ of a compact manifold $X$ with boundary $\partial X$ and endowed with an asymptotically hyperbolic metric. A basic model is the well-known Poincar\'{e} disc, which is the ball $\mathbb{B}^{n + 1} = \{z \in \mathbb{R}^{n + 1} : |z| < 1\}$ equipped with metric 
\begin{equation}
\frac{4 dz^2}{(1 - |z|^2)^2}. 
\label{Poincare-disc}\end{equation}
%Asymptotically hyperbolic spaces, as variants of the Poincar\'{e} disc, inherit the basic features near the boundary.

Let $x$ be a boundary defining function for $X$. A metric $g$ is said to be conformally compact, if $x^2 g$ is a Riemannian metric and extends smoothly to the closure of $X$. Then the interior $X^\circ$ of $X$ is metrically complete; that is, the boundary is at spatial infinity.  Mazzeo \cite{Mazzeo-JDG-1988} showed its sectional curvature approaches $- |dx|^2_{x^2 g}$ as $x \rightarrow 0$; i.e. at `infinity'. In particular, a conformally compact metric $g$ is said to be \emph{asymptotically hyperbolic} if $- |dx|^2_{x^2 g} = - 1$ at boundary. In a collar neighbourhood of the boundary, with suitable local coordinates $(x, y) \in \RR_+ \times \RR^n$, one can write
\begin{equation}
g = \frac{dx^2}{x^2} + \frac{g_0(x, y, dy)}{x^2},
\label{g-normalform}\end{equation}
 where $x$ is a boundary defining function, and $g_0$ is a metric on the boundary but depending parametrically on $x$. For example, the metric \eqref{Poincare-disc} can be written in this form: we take as boundary defining function $\rho = (1 - |z|)(1+|z|)^{-1}$. Let $\theta$ be coordinates on $S^n$, and write the standard metric on the sphere as $d\theta^2$. Then the Poincar\'e metric takes the form 
$(d\rho^2 + \frac1{4}(1-\rho^2)^{2} d\theta^2) / \rho^2$ near $\rho = 0$, which is of the form \eqref{g-normalform}.  

The asymptotically hyperbolic metric $g$ on $X$ is said to be \emph{nontrapping} if every geodesic reaches spatial infinity (that is, $x \to 0$ along it) both forwards and backwards.

Consider the Laplacian $\Delta$, on an $(n+1)$-dimensional  asymptotically hyperbolic space $(X, g)$. The continuous spectrum of Laplacian is contained in $[n^2/4 , \infty)$, whilst the point spectrum is contained in $(0, n^2/4)$. See \cite{Mazzeo-Melrose}. In particular, the resolvent $(\Delta - n^2/4 - \lambda^2)^{-1}$ on Poincar\'{e} disc $\mathbb{B}^{n + 1}$ for $\lambda \notin \mathbb{C} \setminus [0, \infty)$ is \begin{equation}\label{hyperbolic resolvent} - \frac{1}{2i \lambda} \bigg( - \frac{1}{2\pi \sinh(r)} \frac{\partial}{\partial r} \bigg)^k e^{i \lambda r} \bigg|_{r = d(z, z^\prime)}\end{equation} where $n = 2k$ and $d(z, z^\prime)$ is geodesic distance on Poincar\'{e} disc.

There are several reasons to study the resolvent of the Laplacian on asymptotically hyperbolic spaces near the spectrum. One reason is to study the resonances by analytically continuing through the spectrum. Another reason is to obtain the spectral measure of $\Delta$ as the difference between the limit of the resolvent above and below the spectrum, according to Stone's formula.

Mazzeo and Melrose \cite{Mazzeo-Melrose} introduced the so-called $0$-pseudodifferential operators on $0$-blown-up double space $X \times_0 X$ (or $X^2_0$ for simplicity) to construct the resolvent $$R(\zeta) = \big(\Delta - \bzeta (n - \bzeta)\big)^{-1},$$ where the boundary of diagonal of double space is blown up, for example see Figure \ref{fig: blown-up double space}. 
\begin{figure}
\includegraphics[width=0.5\textwidth]{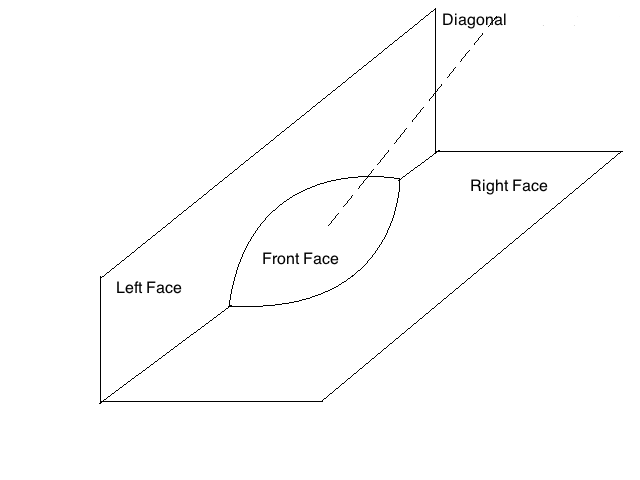}\caption{\label{fig: blown-up double space}The $0$-blown-up double space $X \times_0 X$}
\end{figure}
 The Schwartz kernel of the resolvent lies in the space 
\begin{equation}\Psi^{-2}_0(X) + \rho_L^\bzeta \rho_R^\bzeta C^\infty(X \times_0 X),
\label{MM}\end{equation} 
where $\Psi_0^{-2}(X)$ is the space of $0$-pseudodifferential operators of order $-2$ and $\rho_L$, $\rho_R$ are boundary defining functions for the left and right faces as in Figure~\ref{fig: blown-up double space}. The resolvent has a meromorphic extension in $\bzeta$ except at points $(n + 1)/2 - \mathbb{Z}_+$.  Guillarmou \cite{Guillarmou} showed that those points are at most poles of finite multiplicity if and only if the metric is even.

In terms of the $\bzeta$ variable, the spectrum is located at $\bzeta = n/2 + i \RR$. For definiteness, we shall consider here the \emph{outgoing} resolvent, corresponding to $\bzeta = n/2 - i\blambda$ for $\blambda > 0$. Exactly the same methods apply to construct the incoming resolvent, with $\blambda < 0$. 
As well as the resolvent at fixed $\bzeta = n/2 - i \blambda$, we are concerned about the asymptotic behaviour of the resolvent as parameter $\blambda$ approaches infinity, which is called the high energy limit. We shall introduce $h = \blambda^{-1}$ and multiply through by $h^2$ to write the operator in semiclassical form
\begin{equation}\label{semiclassical eqn}
 P_h A_h = \Id, \quad P_h = h^2 \Delta - h^2 \frac{n^2}{4} - 1.
\end{equation}
We view the operator $P_h$ in \eqref{semiclassical eqn} as a semiclassical differential operator. The semiclassical symbol is defined as follows: given local coordinates $z_1, \dots, z_{n+1}$ and dual coordinates $\zeta_1, \dots, \zeta_{n+1}$ on the cotangent bundle, the symbol of $a(z)(-i h \partial_{z})^\alpha$ is $a(z)\zeta^\alpha$, and is extended to differential operators by linearity. The \emph{principal} symbol, in the semiclassical sense, is obtained by taking only the leading order terms in $h$. So the principal symbol of $P_h$ is $g^{ij} \zeta_i \zeta_j  - 1$.\footnote{Here we use the summation convention. Namely, $g^{ij} \zeta_i \zeta_j$ denotes $\sum_{i, j = 1}^{n + 1}g^{ij} \zeta_i \zeta_j$.}
We define the characteristic variety $\Sigma \subset T^* X^\circ$ to be the zero set of the symbol: $\Sigma = \{ (z, \zeta) \mid p = 0 \}$.
We see from this that the principal symbol $p$ of $P_h$ satisfies $dp \neq 0$ when $p = 0$. Equivalently, the Hamilton vector field $H_p$ is nowhere vanishing when $p=0$. This is clear in the present case, as the Hamilton vector field at $\{ p = 0\}$  is precisely the generator of geodesic flow on the cosphere bundle.
We emphasize that, because of the scaling in $h$, the `true' frequency represented by $\zeta$ is actually $\zeta/h$, which tends to infinity as $h \to 0$ whenever $\zeta \neq 0$. Because of this, semiclassical analysis is the study of the high-frequency or short wavelength limit.

From the perspective of microlocal analysis, we are concerned about the behaviour of a distribution $A$ defined on a manifold $M$ on its wavefront set. Microlocally solving semiclassical equation (\ref{semiclassical eqn}) amounts to understanding the solution $A$ on its semiclassical wavefront set $\text{WF}_h (A)$, which is a subset of the cotangent space giving both the spatial locations and the directions in which  the Schwartz kernel of $A$ is singular. Here, by `singularity' is meant nontrivial behaviour (that is, not $O(h^\infty)$) as $h \to 0$. More precisely,  we say $(z_0, \zeta_0) \notin \text{WF}_h (A)$ if there exist $\phi, \psi \in C_c^\infty$ with $\phi = 1$ near $z_0$ respectively $\psi = 1$ near $\zeta_0$ such that $$\psi \mathcal{F}_h (\phi A) = O(h^\infty), $$
where $\mathcal{F}_h$ is the semiclassical Fourier transform
$$\mathcal{F}_h f = \int_{M} e^{- i \langle z, \zeta \rangle / h} f(z) dz.$$

We ask what is the (semiclassical) wavefront set of a Schwartz kernel $A_h$ satisfying \eqref{semiclassical eqn}. By elliptic estimates, the wavefront set is contained in the union of the wavefront set of the identity operator (that is, the conormal bundle to the diagonal $N^* \diag$, given in local coordinates by $\{ (z, \zeta; z, -\zeta) \}$), together with the characteristic variety $\Sigma$. Under the additional condition that $(X, g)$ is nontrapping, the propagation of singularity theorem due (in the setting of homogeneous operators) to Duistermaat and H\"ormander \cite{Duistermaat-Hormander-Acta-1972} says that there will be (at least microlocal) solutions to \eqref{semiclassical eqn} $A_h^\pm$ with wavefront set 
 \emph{equal to} the wavefront set of the identity operator $N^* \diag$, together with the union of bicharacteristics emanating in the forward (for $A_h^+$) or backward (for $A_h^-$) direction from $N^* \diag \cap \Sigma$. Here, by bicharacteristics we mean the integral curves of the Hamilton vector field $H_p$, i.e. geodesics (in the cotangent bundle) in the present case.  It is crucial fact for us that these sets are smooth Lagrangian submanifolds of the cotangent bundle. This is clear in the case of $N^* \diag$, while in the case of the forward bicharacteristic relation, which may be expressed in the form
\begin{multline} 
\FBR = \{(z, \zeta, z', -\zeta^\prime) \mid (z, \zeta) \in \Sigma, (z, \zeta) \text{ is obtained by flowing along an $H_p$-integral curve }   \\
\mbox{in the forward direction, starting at } (z', \zeta') \}
\label{Lambda1def}\end{multline}
this follows from some standard symplectic geometry, together with the nontrapping condition (which provides the `pseudoconvexity' condition of Duistermaat-H\"ormander \cite{Duistermaat-Hormander-Acta-1972}). We shall show that there is a unique exact solution $A_h^+$ with wavefront set equal to $N^* \diag \cup \FBR$, and it is precisely the outgoing resolvent $R(n/2 - i/h)$. 
%It turns out that there are unique exact solutions $A_h^\pm$ with these microlocal properties, and they are equal to $R(n/2 \pm i/h)$ --- see the end of Section~\ref{sec:parametrix}. 

Given that the resolvent $A_h$ has wavefront set in the union of two Lagrangian submanifolds, one can expect that the resolvent itself is some sort of Lagrangian distribution, otherwise known as a Fourier integral operator or WKB expression. According to the classical theory of non-degenerate Fourier integral operators formulated by H\"{o}rmander \cite{Hormander-Acta-1971}, a Lagrangian distribution is essentially determined by its phase function and symbol. The phase function, in some sense, is determined by the Lagrangian submanifold. Once one understands the phase, the symbol can be obtained by solving certain transport equations along the Hamilton vector field $H_p$ arising from the symbol $p$ of $P_h$.

In the present case, the phase function for the flow-out Lagrangian $\FBR$ is given by the geodesic distance function, at least close to the diagonal.  A simple example is the Poincar\'e disc $\mathbb{B}^{n + 1}$, on which the resolvent (\ref{hyperbolic resolvent}) is a Lagrangian distribution whose phase function is visibly given by the geodesic distance function. More generally, Melrose, S\'{a} Barreto and Vasy \cite{Melrose-Sa Barreto-Vasy} constructed the high-energy resolvent for hyperbolic metric with small perturbation near boundary, which is a simple version of asymptotically hyperbolic space. The advantage of this condition is that the sectional curvature is globally negative so that geodesic distance function is smooth, away from the diagonal,  and globally parametrizes the bicharacteristic flow. The geometric significance of this is that the forward bicharacteristic relation $\FBR$ is a Lagrangian submanifold which is a graph over the base manifold; analytically, the consequence is that the resolvent is a Fourier integral operator in the form of an oscillatory function like (\ref{hyperbolic resolvent}) on $\mathbb{B}^{2k + 1}$. 

However, in the present setting, the only geometric restriction on our asymptotically hyperbolic manifold is that it is nontrapping. In particular, conjugate points may occur. 
The geodesic distance function fails to be smooth in a neighbourhood of such points. In this case, the geodesic distance function fails to globally parametrize the Lagrangian. In fact, what happens is that the Lagrangian $\Lambda_1^+$ remains smooth, but it no longer projects diffeomorphically to the base, i.e. it is no longer a graph over the base manifold $X^\circ$. In this case,  some other variables are needed in the parametrization. A microlocal solution $A_h$ can then be written locally in the form of a semiclassical oscillatory integral over the extra variables. This is described in great detail in H\"ormander's work \cite{Hormander-Acta-1971} in the homogeneous case. See Appendix~\ref{app:sld} for the (routine) adaptation to the semiclassical setting.

We now state a crude version of the the main result: the full version is stated in  Theorem \ref{asymptotically hyperbolic resolvent}. 

\begin{theorem} The outgoing resolvent, $R(n/2 - i/h)$, for the semiclassical operator $P_h = h^2 \Delta - h^2 n^2 /4 - 1$ is the sum of a semiclassical 0-pseudodifferential operator and a Lagrangian distribution associated to $N^* \diag$ and $\FBR$. 
\end{theorem} 

\begin{remark}
A similar statement is true for the incoming resolvent $R(n/2 + i/h)$, with the backward bicharacteristic relation replacing $\FBR$. We work with the outgoing resolvent in this paper only, to simplify the notation and exposition. 
\end{remark}

 \begin{remark} As the result indicates, we apply two kinds of semiclassical calculus to invert Laplacian $P_h = h^2 \Delta - h^2 n^2 / 4 - 1$. The part away from the  characteristic variety on the $0$-cotangent bundle will be resolved by the semiclassical version of the $0$-calculus due to Mazzeo and Melrose. To deal with the region near $N^* \diag \cap \FBR$, we invoke the semiclassical version of intersecting Lagrangian distribution theory due to Melrose and Uhlmann \cite{Melrose-Uhlmann-CPAM-1979}.
 \end{remark}
 
 \begin{remark} We could equally well deal with a Schr\"odinger operator of the form $h^2 \Delta + V - h^2 n^2/4 -1$, where $V$ is a smooth function on $X$ vanishing at the boundary (and much less regular potentials could also be considered). In this case, the bicharacteristic flow would be with respect to the Hamiltoniam $|\zeta|^2_g + V - 1$, and the nontrapping assumption would apply to such bicharacteristics (rather than the geodesic flow on $X$). However, for simplicity, we shall deal only with the case $V=0$, which suffices for our applications in \cite{Chen-Hassell2} and \cite{Chen-Hassell3}. 
 \end{remark}
 
In the second article in this series, \cite{Chen-Hassell2}, we shall use this result to study the spectral measure for $\Delta$. We do this via Stone's formula, which expresses the spectral measure in terms of the difference between the outgoing and incoming resolvent kernels along the spectrum. Estimates on the kernel of the spectral measure give a variety of results on restriction theorems and spectral multipliers. We will show, for example, 

\begin{theorem}\label{thm:spec-mult}\cite{Chen-Hassell2}
Let $X$ be an asymptotically hyperbolic nontrapping manifold, let $L = (\Delta - n^2/4)_+$,
and let $dE _{\sqrt{L}}(\lambda)$ denote the spectral measure for the operator $\sqrt{L}$. Then  we have the restriction theorem, 
$$\|dE_{\sqrt{L}}(\lambda)\|_{L^p \rightarrow L^{p^\prime}} \leq 
\begin{cases} 
C \lambda^2, \qquad  \qquad \qquad \quad \lambda < 1, \ 1 \leq p < 2 \\
C \lambda^{(n + 1) (1/p - 1/p^\prime) - 1}, \  \lambda \geq 1, \ 1 \leq p \leq \frac{2(n + 2)}{n + 4}, \\  
C \lambda^{n(1/p - 1/2)}, \qquad \quad \lambda \geq 1, \ \frac{2(n+2)}{n+4} \leq p < 2. 
\end{cases}
$$
%Moreover, for $F \in H^s$, $s > n/2$, supported in $[0, 1]$, and $R \geq 0$, 
%the spectral multiplier $F(\sqrt{\Delta}/R)$ is a bounded operator on $L^p + L^2$ for $1 < p \leq 2$, uniformly for $R \in [1, \infty)$. 
Here $C$ depends on $(X, g)$ and $p$. 
\end{theorem}

In the third article in this series, \cite{Chen-Hassell3},  the first-named author will use the spectral measure estimates to prove Strichartz estimates for the Schr\"odinger equation on $X$, namely
\begin{theorem}\label{thm:Strichartz}\cite{Chen-Hassell3}
Let $X$ be as above, and assume that $\Delta$ has no discrete spectrum.  
For the Cauchy problem of Schr\"{o}dinger equation 
\begin{equation*}\label{schrodinger cauchy} \left\{ \begin{array}{c} i \frac{\partial}{\partial t} u  +  \Delta u = F(t, z)  \\  u(0, z) = f(z) \end{array} \right. ,\end{equation*}  we have 
$$\|u\|_{L^p(\mathbb{R}, L^q(X))} \leq C \Big( \|f\|_{L^2(X)} + \|F\|_{L^{\tilde{p}^\prime}(\mathbb{R}, L^{\tilde{q}^\prime}(X))} \Big)
$$ provided that $(q, r)$ and $(\tilde{q}, \tilde{r})$ are \emph{hyperbolic} Schr\"{o}dinger admissible pairs, that is, such that 
\begin{equation}\label{schrodinger admissibility} \frac{2}{q} + \frac{n + 1}{r} \geq \frac{n + 1}{2}, \quad q, r \geq 2, \quad (q, r) \neq (2, \infty).\end{equation}
\end{theorem}

\begin{remark} In both Theorems~\ref{thm:spec-mult} and \ref{thm:Strichartz}, the results hold for a greater range of spaces compared to Euclidean space. In Theorems~\ref{thm:spec-mult}, it is well known that the corresponding result on $\RR^{n+1}$ holds only in the range $1 \leq p \leq \frac{2(n + 2)}{n + 4}$, while for the Euclidean Strichartz estimates, a necessary condition is to have equality in \eqref{schrodinger admissibility}. 
\end{remark}

%A very similar resolvent with ours is constructed in the paper by the method in the spirit of Melrose,  S\'{a} Barreto and Vasy \cite{Melrose-Sa Barreto-Vasy}. Wang, following Melrose,  S\'{a} Barreto and Vasy's idea, blows up the diagonal of semiclassical face $\{h = 0\}$. They study the Taylor's expansion of the resolvent near  the new face, whilst it is an elliptic pseudodifferential operator problem away from the face. The errors are removed face by face by normal operator arguments. 
Our approach is symbolic and essentially in the flavour of H\"{o}rmander \cite{Hormander-Acta-1971}, Duistermaat and  H\"{o}rmander \cite{Duistermaat-Hormander-Acta-1972}, Melrose and Uhlmann \cite{Melrose-Uhlmann-CPAM-1979}, Hassell and Wunsch \cite{Hassell-Wunsch}, and Melrose, S\'{a} Barreto and Vasy \cite{Melrose-Sa Barreto-Vasy}.
We also remark that  there is an independent work by Yiran Wang \cite{Wang} based on  S\'a Barreto-Wang \cite{SBY}, also studying the semiclassical resolvent on asymptotically hyperbolic space with application to radiation fields.

The paper is organized as follows. First of all, the $0$-geometry and $0$-calculus full  is briefly reviewed  in Section \ref{sec:calculus}. We shall understand the smoothness and parametrization of the flow-out Lagrangian $\FBR$ near the boundary and determine the form of the phase function in Section \ref{sec:flowout}.  In Section \ref{sec:parametrix} we shall construct the full parametrix and the resolvent.  For completeness, we establish the framework of semiclassical Fourier integral operators and semiclassical intersecting Lagrangian distributions in the appendices.

The authors would like to thank C. Guillarmou, F. Rochon and A. Vasy for various helpful discussions while working on this paper.

\section{The $0$-geometry and $0$-calculus}\label{sec:calculus}

Introduced by Mazzeo-Melrose, $0$-geometry is the geometry of a conformally compact metric, which shares the fundamental singularity at the boundary of an asymptotically hyperbolic metric. The boundary behaviour leads to a discussion of the $0$-cotangent bundle and corresponding theory of $0$-pseudodifferential operators.

\subsection{The $0$-cotangent and $0$-tangent bundles}\label{subsec:0bundle}

%The novelty of the $0$-cotangent bundle  only occurs near boundary; in the interior of $X$, it is canonically isomorphic to the usual cotangent bundle. 

%To introduce the so-called $0$-cotangent bundle, it is convenient to choose local coordinates $\{(x, y_1, \cdots, y_n)\}$ near boundary with boundary defining function $x$ whilst we choose $(n + 1)$-dimensional local coordinates $\{z = (z_1, \dots, z_{n+1} )\}$ away from the boundary. Additionally, we use local coordinates $\{(x, y, \xi, \eta)\}$ near the boundary cotangent bundle $T^\ast X$. \textcolor[rgb]{1,0,0}{The simplest example to display this kind of coordinate is $$\mathbb{H}^{n + 1} = \bigg(\mathbb{R}_+^{n + 1} , \, \frac{dx^2 + dy^2}{x^2}\bigg).$$}

To motivate the $0$-cotangent bundle, consider the hyperbolic Laplacian on the upper half plane $\mathbb{H}^{n + 1}$, with respect to the standard hyperbolic metric $(dx^2 + dy^2)/x^2$. This is  $$- \bigg( x\frac{\partial}{\partial x} \bigg)^2 - \sum_{j = 1}^n \bigg( x\frac{\partial}{\partial y_j} \bigg)^2 + n x \frac{\partial}{\partial x}.$$  In the usual sense, its symbol is $$x^2\xi^2 + x^2|\eta|^2 \quad \mbox{on $T^\ast X$.}$$ Since it is not elliptic as $x \rightarrow 0$, the standard elliptic theory does not apply at the boundary of $T^\ast \mathbb{H}^{n + 1} $.
However if we work on a larger bundle, the $0$-cotangent bundle ${}^0T^\ast \mathbb{H}^{n + 1}$, whose sections are spanned by the basis  $$\bigg\{\frac{dy_1}{x}, \cdots, \frac{dy_n}{x}, \frac{dx}{x}\bigg\},$$ the symbol of the Laplacian on $\mathbb{H}^{n + 1}$ is uniformly elliptic on ${}^0 T^\ast X$. In fact, any cotangent vector $\xi dx + \eta dy$ can be viewed as a $0$-cotangent vector $\lambda dx/x + \mu dy/x$ under the natural inclusion $T^\ast X \rightarrow {}^0 T^\ast X$, from which we see that $\xi = \lambda/x$ and $\eta = \mu/x$. So the symbol of the hyperbolic Laplacian is transformed to $$\lambda^2 + |\mu|^2  \quad \mbox{on ${}^0 T^\ast X$,}$$ which is elliptic uniformly down to $x=0$. 

For any manifold $X$ with boundary, and boundary defining function $x$, we thus define the $0$-cotangent bundle to be that bundle whose sections are
spanned by one-forms of the form $\alpha/x$, where $\alpha$ is a smooth one-form on $X$. In a similar way, we find that the symbol of the Laplacian with respect to any conformally compact metric on $X$ is uniformly elliptic on the $0$-cotangent bundle.

The $0$-tangent bundle is the dual of the $0$-cotangent bundle  ${}^0TX$. Its sections are smooth vector fields that vanish at the boundary, i.e. can be written in the form of $x V$ 
where $V$ is a smooth vector field on $X$. 
Such vector fields we call $0$-vector fields, which form a Lie algebra.  A $0$-differential operator is then defined to be one that can be expressed as a sum of $k$-fold products of $0$-vector fields. The supremum of $k$ in the sum is defined to be the order of the operator. So the hyperbolic Laplacian is a $0$-differential operator of order $2$.

\subsection{The $0$-blowup on $X \times X$}\label{subsec:blowup}

Thinking of the resolvent in terms of its Schwartz kernel, which is a distribution on the product of the manifold with itself, we work on $X \times X$. In order to write the resolvent near the boundary of the diagonal in geometrically natural coordinates,  we perform the so-called $0$-blowup as mentioned in the introduction. 

Generally, given a p-submanifold\footnote{A p-submanifold of a manifold with corners is a submanifold $S$ with the following property. In a neighbourhood of any point $s \in S$, there are local coordinates of the form  $x_1, \dots, x_k, y_1, \dots, y_l$, where $x_i$ are boundary defining functions and $y_i \in (-\epsilon, \epsilon)$ such that $S$ is locally given as the vanishing of some subset of these coordinates} $Y$ of a manifold with corners $Z$, the real blown-up space $[Z; Y]$ is defined (as a set) by $(Z \setminus Y) \cup  SN_+\, Y,$ where $SN_+$ denotes the inward-pointing
%\footnote{\textcolor[rgb]{1,0,0}{The inward-pointing part of a tangent space at the boundary refers to the half pointing into the interior of the manifold. The 'doubly inward-pointing' part at the corner is the intersection of the two inward-pointing tangent spaces at each boundary.}} 
part of the spherical normal bundle. There is a natural differential structure on $[Z;Y]$ making it a smooth manifold with corners. 
For more information, see \cite{DiffAnal}. 

In our circumstance of double asymptotically hyperbolic space $X \times X$, we blow up the boundary of diagonal $\partial \text{diag}$ locally expressed by $\partial \text{diag} = \{(0, y, 0, y)\}$ to produce the manifold $$X \times_0 X = SN_{+} (\partial \text{diag}) \cup \big((X \times X) \setminus \partial \text{diag}\big).$$ 
Since $\partial \text{diag}$ lies in a codimension 2 face of $X^2$, the fibres of $SN_+ (\partial \text{diag})$ are quarter-spheres of dimension $n+1$.  

The space $X^2_0$ is a manifold with corners of codimension three. The boundary hypersurfaces we denote $\FL$, for left face (the lift to $X^2_0$ of $\partial X \times X$); $\FR$, for right face (the lift to $X^2_0$ of $X \times \partial X$); and $\FF$, for front face, created by the blowup. See Figure~\ref{fig: blown-up double space}. We also denote by $\diagz$ the lift of the diagonal to $X^2_0$; notice that $\diagz$ is a p-submanifold of $X^2_0$. 
In contrast, $\text{diag} \subset X^2$  is not a p-submanifold. 

The $0$-blowdown map is the natural map $$\beta : X \times_0 X \longrightarrow X \times X.$$ We will often use the notation $\rho_L$, $\rho_R$, $\rho_F$ for boundary defining functions for the left boundary $\FL$, the right boundary $\FR$ and the front face $\FF$, respectively, without necessarily specifying a particular function.

We now write down coordinate systems in various regions of $X^2_0$, in terms of coordinates $(x,y) = (x, y_1, \dots, y_n)$ near the boundary of $X$, or $z = (z_1, \dots, z_{n+1})$ in the interior of $X$.  
Below, the unprimed coordinates indicate those lifted from the left factor of $X$, and primed coordinates indicate those lifted from the right factor. 
We label these different regions as follows:
\begin{itemize}
\item{Region 1:} In the interior of $X^2_0$. Here we use coordinates $$(z, z') = (z_1, \dots, z_{n+1}, z'_1, \dots, z'_{n+1}).$$ 

\item{Region 2a:} Near $\FL$ and away from $\FF$ and $\FR$. 
In this region, we use $(x, y, z')$. 
\item{Region 2b:} Near $\FR$ and away from $\FF$ and $\FL$. 
Symmetrically,  we use $(z, x', y')$. 

\item{Region 3:} Near $\FL \cap \FR$ and away from $\FF$. 
Here we use $(x, y, x', y')$. 

\item{Region 4a:} Near $\FF$ and away from $\FR$.
This is near the blowup. In this region we can use $s = x/x'$ for a boundary defining function for $\FF$. 
$$
s = \frac{x}{x'}, \ x', \ y, \ Y = \frac{y' - y}{x'}. 
$$
\item{Region 4b:} Near $\FF$ and away from $\FL$.
Symmetrically, we use 
$$
s' = \frac{x'}{x}, \ x, \ y', \ Y' = \frac{y - y'}{x}. 
$$
\item{Region 5:} Near the triple corner $\FL \cap \FF \cap \FR$.
In this case, a boundary defining function for $\FF$ is $|y'-y|$. By rotating the $y$ coordinates, we can assume that $|y'_1 - y_1| \geq c |y' - y|$ in a neighbourhood of any given point in the triple corner. Assuming this, we use coordinates
$$
s_1 = \frac{x}{y_1' - y_1}, \ s_2 = \frac{x'}{y_1' - y_1}, \ t = y_1' - y_1, \ Z_j = \frac{y_j^\prime - y_j}{y_1^\prime - y_1} \, (j > 1).
$$
\end{itemize}

\subsection{$0$-pseudodifferential operators}

To invert $P_h$ as a $0$-differential operator of order $2$, we shall employ the $0$-calculus, due to Mazzeo and Melrose \cite{Mazzeo-Melrose}. To make this paper more  self-contained, we briefly review their arguments and develop the corresponding semiclassical theory as it will be needed.

Recall the $0$-vector fields, the Lie algebra generated by the smooth sections of $0$-tangent bundle. The space of $k$-th order $0$-differential operators, ${}^0 \text{Diff}^k$, consists of the sum of at most $k$-fold products of $0$-vector fields. To be more explicit, one can write a $0$-differential operator of $k$-th order near the boundary as $$P = \sum_{j + |\alpha| = 0}^k p_{j, \alpha}(x, y) \bigg(x \frac{\partial}{\partial x}\bigg)^j \bigg(x\frac{\partial}{\partial y}\bigg)^\alpha.$$ 
Such an operator has a symbol which is a smooth function on the $0$-cotangent bundle, polynomial in each fibre. Conversely, each such function on the $0$-cotangent bundle is the symbol of a $0$-differential operator. 
 Clearly the Laplacian on asymptotically hyperbolic space is a $0$-differential operator of order $2$. 

So-called $0$-pseudodifferential operators are the microlocal generalization of $0$-differential operators, essentially obtained by replacing polynomial symbols on the $0$-cotangent bundle with general symbols. However, we shall give a definition in terms of the Schwartz kernel. The Schwartz kernel of a $0$-pseudodifferential operator is a distributional section of the half density bundle on the blown up double space. Let $\tilde g(x, y)$ be a usual Riemannian metric on manifold $M$ and $x \in C^\infty(M)$ be a positive defining function for the boundary. Consider the metric in the interior of $M$, 
$$
g_{ij}(x, y) = x^{-2}\tilde g_{ij}(x, y).
$$ 
The Riemannian density is of the form 
$$
|dg| := \sqrt{\det \tilde g_{ij}(x, y)} \Big| \frac{dx}{x}\frac{dy}{x^n} \Big|,
$$ 
which is singular at the boundary. The $C^\infty$ multiples of such a density are the smooth sections of a vector bundle, ${}^0 \Omega$, while the $C^\infty$ multiples of the half density
$$
|dg|^{1/2} = {\det \tilde g_{ij}(x, y)}^{1/4}\bigg|\frac{dx}{x}\frac{dy}{x^n}\bigg|^{1/2}
$$ 
are the smooth sections of the $0$-half-density bundle ${}^0 \Omega^{1/2}$. This half density bundle can be written in terms of ordinary half density bundle $\Omega^{1/2}(M)$ with a boundary defining function $\rho$ as $${}^0 \Omega^{1/2}(M) = \rho^{- n -1}\Omega^{1/2}(M).$$ 

We let ${}^\Phi T^* X^2_0$ be the bundle ${}^0 T^* X \times {}^0 T^* X$ lifted to $X^2_0$ via the blowdown map, and denote by ${}^\Phi \pi : {}^\Phi T^* X^2_0 \to X^2_0$ the bundle projection. 
With some abuse of notation, we denote by $\Omegazh(X^2)$ the tensor product $\Omegazh(X) \otimes \Omegazh(X)$, and the lift of this bundle to $X^2_0$ we denote $\Omegazh(X^2_0)$. This is related to $\Omega(X^2_0)$ by 
\begin{equation}
\Omegazh(X^2_0) = (\rho_{\FL}\rho_{\FR}\rho_{\FF})^{-(n+1)} \Omega(X^2_0). 
\label{Omega}\end{equation}
%Here we work on the half density ${}^0 \Omega^{1/2} (X \times_0 X)$. 

A $m$-th order $0$-pseudodifferential operator acting on half-densities is defined in terms of its Schwartz kernel: it is a distributional section of ${}^0 \Omega^{1/2}(X \times_0 X)$ conormal, of order $m$, to the $0$-diagonal, and vanishing to infinite order at all faces except $\FF$. The space of such operators is denoted ${}^0 \Psi^m(X \times_0 X)$. 

To be more explicit, a $0$-pseudodifferential operator of $m$-th order has the usual oscillatory integral representation locally near the diagonal away from $\FF$, and  a local expression near $\FF$ of the form 
$$
\int_{X} e^{i \big( (x - x^\prime) \cdot \lambda + (y - y^\prime) \cdot \mu \big) / x} a(x, y, \lambda, \mu) \, d\lambda d\mu |dg dg'|^{1/2}, 
$$ 
with $a (x, y, \lambda, \mu) \in S^m({}^0T^* X)$, which can be thought of as the (boundary rescaled) Fourier transform of symbol $a$. Since $(x-x')/x$ and $(y-y')/x$ are smooth defining functions for $\diagz$ near $\diagz \cap \FF$,  it is easy to see the conormality of the $0$-pseudodifferential operator at $\diagz$, as well as the rapid vanishing at the left and right boundaries $\FL$ and $\FR$, because the phase is non-stationary there.

We extend the usual notion of the symbol of a pseudodifferential operator. For any space of conormal distributions, there is a principal symbol isomorphism  map, $$\sigma_l : {}^0 \Psi^l(X \times_0 X) / {}^0 \Psi^{l - 1} \longrightarrow S^l\big({}^0T^\ast (X \times_0 X)\big) / S^{l - 1},$$ and we have the $0$-symbol calculus $${}^0\sigma_{l_1 + l_2}(A_1 \circ A_2) = {}^0\sigma_{l_1}(A_1){}^0\sigma_{l_2}(A_2),$$ where $A_1 \in {}^0 \Psi^{l_1}$ and $A_2 \in {}^0 \Psi^{l_2}$. The symbol map gives an exact sequence $$0 \longrightarrow {}^0\Psi^{l - 1}(X) \longrightarrow {}^0\Psi^l(X) \longrightarrow S^l\big({}^0T^\ast (X \times_0 X)\big) / S^{l - 1} \longrightarrow 0 .$$

A semiclassical $0$-pseudodifferential operator on $X$ of differential order $m$ and semiclassical order $k$ has a Schwartz kernel depending parametrically on $h \in (0, h_0)$, which has the usual semiclassical form locally near the diagonal and away from $\FF$ (see Appendix~\ref{app:sld}). Near $\diagz \cap \FF$, it  takes the form  
$$
h^{ - (n + 1+k)} \int_{X} e^{i \big( (x - x^\prime) \cdot \lambda + (y - y^\prime) \cdot \mu \big) / (x h)} a(h, x, y, \lambda, \mu) \, d\lambda d\mu \, |dg dg'|^{1/2} 
$$ 
with $a (h, x, y, \lambda, \mu)$ an element of $S^m({}^0T^* X)$ uniformly in $h$. The space ${}^0 \Psi^{m, k}_h(X)$ consists of such operators.  Here the first superscript denotes the differential order, whilst the second denotes the semiclassical order.

\subsection{Boundary terms}

The Laplacian with respect to a $0$-metric is elliptic in the $0$-calculus, and the usual construction therefore produces an inverse modulo an error term in ${}^0 \Psi^{-\infty}(X \times_0 X)$. 
However, such an error term is not compact; to construct a parametrix with compact error, boundary terms (that is, nontrivial expansions at the left and right boundary) are required. 

We define, ${}^0\Psi^{- \infty, m_l, m_r}(X)$, the space conormal distributions of order $(m_l, m_r)$ to left and right boundaries as the tensor product of $$\bigg\{u \in C^{-\infty}(X) : \prod_{j = 1}^N L_j u \in \rho_L^{m_l} \rho_R^{m_r} L^\infty(X), \forall N \in \mathbb{N} \bigg\}$$ with $C^\infty(X \times_0 X; {}^0\Omega^{1/2})$, where $L_j$s are vector fields tangent to left and right boundaries.

Then the full space of $0$-pseudodifferential operators of order $(m, m_l, m_r)$ is defined as $${}^0 \Psi^{m, m_l, m_r}(X, {}^0\Omega^{1/2}) = {}^0 \Psi^m(X, {}^0\Omega^{1/2}) + {}^0 \Psi^{- \infty, m_l, m_r}(X, {}^0\Omega^{1/2}).$$ It can be composed with differential operators $${}^0\text{Diff}^k(X, \Omega^{1/2}) \cdot {}^0\Psi^{m, m_l, m_r}(X, {}^0\Omega^{1/2}) \subset  {}^0\Psi^{m + k, m_l, m_r}(X, {}^0\Omega^{1/2}).$$

\subsection{Normal operator}

One may note there is no index family at the front face. That is because the error at the front face can be solved away by solving iterated normal operator equations.

Roughly speaking, the normal operator is the restriction of the kernel of  the operator on $X \times_0 X$ to the front face. To state the results, we have to introduce some notions. The front face $\FF$ is a bundle over the boundary of $X$.  We denote the fibre over $p \in \partial X$ by $F_p$, and its interior by $F^\circ_p$. 
Then $F_p^\circ$ has extra structure. 
To describe this, we let $X_p$ denote the inward pointing half of the tangent space $T_pX$, that is, the inward pointing connected component of $T_pX \setminus T_p(\partial X)$. 
Let $G_p$ be the subgroup of linear transformations of $T_pX$ consisting of the elements which preserve $X_p$ and leave the boundary $T_p(\partial X)$ fixed pointwise. This group $G_p$ is isomorphic to the semidirect product $\mathbb{R}^+_s \times \mathbb{R}^n_v,$ and acts on $X_p$ with coordinates $(x, y), x > 0, y \in \RR^n$ as follows: if $\gamma = (s, v) \in G_p$ then 
$$
\gamma \cdot (x, y) = (sx, y + xv).
$$ 
This action is transitive on $X_p$.

It turns out that $F_p^\circ$ has two natural identifications with $X_p$, given in local coordinates by the coordinates $(s, Y) \mapsto s \partial_x + Y \cdot \partial_y$ or $(s', Y') \mapsto s' \partial_{x'} + Y' \cdot \partial_{y'}$ in regions 4a and 4b above. 
We can think of these as identifications with $X_p^l$ or $X_p^r$, the left or right copy of $X_p$, respectively. 
As $G_p$ acts transitively on $X_p$, and $F_p^\circ$ has a distinguished point, namely the intersection of $F_p$ with $\diagz$, this gives us two natural group structures on $F_p^\circ$. 
Moreover, using the group structure derived from the identification with $X_p^l$, we find that the left-invariant metric on $F_p^\circ$ is (up to scaling) the standard hyperbolic metric on the upper half space model $\{ (s, Y) \mid s \geq 0 \}$ of hyperbolic space. 
For more details, see Sections 2 and 3 of \cite{Mazzeo-Melrose}.

If $B \in {}^0\Psi^{- \infty, m_l, m_r}(X, {}^0\Omega^{1/2})$, the normal operator is defined as $$N_p(B) = B|_{{F}_p} \in {}^0\Psi^{- \infty, m_l, m_r}({F}_p) \otimes {}^0\Omega^{1/2}(X^l_p) \otimes {}^0\Omega^{1/2}(X^r_p).$$

For each $p \in \partial X$, the normal operator can be interpreted as an operator on $F_p^\circ$. For instance, the normal operator of Laplacian on $X$ with respect to metric $g$ is indeed the Laplacian on hyperbolic space $F_p^\circ$.  Moreover, Mazzeo and Melrose showed \begin{proposition}[\cite{Mazzeo-Melrose}] The normal operator defines an exact sequence \begin{eqnarray*}0 \longrightarrow \rho_F \, {}^0\Psi^{- \infty, m_l, m_r}(X, {}^0\Omega^{1/2}) \longrightarrow  {}^0\Psi^{- \infty, m_l, m_r}(X, {}^0\Omega^{1/2}) \\\longrightarrow  {}^0\Psi^{- \infty, m_l, m_r}(F_p) \otimes {}^0\Omega^{1/2}(X^l_p) \otimes {}^0\Omega^{1/2}(X^r_p) \longrightarrow 0\end{eqnarray*} such that 
\begin{equation}
N_p(P \circ B) = N_p(P) \circ N_p(B),
\label{normalop-comp}\end{equation}
 provided $P \in {}^0\text{Diff\,}^m (X, {}^0 \Omega^{1/2})$. 
 \end{proposition}

\begin{remark} The composition of two normal operators can be considered convolution with respect to the group structure on $F_p^\circ$. 
\end{remark}

This suggests that  the boundary behaviour of the resolvent kernel is governed by the hyperbolic Laplacian. We work on Poincar\'{e} disc model $\mathbb{B}^{n + 1}$ with boundary defining function $\rho$. By explicit calculation over Green function, Mazzeo and Melrose proved that for $k \geq 1 \in \ZZ$, 
\begin{proposition}[\cite{Mazzeo-Melrose}] The hyperbolic resolvent $R_{\mathbb{B}^{n+1}}(\zeta)$  is analytic near $\zeta = n/2 \pm i/h$ and maps \begin{eqnarray*}\dot{C}^\infty(\mathbb{B}^{n + 1}) &\longrightarrow& \rho^\zeta C^\infty(\mathbb{B}^{n + 1})\\ \rho^{\zeta + k} C^\infty(\mathbb{B}^{n + 1}) &\longrightarrow& \rho^\zeta C^\infty(\mathbb{B}^{n + 1}). \end{eqnarray*}\end{proposition}

Each fibre $F_p$ of the front face is actually a quarter-sphere of dimension $n+1$, which we denote $\mathbb{Q}$; it can be obtained from $\mathbb{B}^{n + 1}$ by blowing up any point at the boundary of $\mathbb{B}^{n + 1}$.  Mazzeo and Melrose also showed 
\begin{proposition}[\cite{Mazzeo-Melrose}] For any $j \in \ZZ$ and $k = 1, 2, 3, \dots$, the hyperbolic resolvent $R_{\mathbb{B}^{n+1}}(\zeta)$ maps 
 \begin{eqnarray*}\rho_l^{\zeta + k} \rho_r^{\zeta + j} C^\infty(\mathbb{Q}) & \longrightarrow & \rho_l^{\zeta} \rho_r^{\zeta+j} C^\infty(\mathbb{Q})\end{eqnarray*} \end{proposition}

In order to utilize these mapping properties to remove the error from the front face in Section~\ref{subsec:ff},  we show 

\begin{proposition}\label{normal operator}If one denotes $\zeta = n/2 \pm i/h$, the hyperbolic resolvent at $\zeta$ maps 
\begin{eqnarray*}
h^\infty \dot{C}^\infty(\mathbb{B}^{n + 1}) &\longrightarrow& h^\infty \rho^\zeta C^\infty(\mathbb{B}^{n + 1})\\ 
h^\infty \rho^{\zeta + k} C^\infty(\mathbb{B}^{n + 1}) &\longrightarrow& h^\infty\rho^\zeta C^\infty(\mathbb{B}^{n + 1}), \quad k = 1, 2, \dots  \\ h^\infty \rho_l^{\zeta + k} \rho_r^{\zeta + j} C^\infty(\mathbb{Q}) & \longrightarrow & h^\infty\rho_l^{\zeta} \rho_r^{\zeta+j} C^\infty(\mathbb{Q}), \quad k = 1, 2, \dots, \ j = 0, 1, 2, \dots
\end{eqnarray*} as $h \rightarrow 0$.
\end{proposition}

%\textbf{Proof.} PROOF REQUIRED HERE

\begin{proof}
We only show the first one, since the other two can be obtained from the first one by following the power series arguments in \cite{Mazzeo-Melrose} verbatim. 

We note that the first mapping property is established if for any integer $k$ we can show the hyperbolic resolvent maps 
\begin{equation}
 \dot{C}^\infty(\mathbb{B}^{n + 1}) \longrightarrow  h^{- N(k)} \rho^\zeta C^k(\mathbb{B}^{n + 1}),
 \label{hypresmap}\end{equation}
  where $N(k)$ depends on $k$. 

To this end, consider $G_\zeta(z, z')$, the Green function for the resolvent. This is a function only of the hyperbolic distance $r$ between $z$ and $z'$. It has the exact expression 
from \cite[p. 105]{Taylor2} \begin{eqnarray*}
G_\zeta  = \left\{ \begin{array}{ll} -\frac{h}{2i} \big( - \frac{1}{2\pi} \frac{1}{\sinh r} \frac{\partial}{\partial r} \big)^{n /2}e^{ir/h} & \mbox{if $n$ even}, \\ - \frac{h}{\sqrt{2} i} \int_{r}^\infty \big(- \frac{1}{2\pi} \frac{1}{\sinh s} \frac{\partial}{\partial s} \big)^{(n + 1)/2} e^{i s/h} (\cosh s - \cosh r)^{-1/2} \sinh s \, ds  & \mbox{if $n$ odd},\end{array} \right.
\end{eqnarray*}

Let us decompose this Green function into $G_\zeta^{near}(r) = G_\zeta(r) 1_{r \leq 1}$ and $G_\zeta^{far}(r) = G_\zeta(r) 1_{r \geq 1}$. To prove \eqref{hypresmap} for $G_\zeta^{near}$, it suffices to show that this is an integrable function of $r$ with $\| G_\zeta^{near}(r) \|_{L^1} \leq C h^{-N}$ for some $N$; if this is so, then acting with $G_\zeta^{near}$ in fact maps $ \dot{C}^\infty(\mathbb{B}^{n + 1})$ to $h^{-N}  \dot{C}^\infty(\mathbb{B}^{n + 1})$, which is much stronger than \eqref{hypresmap}.

In the case $n$ even, this expression of $G_\zeta$ takes the form 
$$
\sum_{j=0}^{n} h^{-j} b_j(r),
$$
so it suffices to show that each $b_k$ is locally integrable. But this is clear, since the $b_k$ are smooth for $r \in (0, 1]$, and the singularity at $r=0$ is at worst $r^{1-n}$, which is integrable. 

When $n$ is odd, we use the change of variables $\tilde{s} = \cosh s$ and $\tilde{r} = \cosh r$ to convert the expression of $G_\zeta$ to 
$$
- \frac{h}{\sqrt{2} i} \int_{\tilde r}^\infty \big(- \frac{1}{2\pi}  \frac{\partial}{\partial \tilde s} \big)^{(n + 1)/2} e^{i \cosh^{-1}(\tilde s)/h} (\cosh \tilde s - \cosh \tilde r)^{-1/2}  \, d\tilde s.
$$
This is bounded by a finite sum of the form 
$$
\sum_{k=0}^{(n-1)/2} h^{-k} \int _{\tilde{r}}^\infty \frac{1}{(\tilde{s}^2 - 1)^{n/2 - k} (\tilde{s} - \tilde{r})^{1/2}  }  \,d\tilde{s}.
$$
 One may divide this integral into two parts $$\bigg(\int_{\tilde{r}}^{2\tilde{r}}   + \int_{2\tilde{r}}^\infty\bigg) \frac{\tilde{r}^{n/2}}{(\tilde{s} + 1)^{n/2} (\tilde{s} - 1)^{n/2} (\tilde{s} - \tilde{r})^{1/2}  }  \,d\tilde{s}. $$ The first one may be crudely bounded by 
 $$
 \int_{\tilde{r}}^{2\tilde{r}}   \frac{1}{ (\tilde{r}^2 - 1)^{n/2} (\tilde{s} - \tilde{r})^{1/2}  }  \,d\tilde{s} \approx  
 \tilde{r}^{1/2 } (\tilde{r}^2 - 1)^{-n/2} \approx r^{-n}, \quad r \leq 1, 
 $$ 
 which is a locally integrable function on $\mathbb{H}^{n + 1}$. The other one is bounded by 
 $$
 \int_{2\tilde{r}}^\infty \frac{1}{(\tilde{r} + 1)^{n/2} (\tilde{s} - \tilde{r})^{(n + 1)/2} }  \,d\tilde{s} \approx  \frac{1}{(\tilde{r} + 1)^{n/2} \tilde{r}^{(n - 1)/2}  } \approx r^{-(n-1)}, \quad r \leq 1,$$ which is locally integrable on $\mathbb{H}^{n + 1}$ as well.
%
%It requires a bit of work to show (\ref{eqn: semiclassical order of green function}) as the Green's function is more complicated when $n$ is odd than when $n$ is even. Rewriting $(\tilde{s} - \tilde{r})^{-1/2}$ as a pseudodifferential operator, we will get following oscillatory integral from (\ref{eqn: worst term in green function}) $$h^{- n}\int_{\tilde{r}}^\infty \int_0^\infty e^{i(\cosh^{-1}(\tilde{s}) - (\tilde{s} - \tilde{r})\sigma)/h} \sigma^{-1/2} \, d\sigma d\tilde{s}.$$ We use stationary phase to get the semiclassical asymptotic. The phase is stationary at $\tilde{s} = \tilde{r}$ and $\sigma = (\tilde{s}^2 - 1)^{-1/2}$. So the oscillatory has expansion $$h^{-n}\bigg(h e^{i\pi/2} e^{i \ln (\tilde{r} + \sqrt{\tilde{r}^2 - 1})/h} \frac{1}{\sqrt{\tilde{r}^2 - 1}}  +  O(h^2)\bigg).$$ It is clear finite order spatial derivative acting on the expansion only creates finite number of $1/h$, which proves (\ref{eqn: semiclassical order of green function}).
%

To prove \eqref{hypresmap} for $G_\zeta^{far}$, we use instead the expression of the hyperbolic resolvent as a hypergeometric function. This takes the form 
 \begin{equation}\label{eqn: hypergeometric function}\frac{2^{-2\zeta - 1} \pi^{- n /2} \Gamma(\zeta)}{\Gamma(\zeta - n/2 + 1) (\cosh r/2)^{2\zeta}} F(\zeta, \zeta - (n-1)/2, 2\zeta - n + 1, (\cosh r/2)^{-2})
 \end{equation} 
 for $d(z, z') > 0$, where $F(\zeta, \zeta - (n-1)/2, 2\zeta - n + 1, (\cosh r/2)^{-2})$ is the Gauss hypergeometric function with expression 
 \begin{equation}\begin{gathered}
 \frac{\Gamma(2 \zeta - n + 1)}{\Gamma(\zeta - (n - 1)/2) \Gamma(\zeta - (n - 1)/2)} \int_0^1 \frac{t^{\zeta - n/2 - 1/2} (1 - t)^{\zeta - n/2 - 1/2}}{\big(1 - t (\cosh r/2)^{-2})^\zeta} \,dt \\
 =  \frac{\Gamma(1 \pm 2i/h)}{\Gamma(1/2 \pm i/h) \Gamma(1/2 \pm i/h)} \int_0^1 \frac{t^{ - 1/2 \pm i/h} (1 - t)^{- 1/2 \pm i/h}}{\big(1 - t (\cosh r/2)^{-2})^{n/2 \pm i/h}} \,dt
\end{gathered} \label{Fabc}\end{equation}
  where $\Gamma$ is the gamma function and  $\zeta = n/2 \pm i/h$.

 In the Poincar\'e ball model, if $r$ is hyperbolic distance to the origin, then 
 $$(\cosh r/2)^{-2} = \frac{1-|z|^2}{|z|^2}
 $$
 which is a boundary defining function, say $\rho$, for the ball. This expression makes it clear that $G_\zeta^{far}(r)$ has the form $\rho^{\zeta}$ times a $C^\infty$ function of $\rho$ as $\rho = (\cosh r/2)^{-2}$ tends to $0$. It remains to estimate the $C^k$ norm of this $C^\infty$ function. To do this, we differentiate \eqref{Fabc} $k$ times in $r$, and estimate. 
On one hand, applying the formulas $$|\Gamma(1/2 + iy)|^2 = \frac{\pi}{\cosh (\pi y)}  \quad \mbox{and} \quad |\Gamma(1 + iy)|^2 = \frac{\pi y}{\sinh (\pi y)},$$ we gain, for both even and odd $n$, \begin{equation}\label{eqn: coefficient of hypergeometric function}\frac{\Gamma(n/2 \pm i/h)}{\Gamma(1 \pm i/h)} \frac{\Gamma(1 \pm 2i/h)}{\Gamma(1/2 \pm i/h) \Gamma(1/2 \pm i/h)} \leq C h^{-n/2 + 1/2} \quad \mbox{as $h \rightarrow 0$.}\end{equation}
On the other hand, we have to estimate the integral $$ \big( \frac{\partial}{\partial r} \big)^k \int_0^1 \frac{t^{-1/2 \pm i/h} (1 - t)^{-1/2 \pm i/h}}{\big(1 - t(\cosh r/2)^{-2})^{n/2 \pm i/h}} \,dt.$$  
Notice that for $G_\zeta^{far}$, we always have $r \geq 1$, therefore $(\cosh r/2)^{-2}$ is always less than and bounded away from $1$. Therefore the $k$th $r$ derivative of the integrand is absolutely integrable for all $k$, so we may differentiate under the integral sign. We see that the integral is bounded by $C_k h^{-k}$ uniformly in $h$ (where the negative powers of $h$ arise from the exponent in the denominator). This establishes \eqref{hypresmap}, and hence completes the proof.  
\end{proof}

\subsection{Geodesics}\label{subsec:geo}
Let $p$ be the symbol of $P_h$ in \eqref{semiclassical eqn}.
We consider the structure of (null) bicharacteristics --- that is, integral curves of the Hamilton vector field $H_p$ inside the characteristic variety $p = 0$ --- on the single $0$-cotangent space ${}^0 T^* X$. The Hamilton vector field for any Hamiltonian $p$ (a smooth real-valued function on ${}^0 T^* X$) is
$$x \frac{\partial p}{\partial \lambda} \frac{\partial}{\partial x} + x \frac{\partial p}{\partial \mu} \cdot \frac{\partial}{\partial y} - \bigg(\mu \cdot \frac{\partial p}{\partial \mu} + x \frac{\partial p}{\partial x}\bigg) \frac{\partial}{\partial \lambda} + \bigg(\frac{\partial p}{\partial \lambda}\mu - x \frac{\partial p}{\partial y}\bigg) \cdot \frac{\partial}{\partial \mu}.$$
Consider the Hamiltonian $p$ near the boundary of ${}^0 T^* X$. Assume that we have coordinates $(x, y)$ such that the metric $g$ takes the form \eqref{g-normalform}.
Then $p$ takes the form  $p = \lambda^2 + g_0^{ij}(x, y) \mu_i \mu_j - 1$. We obtain
\begin{equation}\label{eqn: hamilton flow-out}
\left\{ \begin{array}{ccl}
\dot{x} &=& 2x \lambda  \\
\dot{y_i} & = & 2x g_0^{ij} \mu_j  \\
\dot{\lambda} & = & -\Big(2g_0^{ij} +  x \partial_x g_0^{ij}  \Big) \mu_i \mu_j    \\
\dot{\mu_i} & = & \Big( 2\lambda\mu_i - (x \partial_{y_i}  g_0^{jk}) \mu_j \mu_k \Big) \end{array} \right..
\end{equation}
Of course, this is just geodesic flow (viewed in the cotangent bundle) written in these coordinates.
Let $\gamma$ be a bicharacteristic (that is, geodesic) over the interior of $X$. We claim that $\gamma$ extends smoothly to a compact curve in ${}^0 T^* X$. To see this, we first notice that $\lambda \to \pm1$ and $\mu \to 0$ as $x \to 0$ along $\gamma$. In fact, as $x \rightarrow 0$ along the flow, $\dot{x} < 0$, then $\partial p / \partial \lambda = 2 \lambda < 0$; in the mean time, $\dot{\lambda} \leq 0$ when $x$ is sufficiently small. On the other hand, the energy condition $\lambda^2 + |\mu|_{g_0}^2 = 1$ gives
$$
(g_0^{ij} \mu_i \mu_j)^\cdot = 2 \lambda (g_0^{ij} \mu_i \mu_j)     + 2 \lambda  x \frac{\partial g_0^{ij}}{\partial x} \mu_i \mu_j = 4 \lambda (g_0^{ij} \mu_i \mu_j)\big(1 + O(x) \big),
$$
the right hand side of which is bounded above by $- C g_0^{ij} \mu_i \mu_j$ for some $C > 0$ when $x$ is small. A simple application of Gronwall's inequality shows $g_0^{ij} \mu_i \mu_j \rightarrow 0$, namely, $\lambda \rightarrow - 1$, which completes the proof. Similarly, $\lambda \to +1$ as $x \to 0$ backwards along $\gamma$. In fact, we have $g_0^{ij} \mu_i \mu_j / x^{2\delta} \to 0$ for any $\delta < 1$; this follows by computing that
$$
\Big( \frac{g_0^{ij} \mu_i \mu_j}{x^{2\delta}} \Big)^\cdot < 0
$$
when $\lambda$ is sufficiently close to $-1$ (depending on $\delta$), and $x$ is sufficiently small. Using once again the energy condition $\lambda^2 + |\mu|_{g_0}^2 = 1$, we see that $\lambda - 1 = O(x^{2\delta})$ as $x \to 0$.

We now `shift' the bicharacteristic\footnote{See Section \ref{ssec:shift} for more information on this.} so that it meets the boundary at $\lambda = 0$ rather than $\lambda = -1$ (in the forward direction). To do this, we apply the symplectic transformation $q \mapsto q + dx/x = q \mapsto q + d\log x$ in the $0$-cotangent bundle (which is just $\lambda \mapsto \lambda + 1$ in these coordinates), and then introduce the coordinates $\xi = \lambda/x$ and $\eta = \mu/x$, which just amounts to going back to the standard cotangent bundle rather than the $0$-cotangent bundle (since then $\xi, \eta$ are the dual coordinates to $(x, y)$ --- see Section~\ref{subsec:0bundle}).  Combining these two operations means that we substitute $\lambda = -1 + x\xi$ and $\mu = x \eta$ in the Hamiltonian. We obtain the new Hamiltonian
$$
\tilde p = (x \xi)^2 - 2 x \xi + x^2 g_0^{ij} \eta_i \eta_j,
$$
which we note has an overall factor of $x$. Consider the Hamilton vector field for $\tilde p/x$:
\begin{equation}
\left\{ \begin{array}{ccl}
\dot{x} &=& 2 x \xi - 2  \\
\dot{y_i} & = &  x g_0^{ij} \eta_j  \\
\dot{\xi} & = & - \xi^2 - \Big(2g_0^{ij} +x \partial_x g_0^{ij}  \Big) \eta_i \eta_j     \\
\dot{\eta_i} & = &  - x \partial_{y_i}  g_0^{jk} \eta_j \eta_k  \end{array} \right..
\label{shvf}\end{equation}
On the set $\{ \tilde p = 0 \} = \{ \tilde p/x = 0 \}$, this is just the Hamilton vector field for $\tilde p$ divided by $x$, since $H_{\tilde p/x} = x^{-1} H_{\tilde p} + \tilde p H_{1/x}$. Moreover, since the map $\lambda \mapsto \lambda + 1$ is symplectic, this is the pushforward of the Hamilton vector field of $p$ under this map. Therefore, the integral curve $\tgamma$ of this flow is the image of $\gamma$ under a symplectic transformation.  The flow of this vector field is smooth down to $x=0$. To see this it is enough to check that
the flow line reaches $x=0$. Note that the nontrapping condition implies that $x$ becomes arbitrarily small along $\gamma$. Then, since we observed above that $\lambda + 1 = x\xi = O(x^{2\delta})$, we have, using \eqref{shvf}, $\dot x \to -2$ as $x \to 0$. So it is enough to check that the RHS of \eqref{shvf} stays bounded as $x \to 0$. This follows since we have $\lambda \geq -1$ which implies $\xi \geq 0$,  and $\dot \xi \leq 0$ for small $x$. So clearly $\xi$ remains bounded. As for $\eta$, we have $x \eta_j \eta_k= x^{-1} \mu_j \mu_k = O(x^{2\delta - 1})$ as $x \to 0$, so this also remains bounded. It follows that $\tgamma$ is smooth in the standard cotangent bundle. As the inverse map $\lambda = -1 + x\xi$, $\mu = x \eta$ is smooth, we see that also $\gamma$ is smooth in ${}^0 T^* X$.

From now on, we will take $\gamma$ or $\tgamma$ to be the closure of the actual integral curve, that is, including the initial and final endpoints at $x = 0$.
One advantage of considering $\tgamma$ instead of $\gamma$ is that the $\tgamma$ are all disjoint (considered as subsets of $T^*X$), while the $\gamma$ are not (considered as subsets of ${}^0 T^* X$). In fact, all the bicharacteristics with a fixed initial direction $y_{-\infty}$ or final direction $y_\infty$ meet at their initial or final endpoints, since we have $x=0, \lambda = \pm 1$, $\mu = 0$ there. On the other hand, the endpoints of shifted bicharacteristics are all different, as follows from the nonvanishing of the vector field \eqref{shvf} at $x=0$.

\subsection{Leaves}\label{subsec:leaves}

We now consider the product of a bicharacteristic $\gamma$ with itself in the double cotangent space. The (forward and backward) bicharacteristic relation is  foliated by these leaves $\gamma$ emanating from the diagonal conormal bundle $N^* \diag$. Understanding these leaves gives us a good microlocal view of $\FBR$.

Initially we work with the product $({}^0 T^* X)^2$. We first introduce some notation.
Let $\gamma^2$ denote $\gamma \times \gamma$, and let
$\gamma^{2, \prime}$ denote the same space with the second fibre coordinate negated:
\begin{equation}
\gamma^{2,\prime} = \{ (q, q') \mid (q, -q') \in \gamma \times \gamma \}.
\end{equation}
We denote by $\gammatf$ half of this space corresponding to forward propagation. Let $r : [0, \pi]$ be a parametrization of $\gamma$, so that $\dot r > 0$ under forward propagation.
\begin{equation}
\gammatf = \{ (q, q') \in \gamma^{2,\prime} \mid q = \gamma(r), \ -q' = \gamma(r'), \ r \geq r' . \}.
\end{equation}
Finally, $\gammatb$ is $\gamma^{2,\prime}$ blown up at the diagonal corners. Let $\partial \gamma = \{ \gamma(0), \gamma(\pi) \}$. We define
\begin{equation}
\gammatb = \big[ \gamma^{2,\prime}; \{ (\gamma(0), -\gamma(0)) \} \cup \{ (\gamma(\pi), -\gamma(\pi)) \} \big].
\end{equation}
and similarly,
\begin{equation}
\gammatfb = \big[ \gammatf; \{ (\gamma(0), -\gamma(0)) \} \cup \{ (\gamma(\pi), -\gamma(\pi))  \} \big].
\end{equation}

Clearly, $\gamma^{2,\prime}$ is a smooth p-submanifold  of $({}^0 T^* X)^2$. Next consider the structure of $\gamma^{2,\prime}$ as a subset of ${}^\Phi T^* X^2_0$, which is obtained from $({}^0 T^* X)^2$ by blowing up $\{ x = x' = 0, y = y' \}$. Notice that $\gamma^{2, \prime}$ meets this set at the diagonal endpoints, corresponding to $r=r'=0$ and $r=r'=\pi$, provided that $y(0) \neq y(\pi)$. 
Then the two boundary hypersurfaces $x=0$ and $x' = 0$, the set $\{ x = x' = 0, y = y' \}$ and $\gamma^{2,\prime}$ intersect cleanly, in the sense that near any point of $({}^0 T^* X)^2$ one can find local coordinates such that each of these submanifolds is given by the vanishing of a subset of such coordinates. It follows that the lift of $\gamma^{2,\prime}$ to the blowup ${}^\Phi T^* X^2_0$ is a p-submanifold of ${}^\Phi T^* X^2_0$ naturally diffeomorphic to $\gammatb$ (see \cite[Proposition 5.7.2]{DiffAnal}). Moreover, after this blowup, then the forward half $\gammatfb$ of this is also a p-submanifold. We refer to $\gammatfb$ as a (forward) leaf. See Figure \ref{fig: leaves}.

\begin{remark}The argument above isn't valid at the anti-diagonal corners in Figure \ref{fig: leaves} in the case that $y(0) = y(\pi)$, since then the antidiagonal corners would also need to be blown up. This causes some inconvenience in the argument in Section~\ref{sec:flowout}. 
We deal with this issue in Section~\ref{subsec:Lambdastar}. 
\end{remark}

\begin{center}\begin{figure}
 \includegraphics[width=0.7\textwidth]{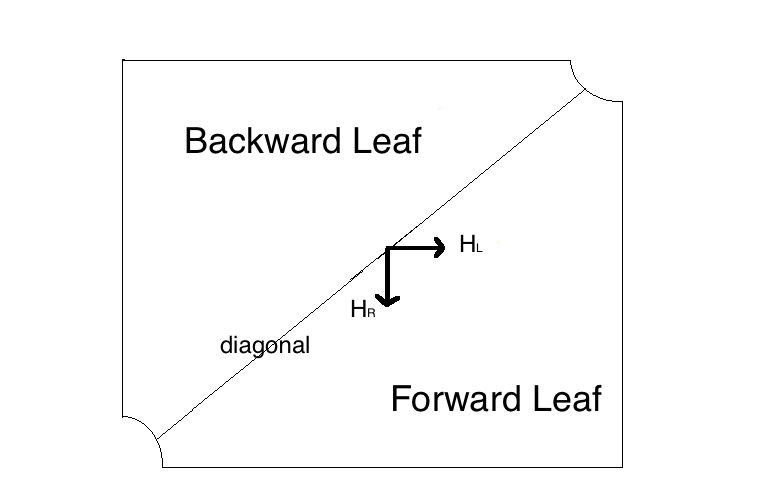}\caption{\label{fig: leaves}A leaf $\gamma_b^{2, \prime}$: the part below the diagonal is $\gammatfb$.}\end{figure} \end{center}
 
Let $\tgamma$ be a shifted bicharacteristic. We define the corresponding leaf $\tgamma^{2, \prime}$ to be the subset of $T^* X \times T^* X$ given by
\begin{equation}
\tgamma^{2,\prime} = \{ (q, q') \mid (q, -q') \in \tgamma \times \tgamma \}.
\end{equation}
Since the $\tgamma$ are all disjoint, so are the corresponding leaves $\tgamma^{2, \prime}$. We have the following

\begin{lemma}\label{lem:int-shifted-leaves}
Let $\epsilon > 0$, and let $\tilde\Lambda$ be the subset of $T^* X \times T^* X$ given by
$$
\tilde\Lambda = \bigcup_{\tgamma} \  \tgamma^{2, \prime}
$$
where the union is over all $\tgamma$ that intersect the set $\{ x > \epsilon \}$. Then $\tilde\Lambda$ is a (p-)submanifold of $T^* X \times T^* X$ which is transverse to each boundary hypersurface.
\end{lemma}

\begin{proof} Locally in $T^* X$, the shifted bicharacteristics foliate the set $\{ \tilde p = 0 \}$, where $\tilde p$ is the shifted symbol of $P_h$. It follows that we can choose coordinates $(\tilde p, t, w)$ where $w$ is constant along shifted bicharacteristics, and $t$ is a coordinate along each bicharacteristic. Near the boundary, we can take $t=x$.  Then in $T^* X \times T^* X$, we have coordinates $(\tilde p, t, w; \tilde p', t', w')$ where the primed/unprimed coordinates are lifted from the right/left factor of $T^* X$. The subset in the lemma is given in these coordinates by
$$
\{ \tilde p = \tilde p' = 0, w = w' \},
$$
and is clearly a p-submanifold.
\end{proof}

\begin{center}\begin{figure}
\includegraphics[width=0.8\textwidth]{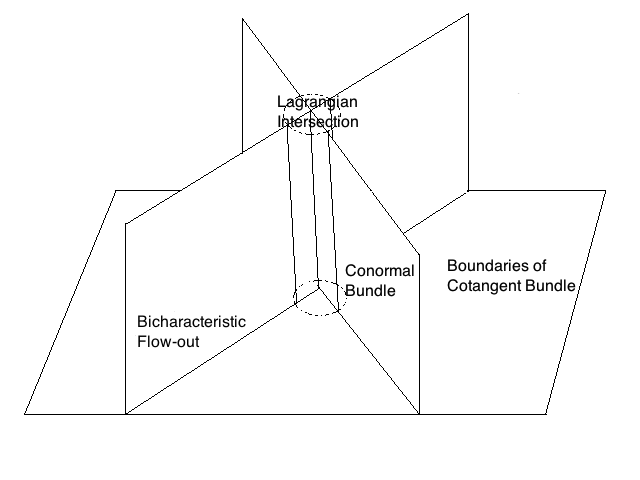}\caption{\label{fig: cotangent bundle decomposition}The Lagrangian intersection  $N^* \diagz \cap \FBR$}\end{figure}\end{center}

\subsection{Outline of proof}\label{subsec:outline}
We wish to invert the  operator $P_h$ given by \eqref{semiclassical eqn}, which is an operator of semiclassical order $0$ and differential order $2$. To do this we first construct a parametrix $G$, that is, an operator such that the error term $E_h := P_h G_h - \Id$ is very mild. 
We will successively solve away the differential singularity as $\zeta \rightarrow \infty$ (so that the Schwartz kernel of the error term $E_h$ is smooth), the semiclassical singularity as $h \rightarrow 0$ (so that $E_h$ is $O(h^\infty)$) and the boundary singularity as $x \to 0$ (so that $E_h$ is $O(x^\infty)$). 

We view $P_h G_h = \Id$ as a PDE on $X^2_0$; thus, $P_h$ is viewed as a differential operator $P_L$ on $X^2_0$, acting in the left variables. We denote  the characteristic variety of $P_L$ by $\Sigma_L$, and the conormal bundle of $\diagz \subset X^2_0$ by $N^* \diagz$. 
Select a fibre-compact neighbourhood of the Lagrangian intersection $N^* \diagz \cap \Sigma$ up to boundary, say $\Gamma$, such that $|\zeta|^2  > 1$ in $N^* \diagz \setminus \Gamma$. The set $\Gamma$ will meet the front face, but will be disjoint from the other boundary hypersurfaces. The operator $P_L$ is thus elliptic on $N^* \diagz \setminus \Gamma$. The theory of $0$-pseudodifferential operator is applicable to solve away the singularity on $N^* \diagz \setminus \Gamma$.

On the other hand, the singularities in $N^* \diagz \cap \Sigma$ propagate according to the Hamilton vector field of $P_L$. The forward integral curves of this Hamilton vector field, starting at $N^* \diagz \cap \Sigma$, sweeps out the Lagrangian submanifold $\FBR$, by definition of $\FBR$ (see \eqref{Lambda1def}). 
This motivates using an intersecting Lagrangian ansatz for $G_h$ associated to the Lagrangian submanifolds $(N^* \diagz, \FBR)$, following Melrose-Uhlmann \cite{Melrose-Uhlmann-CPAM-1979}, to solve these errors away. We also need good control on the closure of $\FBR$, i.e. the way it meets the boundary of ${}^\Phi T^* X^2_0$, so that the singularities of $G_h$ are controlled uniformly to infinity. 
 The geometric structure of the closure $\Lambda_+$ of $\FBR$ is somewhat intricate, and describing this is the subject of Section~\ref{sec:flowout}.

Thus the resolvent construction, carried out in Section~\ref{sec:parametrix}, proceeds in several stages. We solve away first the elliptic singularities, then the singularity at the Lagrangian intersection, then the singularities globally over $\Lambda_+$, then at $x=0$ (both at the front face, using the normal operator, and at $\FL$), and finally end up with a very benign error term $E_h$, which can be inverted with precise control over the structure of its inverse. 
%In summary, we shall invert the Laplacian on $(N^* \diagz \cap \Gamma) \cup (\Lambda_1 \cap \Gamma)$, $N^* \diagz \setminus \Gamma$, $\Lambda_1 \setminus \Gamma$ respectively.

%%%%%%%%%%%%%%%%%%%%%%%%%%%%%%%%
%%%%%%%%%%%%%%%%%%%%%%%%%%%%%%%%
%%%%%%%%%%%%%%%%%%%%%%%%%%%%%%%%

\section{Geometry of the forward bicharacteristic relation}\label{sec:flowout}

The key to our parametrix construction is understanding the geometry of the Lagrangian $\FBR$, defined in \eqref{Lambda1def}, as well as the Hamilton vector field which is tangent to it, as we approach the boundary of $\Phi$-cotangent bundle ${}^\Phi T^* X^2_0$, which we recall is the bundle obtained by pulling back the product bundle ${}^0 T^* X \times {}^0 T^* X$ to $X^2_0$ by the blowdown map $\beta : X^2_0 \to X^2$. 

Before proceeding we introduce some notation. The operator $P_L$ is the differential operator on $X^2_0$ (depending on the parameter $h$) given by $P = P_h$ acting in the left variables, and $\Sigma_L$ is the zero set of the principal symbol $p_L$ of $P_h$. 
When the operator acts in the right variables, this will be denoted $P_R$, with principal symbol $p_R$ and zero set $\Sigma_R$.
We define $H^L$ and $H^R$ to be the Hamilton vector field of $p_L$, resp. $p_R$.

%We recall that $\Lambda_1$ is given by the closure, in a certain sense, of the flowout by the Hamilton vector field of the symbol of the operator $P_L$ from $N^* \diagz \cap \Sigma_L$, where $P_L$ is the differential operator on $X^2_0$ given by $P = P_h$ acting in the left variables, and $\Sigma_L$ is the zero set of the principal symbol $p_L$ of $P_h$. 
%%The subscript $L$ is because we take $P_h$ to act `on the left' of our parametrix, that is, in the left (unprimed) variables on $X^2_0$.
%(When the operator acts in the right variables, this will be denoted $P_R$, with principal symbol $p_R$ and zero set $\Sigma_R$.) 
%
%More precisely, $\Lambda_1$ is given, in a suitable sense, as the closure of the forward bicharacteristic relation $\FBR$, which is the 
%subset of $T^* X^\circ \times T^* X^\circ$ given by
%\begin{equation}
%\FBR = \{ (q, q') \in T^* X^\circ \times T^* X^\circ \mid q = \exp_{r H_p} (-q'), \text{ for some } r \geq 0 \}. 
%\label{FBR}\end{equation}

\subsection{Shifting the Lagrangian}\label{ssec:shift}

We start by recalling how the question of regularity of the forward bicharacteristic relation was tackled by Melrose, S\'{a} Barreto and Vasy \cite{Melrose-Sa Barreto-Vasy}, in the case that $X$ is a small perturbation of the standard metric on hyperbolic space. Their essential idea is to pass from the $\Phi$-cotangent bundle, ${}^\Phi T^* X^2_0$, to the $b$-cotangent bundle. 
They showed that the $b$-cotangent bundle of $X^2_0$ could be obtained by blowing up three submanifolds in ${}^\Phi T^* X^2_0$. Viewing the left Hamilton vector field $H^L$ as living in the $b$-cotangent bundle (which is certainly valid over the interior of $X^2_0$), they show that it is of the form $H^L = \rho_L V^L$, where $\rho_L$ is a boundary defining function for $\FL$, and $V_L$ is a smooth vector field on ${}^b T^* X^2_0$ that is transverse to $\FL$. As a consequence, $\Lambda_+$, the closure of $\FBR$, in the $b$-cotangent bundle, is smooth up to $\FL$, and meets it transversally. Exactly the same analysis applied to the right Hamilton vector field shows smoothness up to the right boundary $\FR$. 

We shall perform a very similar analysis, but with a slight twist. Instead of passing to the $b$-cotangent bundle, we shall pass to the usual cotangent bundle $T^* X^2_0$. One problem with doing this is that the closure of $\Lambda_+$ cannot possibly be smooth viewed as a subset of $T^* X^2_0$. Indeed, in the case of a small perturbation of hyperbolic space, Melrose, S\'{a} Barreto and Vasy showed that $\Lambda_+$ is the graph of the differential of the distance function on $X^2_0$, away from the diagonal, which takes the form $-d\rho_L/\rho_L - d\rho_R/\rho_R$ plus the differential of a smooth function. This is smooth on the $b$-cotangent space, but not on the usual cotangent space. 

On the other hand, this result suggests the following strategy. We `shift' the Lagrangian $\Lambda_+$ by modifying it precisely so as to remove the explicit divergence noted above. That is, we define a shifted relation $\tFBR$, which we view as living in the standard cotangent bundle $T^* X^2_0$ of $X^2_0$, as follows:
\begin{equation}
\tFBR = \{ q \in T^* X^2_0 \mid q = q' + \frac{d\rho_L}{\rho_L} + \frac{d\rho_R}{\rho_R}, \ q' \in \FBR \}. 
\label{tL1}\end{equation}
Based on the results of Melrose, S\'{a} Barreto and Vasy, we can hope that the closure of $\tFBR$ is smooth, as a subset of the standard cotangent bundle $T^* X^2_0$. That is essentially what we shall show, although there is a subtlety involving the possibility of geodesics in the interior of $X$ such that their initial and final directions coincide\footnote{This possibility cannot occur in the geometric setting considered by Melrose, S\'{a} Barreto and Vasy}. We deal with this by decomposing $\tFBR$ into two parts and considering each separately. The second part, $\tFBR^*$, we view as living in $T^* X^2$, rather than $T^* X^2_0$ (see Prop~\ref{prop:flowout-structure} and Remark~\ref{rem:comparison} below). 

\begin{definition} The set $\Lambda_+ \subset {}^\Phi T^* X^2_0$ is the closure of $\FBR$ in ${}^\Phi T^* X^2_0$. We similarly define $\tL_+$ to be the closure of the $\tFBR$ in $T^* X^2_0$. 
\end{definition}

Observe that \eqref{tL1} can be written $\tFBR = T^{-1}(\FBR)$, where the transformation $T: q' \mapsto q' - df$, $f =  \log (\rho_L \rho_R)$,  is a symplectic transformation (this is true for any smooth $f$). It follows that $\tFBR$ is another Lagrangian submanifold, and indeed it is given by the flowout of a shifted Hamilton vector field $\tHL$, that of the shifted Hamiltonian $\tpL$ which is the pullback of the Hamiltonian $p_L$ by the map $T$. 

\begin{remark} In practice, to simplify our calculations, we shall choose a different boundary defining function $\rho_L$ or $\rho_R$ in each local coordinate patch. In spite of this, we shall continue to refer to `the' shifted Lagrangian, and denote it by a single symbol $\tL_+$. On the overlaps of different coordinate patches, we will  therefore get different shifted Lagrangians, but they will be related by a smooth transformation on the standard cotangent bundle, so this will not affect the validity of our results.  
%It is also convenient to assume that $\rho_L$ and $\rho_R$ are equal to $1$ in a neighbourhood of $
\end{remark}

The key to understanding the regularity properties of $\tL_+$ is the following regularity statement for the shifted Hamilton vector fields on $T^* X^2_0$:

\begin{lemma}\label{lem:shifted-vf} The left and right shifted Hamilton vector fields $\tHL$, $\tHR$, restricted to the shifted characteristic variety, lift to smooth vector fields on $T^* X^2_0$ tangent to $\partial_{\FF}T^* X^2_0$, such that $\tHL/\rho_L$ is tangent to  $\partial_{\FR}T^* X^2_0$ and transverse to  $\partial_{\FL}T^* X^2_0$, while $\tHR/\rho_R$ is tangent to  $\partial_{\FL}T^* X^2_0$ and transverse to  $\partial_{\FR}T^* X^2_0$.
\end{lemma}

\begin{proof} To prove this, we need to check the structure of the left and right Hamilton vector fields in the cotangent bundle over  the various regions of $X^2_0$ listed in Section~\ref{subsec:blowup}.
By symmetry, it is only necessary to prove the statements for the left Hamilton vector field. We consider each of the regions above in turn.

$\bullet \ $ In region 1, there is nothing to prove other than smoothness of the Hamilton vector field, which is clear.

$\bullet \ $ In region 2a, we use local coordinates $(x, y, z'; \lambda, \eta, \zeta')$. In terms of these, the left Hamiltonian is
$$
\lambda^2 + g_0^{ij} \mu_i \mu_j - 1,
$$
and the shift is implemented by pulling back $\Lambda_+$ by the map $\lambda \mapsto \lambda - 1$.  Therefore, the shifted Hamiltonian is
\begin{equation}
\tpL = (\lambda-1)^2 +  g_0^{ij} \mu_i \mu_j - 1 = \lambda^2 - 2\lambda + g_0^{ij} \mu_i \mu_j .
\label{shifted-Ham-r1}\end{equation}
Fibre coordinates $(\lambda, \mu)$ on ${}^0 T^* X$ are related to coordinates $(\xi, \eta)$ on $T^* X$ by
$$
\lambda \frac{dx}{x} + \sum_i \mu_i \frac{dy^i}{x} = \xi dx + \sum_i \eta_i dy^i,
$$
which implies that $\lambda = x \xi$ and $\mu_i = x \eta_i$. Therefore, the shifted Hamiltonian $\tpL$,  viewed on the standard cotangent bundle, is (as in Section~\ref{subsec:geo})
$$
(x \xi)^2 - 2 x \xi + x^2 g_0^{ij} \eta_i \eta_j.
$$
We want to compute the shifted Hamilton vector field,  divided by $x$. On the zero set of the symbol, this is the same as the Hamilton vector field of $\tpL/x$ (using $H_{\tpL/x} = x^{-1}  H_{\tpL} + \tpL H_{1/x}$). Since
$$
\frac{\tpL}{x} = x \xi^2 - 2  \xi + x g_0^{ij}(x,y) \eta_i \eta_j,
$$
We find that, on $\{ \tpL = 0 \}$,
$$\begin{gathered}
x^{-1} \tHL = \big( - 2 + 2x\xi \big) \frac{\partial}{\partial x} - \big(2\xi^2 +2 g_0^{ij} \eta_i \eta_j + x \frac{\partial g_0^{ij}}{\partial x}\eta_i\eta_j \big)   \frac{\partial}{\partial \xi} \\ + 2x g_0^{ij}\eta_j \frac{\partial}{\partial y_i}  - x\frac{\partial g_0^{ij}}{\partial y_k}\eta_i\eta_j\frac{\partial}{\partial \eta_k}\end{gathered},
$$
which is a smooth vector field transverse to $\FL$.

$\bullet \ $ In region 2b, there is almost nothing to prove. In this region, we can assume that $\rho_L = 1$ is constant, and the left Hamiltonian is independent of $\lambda'$, so the shift has no effect on the Hamilton vector field which is independent of the primed variables. The Hamilton vector field is clearly smooth and tangent to $\FR$.

$\bullet \ $ In region 3, the calculation is essentially the same as in 2a.

$\bullet \ $ In region 4a, we use coordinates
$$s = \frac{x}{x^\prime}, x^\prime, y, Y = \frac{y^\prime - y}{x^\prime}.
$$
To relate the fibre coordinates, we equate
$$
\lambda \frac{dx}{x} + \lambda' \frac{dx'}{x'} + \sum_i \mu_i \frac{dy^i}{x} + \sum_i \mu'_i \frac{{dy'}^i}{x'}=
\sigma ds + \xi' dx' + \sum_i \eta_i dy^i + \sum_i N_i dY^i,
$$
to obtain
$$
\lambda = s \sigma, \quad \mu = s x'\eta - sN .
$$
This shows that the left Hamiltonian is
\begin{equation}
p_L = (s\sigma)^2 + s^2 g_0^{ij}(x's, y) (x' \eta -  N)_i (x' \eta -  N)_j  - 1.
\label{lH-ff}\end{equation}
The shift transformation $T$ here is $\sigma \mapsto \sigma - 1/s$. So the shifted left Hamiltonian is
\begin{equation}
\tpL = (s \sigma)^2 - 2 s \sigma + s^2 g_0^{ij}(x's, y) (x' \eta - N)_i (x' \eta - N)_j
\label{shifted-Ham-r4a} \end{equation}
and
\begin{equation}
\frac{\tpL}{s} = s \sigma^2 - 2  \sigma + s g_0^{ij}(x's, y) (x' \eta - N)_i (x' \eta - N)_j.
\label{shifted-Ham-4aa} \end{equation}
As above, the Hamilton vector field of $\tpL/s$ is equal to $s^{-1}$ times the Hamilton vector field of $\tpL$, on the zero set of $\tpL$.
Therefore, on the zero set of $\tpL$ we have
\begin{equation}\begin{gathered}
\frac{\tHL}{s}  = 2(s\sigma - 1) \frac{\partial}{\partial s}\\ - \Big(   \sigma^2  + g_0^{ij}(x' \eta - N)_i (x' \eta - N)_j+ x' s \frac{\partial g_0^{ij}}{\partial x}(x' \eta - N)_i (x' \eta - N)_j\Big) \frac{\partial}{\partial \sigma} \\ + 2s g_0^{ij}(x' \eta - N)_j \frac{\partial}{\partial Y_i} +  2x's  g_0^{ij} (x' \eta_j - N_j) \frac{\partial}{\partial y_i} \\ - s \frac{\partial g_0^{ij}}{\partial y_k}(x' \eta - N)_i (x' \eta - N)_j\frac{\partial}{\partial \eta_k}.
\end{gathered}\end{equation}
Therefore, on the characteristic variety we have
$$
\frac{\tHL}{s} = - 2\partial_s + \text{smooth vector field on cotangent bundle tangent to $\FL$ and $\FF$.}$$

$\bullet \ $ In region 4b, we use coordinates
$$
x, \quad s' = \frac{x'}{x}, \quad Y' = \frac{y-y'}{x}, \quad y'
$$
with dual coordinates
$$
\xi, \sigma', N', \eta'
$$
on the fibres of the standard cotangent bundle.
We relate these coordinates with $(\lambda, \mu, \lambda', \mu')$ by equating
$$
\lambda \frac{dx}{x} + \lambda' \frac{dx'}{x'} + \sum_i \mu_i \frac{dy^i}{x} + \sum_i \mu'_i \frac{{dy'}^i}{x'}=
\xi dx + \sigma' ds'  + \sum_i N'_i d{Y'}^i+ \sum_i \eta'_i {dy'}^i .
$$
This shows that
$$
\lambda = \xi x - \sigma' s' - N' \cdot Y', \quad \mu = N' .
$$
Since we are away from $\FL$, we can assume that $\rho_L= 1$,
so we only need to pull back by the shift $\lambda' \mapsto \lambda' - 1$; this does not affect the left Hamiltonian which is independent of $\lambda'$. So the shifted Hamiltonian in this case is
$$
(\xi x - \sigma' s' - N' \cdot Y')^2 + g_0^{ij} N'_i N'_j - 1.
$$
The Hamilton vector field in this case is smooth and is such that $\dot x = O(x)$ and $\dot s' = O(s')$, so it is a smooth vector field tangent to $\FF$ and $\FR$.

$\bullet \ $ In region 5, we use coordinates
$$
s_1 = \frac{x}{y_1^\prime - y_1}, \quad s_2 = \frac{x^\prime}{y_1^\prime - y_1}, \quad t = y_1^\prime - y_1,\quad Z_j = \frac{y_j^\prime - y_j}{y_1^\prime - y_1} \, (j > 1), \quad y
$$
near the corner of front face,  where we assume (by permuting the $y_i$ coordinates as necessary) that  $y_1^\prime - y_1$ dominant, that is, that $|y_i - y_i'| \geq c |y- y'|$ locally for some $c > 0$. We use dual coordinates
$$
\sigma_1, \sigma_2, \tau, \zeta_j, \eta
$$
on the fibres of the cotangent bundle. We can relate these coordinates with $(\lambda, \mu, \lambda', \mu')$ by equating
$$
\lambda \frac{dx}{x} + \lambda' \frac{dx'}{x'} + \sum_i \mu_i \frac{dy^i}{x} + \sum_i \mu'_i \frac{{dy'}^i}{x'}=
\sigma_1 ds_1  + \sigma_2 ds_2  + \tau dt +  \sum_i \zeta_i dZ_i + \sum_i \eta_i {dy}^i .
$$
This shows that
$$
\lambda = \sigma_1 s_1, \quad \mu_1 = s_1 \big( s_1 \sigma_1  + s_2 \sigma_2  - t \tau +  \zeta \cdot Z +  t \eta_1 \big) , \quad
\mu_j = - s_1 \zeta_j .
$$
So the left Hamiltonian in these coordinates is
\begin{equation}\begin{gathered}
(s_1 \sigma_1 )^2 + s_1^2 g_0^{11} \big(s_1 \sigma_1  + s_2 \sigma_2  - t \tau +  \zeta \cdot Z +  t \eta_1\big)^2 \\ + s_1^2 \sum_{j \geq 2} g_0^{1j} \big( s_1 \sigma_1  + s_2 \sigma_2  - t \tau +  \zeta \cdot Z +  t \eta_1 \big)  \zeta_j + s_1^2 \sum_{i, j \geq 2} g_0^{jk} \zeta_i \zeta_j - 1.
\end{gathered}\end{equation}
The shift in these coordinates is pullback by $\sigma_1 \mapsto \sigma_1 - 1/s_1$,  $\sigma_2 \mapsto \sigma_2 - 1/s_2$. So the shifted left Hamiltonian is
\begin{equation}\begin{gathered}
\tpL = (\sigma_1 s_1)^2 - 2 \sigma_1 s_1 + s_1^2 g_0^{11} \big(s_1 \sigma_1  + s_2 \sigma_2 -2 - t \tau +  \zeta \cdot Z +  t \eta_1
\big)^2 \\ + s_1^2 \sum_{j \geq 2} g_0^{1j} \big(s_1 \sigma_1  + s_2 \sigma_2 -2 - t \tau +  \zeta \cdot Z +  t \eta_1 \big)  \zeta_j + s_1^2 \sum_{i, j \geq 2} g_0^{jk} \zeta_i \zeta_j.
\end{gathered}\label{shifted-Ham-r5} \end{equation}
Therefore
\begin{equation}\begin{gathered}
\frac{\tpL}{s_1}  = s_1\sigma_1^2 - 2 \sigma_1  + s_1 g_0^{11} \big(s_1 \sigma_1  + s_2 \sigma_2 -2 - t \tau +  \zeta \cdot Z +  t \eta_1
\big)^2 \\ + s_1 \sum_{j \geq 2} g_0^{1j} \big(s_1 \sigma_1  + s_2 \sigma_2 -2 - t \tau +  \zeta \cdot Z +  t \eta_1 \big)  \zeta_j + s_1 \sum_{i, j \geq 2} g_0^{jk} \zeta_i \zeta_j.
\end{gathered}\label{shifted-Ham-r55} \end{equation}
As above, we can compute the shifted left Hamilton vector field, divided by $s_1$ on $\{ \tpL = 0 \}$, by the Hamilton vector field of $\tpL/s_1$. A straightforward calculation shows that this has the form
\begin{equation}\begin{gathered}
\frac{\tHL}{s_1} = 2(s_1 \sigma_1 -1)\frac{\partial}{\partial s_1}
+ O(s_1 t) \frac{\partial}{\partial t} + O(s_1 s_2) \frac{\partial}{\partial \sigma_2} \\
+ \text{ $C^\infty$-linear combination of }
% - \big( 2\frac{\sigma_1}{s_1} + a_1 \big) \frac{\partial}{\partial \sigma_1} +
%s_1 t a_2 \frac{\partial}{\partial t} + s_1 s_2  a_3 \frac{\partial}{\partial s_2}\\
\frac{\partial}{\partial \sigma_2}, \
 \frac{\partial}{\partial \tau}, \  \frac{\partial}{\partial Z_j}, \
 \frac{\partial}{\partial y},  \  \frac{\partial}{\partial \zeta_j} , \ \frac{\partial}{\partial \eta_j}.
\end{gathered}\end{equation}
 We see that $\tHL/s_1$, restricted to $\tpL = 0$, is a smooth vector field transverse to $\FL$ and tangent to $\FF$ and $\FR$.

\end{proof}

\subsection{Structure of $\Lambda_+$ near $N^* \diagz$} 
We start with some properties of the left and right Hamilton vector fields.

\begin{lemma}\label{lem:commute} The left and right Hamilton  vector fields $H^L$ and $H^R$, viewed as vector fields either on $T^* X^2_0$ or ${}^\Phi T^* X^2_0$,  satisfy
\begin{itemize}
\item[(i)] $H^L$ and $H^R$ commute; 
\item[(ii)] Both vector fields are tangent to $\Sigma_L \cap \Sigma_R$;
\item[(iii)] On $\Sigma_L \cap \Sigma_R$,  both vector fields are transverse to $\Sigma \cap N^* \diagz$;
\item[(iv)] $H^L - H^R$ is tangent to $\Sigma \cap N^* \diagz$;
\item[(v)] Both vector fields are tangent to $\partial_{\FF} {}^\Phi T^* X^2_0$.
\end{itemize}
\end{lemma}

\begin{proof}
Statement (i)  is a direct consequence of the fact that the differential operators $\Delta_L$ and $\Delta_R$, the left and right Laplacians, commute, as they operate in different sets of variables.

Since $\Sigma_L = \{ p_L = 1 \}$, the left vector field $H^L$ is tangent to $\Sigma_L$. On the other hand, as $[H^L, H^R] = 0$ we have $\{ p_L, p_R\} = 0$ which implies $H^R(p_L) = 0$. Hence $H^R$ is also tangent to $\Sigma_L$. By symmetry, both vector fields are tangent to $\Sigma_R$. This proves (ii).

Statement (iii) is
easily checked in local coordinates. By symmetry, it is enough to check the left vector field. In local coordinates away from $\FF$, this has the form
\begin{equation}
H^L = 2 g^{ij} \zeta_j \frac{\partial}{\partial z_i} - \frac{\partial g^{jk}(z)}{\partial z_i} \zeta_j \zeta_k \frac{\partial}{\partial \zeta_i}.
\label{lhvf}\end{equation}
Since $g^{ij} \zeta_i \zeta_j = 1$ on $\Sigma_L$, the length of the $\partial_z$ component is $1$ (it traces a unit-speed geodesic in $X^\circ$). On the other hand, $z'$ is fixed under the flow of $H^L$. It follows that $H^L$ is transverse to $\{ z = z'\}$. By symmetry, $H^R$ is also transverse to $\{ z = z'\}$.

A similar argument is valid near $\rho_F = 0$. In that case, the Hamiltonian $p_L$, viewed as a function on the standard cotangent bundle,  takes the form \eqref{lH-ff}. Since, according to \eqref{lH-ff}, we have
$$
\sigma^2 + g_0^{ij}(0, y) N_i N_j = 1
$$
at $\partial_{\FF} (N^* \diagz \cap \Sigma_L)$ (since $s=1$ at $\diagz$), we see that either $\dot s \neq 0$ or $\dot Y \neq 0$ along the flow of $H^L$ at $\partial_{\FF} (N^* \diagz \cap \Sigma_L)$. This shows transversality near $\rho_F = 0$. If we view the Hamilton vector field instead as living in the $\Phi$-cotangent bundle, then we use fibre coordinates $(\lambda, \mu, \lambda', \mu')$ and the expression for the left Hamilton vector field becomes
\begin{equation}\begin{gathered}
H^L = 2 \lambda s \frac{\partial}{\partial s} - 2s g_0^{ij}(x's, y) \mu_j\frac{\partial}{\partial Y_i} + 2x's g_0^{ij}(x's, y) \mu_j\frac{\partial}{\partial y_i} \\
- \Big( 2g_0^{ij}(x's, y) + (x's) \partial_x g_0^{ij}(x's, y) \Big) \mu_i \mu_j \frac{\partial}{\partial \lambda}
+ \Big( 2\lambda \mu_i - x' s \frac{\partial g_0^{jk}}{\partial y_i} \mu_j \mu_k \Big) \frac{\partial}{\partial \mu_i} .
\end{gathered}\label{lHvf-ff-0}\end{equation}
Again, we see that that, as $\lambda^2 + g_0^{ij} \mu_i \mu_j = 1$ on $\Sigma_L$, that  either $\dot s \neq 0$ or $\dot Y \neq 0$ along the flow of this vector field at $\partial_{\FF} (N^* \diagz \cap \Sigma_L)$. This establishes (iii).

To check (iv), we compute the time derivative of $z - z'$ and $\zeta + \zeta'$ under the flow of the vector field
$H^L - H^R$ in the interior of $T^* x^2_0$. Using \eqref{lhvf}, we obtain
\begin{equation}\begin{gathered}
\dot {(z-z')}_i = g_0^{ij}(z) \zeta_j + g_0^{ij}(z') \zeta'_j \\
\dot {(\zeta + \zeta')}_i = - \frac{\partial g^{jk}(z)}{\partial z_i} \zeta_j \zeta_k + \frac{\partial g^{jk}(z')}{\partial z'_i} \zeta'_j \zeta'_k.
\end{gathered}\end{equation}
The right hand side vanishes when $z=z'$ and $\zeta = -\zeta'$, showing tangency to $N^* \diagz \cap \Sigma_L$. By continuity this holds down to $\rho_F = 0$.

To prove (v), we note that $x'$ is a boundary defining function for the interior of $\FF$. Since $H^L(x') = 0$, this shows tangency of $H^L$ to $\partial_{\FF} {}^\Phi T^* X^2_0$. By symmetry, the same is true for the right Hamilton vector field.
\end{proof}

It follows from Lemma~\ref{lem:commute} that the interior $\Lambda_+^\circ$ of $\Lambda_+$ is foliated by 2-dimensional leaves, given by integral surfaces for $H^L$ and $H^R$. Those leaves that lie in the interior of ${}^\Phi T^* X^2_0$ are exactly the forward leaves $\gammatfb$ defined in Section~\ref{subsec:leaves}. 

Recall that $\Lambda_+$ is the forward flowout, by the left Hamilton vector field, from $N^* \diagz \cap \Sigma_L$.  It follows that $N^* \diagz \cap \Sigma_L$ is a boundary hypersurface of $\Lambda_+$, which we denote $\partial_{\diagz} \Lambda_+$. 
Due to Lemma~\ref{lem:commute}, it is also the forward flowout, by the right Hamilton vector field, from $N^* \diagz \cap \Sigma_R$ (noting that $N^* \diagz \cap \Sigma_L = N^* \diagz \cap \Sigma_R$).\footnote{It is the forward, rather than backward, flowout by the right Hamilton vector field due to the change in sign in the right fibre coordinate}

\begin{proposition}\label{prop:nd}
In a neighbourhood $U$ of $N^* \diagz \cap \Sigma_L$, $\Lambda_+$ is a smooth manifold with corners of codimension 2. The projection ${}^\Phi \pi: {}^\Phi T^* X^2_0 \to X^2_0$, restricted to $\Lambda_+ \cap U$, is a diffeomorphism from the interior of $\Lambda_+ \cap U$ to its image in $X^2_0$, while at  $\partial_{\diagz} \Lambda_+$,  the projection drops rank by $n$. Moreover, it does so nondegenerately, in the sense that $\det d({}^\Phi\pi |_{\Lambda_+ \cap U})$ vanishes to order precisely $n$ at $\partial_{\diagz} \Lambda_+$. 
\end{proposition}

\begin{proof}
First, we consider a neighbourhood of a point $q$ in the interior of $N^* \diagz$. Here, the $\Phi$-cotangent bundle of $X^2_0$ is locally isomorphic to the standard cotangent bundle of $X^2$. We use local coordinates $(z, z')$ near $\pi(q)$, with dual cotangent coordinates $(\zeta, \zeta')$. The left Hamilton vector field takes the form \eqref{lhvf}.
Suppose, by rotating the $z$ coordinates, that $\zeta' = (\zeta'_1, 0, \dots, 0)$ at $q$. Then we can use coordinates $z'_1, \dots, z'_{n+1}, \zeta'_2, \dots, \zeta'_{n+1}$ on $\partial_{\diagz} \Lambda_+ = N^* \diagz \cap \Sigma$ near $q$. Let $r$ be a time parameter along the flow; then $(z', \overline{\zeta}' = (\zeta'_2, \dots, \zeta'_{n+1}), r)$ locally furnish coordinates on $\Lambda_+$, where $r$ is a defining function for the boundary. In terms of these coordinates, $\zeta'_1$ is given by the positive root of the quadratic equation $g^{ij}(z') \zeta'_i \zeta'_j = 1$, and
 the $z$ coordinates satisfy
\begin{equation}
z_i = z'_i + r g^{ij}(z') \zeta'_j + O(r^2).
\end{equation}
This implies that
\begin{equation}
\frac{\partial z_1}{\partial r} = g^{11} \text{ at } q,  \quad \frac{\partial z_i}{\partial \zeta'_j} = r g^{ij} + O(r^2), \ i, g  = 2 \dots n+1.
\label{detdpi}\end{equation}
In these coordinates, $\pi$ is the map $(z', \overline{\zeta}', r) \mapsto (z, z')$. Since $g^{11}(q) > 0$, and  the matrix $g^{ij}$ for $i, j = 2 \dots n+1$ is positive definite, \eqref{detdpi} shows that for small $r$ we have
\begin{equation}
\det d\pi = O(r^n), \text{ and } \det d\pi  \geq c r^n \text{ for small } s.
\label{rn}\end{equation}
This proves the Proposition in a neighbourhood of an interior point $q \in N^* \diagz \cap \Sigma_L$.

A very similar argument proves the proposition in a neighbourhood of a point $q$ on the boundary of $N^* \diagz \cap \Sigma_L$, that is, such that $\pi(q) \in \FF$. In this case, we use coordinates as in region 4a above, assuming that $g_0^{ij}(y_0) = \delta^{ij}$ where $\pi(q) \in F_{y_0}$. In terms of these coordinates, the left Hamiltonian is
\begin{equation}
(s\sigma)^2 + g_0^{ij} (x \eta - s N)_i (x \eta - s N)_j  - 1
\label{lh-ff}\end{equation}
(we do not consider the shifted Hamiltonian here as we are working away from $\rho_L = 0$ or $\rho_R = 0$).
Under the Hamilton flow we have
$$
\dot s = 2 \sigma s^2, \quad \dot Y_i = 2g_0^{ij} (x \eta - s N)_j.
$$
Now we divide into two cases. Since \eqref{lh-ff} vanishes on $\Sigma_L$ by definition, we have either $|\sigma| > 1/2$ or $|N| > 1/2$ at $q$ (since $g_0^{ij} = \delta^{ij}$ at $q$). First suppose that $|\sigma| > 1/2$ at $q$. Then coordinates on $N^* \diagz \cap \Sigma_L$ can be taken to be $x', y, N$. In terms of these, $\sigma$ is given as the appropriate root of \eqref{lh-ff}, while $s, Y$ satisfy
\begin{equation}\begin{aligned}
s &= 1 + 2 \sigma s^2 r + O(r^2), \\
Y_i &= -2 g_0^{ij}(0, y') s N_jr + O(r^2 + r x')
\end{aligned}\label{detdpi2}\end{equation}
where again, $r$ is a time parameter along the flow.
In these coordinates, the map $\pi$ is given by $(x', y, N) \mapsto (x', y, s, Y)$.
Since $s$ is close to $1$, $\sigma$ is bounded away from $0$, and  the matrix $g_0^{ij}$ is positive definite, \eqref{detdpi2} shows that for small $r$ we have \eqref{rn}, proving the proposition in this case. The case when $|N| \geq 1/2$ is similar. In this case, we can, by rotating the $y$ coordinates, assume that $N_1 \geq 1/2$ at $q$. Then we use $(x', y, \sigma, N_2, \dots, N_n)$ and make a similar calculation. This completes the proof.
\end{proof}

We next record a standard result about Riemannian manifolds that will be useful. 

\begin{proposition}\label{prop:dist}
Let $(M,g)$ be any Riemannian manifold, and let $H$ be the Hamiltonian defined on $T^*(M \times M)$ by the (dual) metric $g^{-1}$ on the left factor, that is, $H(m, \xi, m', \xi') = |\xi|^2_{g^{-1}(m)}$. Let 
$S$ be the subset of $T^*(M \times M)$ given by the intersection of $N^* \diag$ and $\{ H = 1 \}$. Then the local forward Hamiltonian flowout from $S$ is given, in a deleted neighbourhood of $S$, by the graph of the differential of the geodesic distance function $\dist(m, m')$; that is, the graph of the differential of the distance function coincides with the flowout in a deleted neighbourhood of $S$. 
\end{proposition}

\begin{proof} This is standard, so we provide just a quick sketch. We consider the geodesic distance function $\Psi(m, m') = \dist(m, m')$, which is smooth in a deleted neighbourhood of the diagonal. For fixed $m'$, this satisfies $|d_m \Psi|^2 = 1$. Therefore, by Hamilton-Jacobi theory, the graph of the differential of $\Psi$ can be constructed as a union of integral curves of the flow 
\begin{equation}\begin{aligned}
\dot x_i &= 2 g^{ij}(x) \xi_j \\
\dot \xi_i &= - \partial_{x_i} g^{jk}(x) \xi_j \xi_j \\
\dot \Psi &= 2 g^{ij} \xi_i \xi_j
\end{aligned}\end{equation}
which is tangent to the set $\{ H = 1 \}$. We see from this that $(x, \xi)$ move according to Hamiltonian flow, while $\dot \Psi = 2$. Thus $\Psi$ is (up to this factor of 2) the time parameter along the flow. 

Performing a Legendre transform, each integral curve for the Hamilton vector field becomes a curve in $TM$ (for fixed $m'$) which is a lift to the tangent bundle of a curve $c$ on $M$ of (locally) shortest length for the Lagrangian function corresponding to our Hamiltonian $h$  (see e.g. \cite[Chapter 3]{Arnold}). In this case, the Lagrangian is the dual metric, i.e. the original metric $g$. Thus, $c$ is a geodesic, with speed $2$. Since $\Psi$ is twice the time parameter along each curve, it follows that $\Psi$ is given by the geodesic distance. 
\end{proof}

\begin{corollary} In a deleted neighbourhood of $\partial_{\diagz} \Lambda_+$, and away from $\rho_F = 0$, $\Lambda_+$ is parametrized by the geodesic distance function. 
\end{corollary}

%%%%%%%%

\subsection{Structure of $\Lambda_+$ near $T^*_{\FF} X^2_0$}

We next investigate the properties of $\Lambda_+$ at, and near, the boundary over $\FF$. Since both $H^L$ and $H^R$ are tangent to the boundary over $\FF$, the flowout from $\Sigma \cap N^* \diagz \cap \{ \rho_F = 0 \}$ remains at $\rho_F = 0$. 

Recall that the boundary hypersurface $\FF$  fibres over $\partial X$, with fibres that are quarter-spheres of dimension $n+1$.  We temporarily use coordinates $(y, z, \rho)$ near the interior of $\FF$, where $z = (x/x', Y = (y'-y)/x')$ is a coordinate on each fibre, $y'$ are coordinates on the base of the fibration, and $\rho = \rho_F$ is a boundary defining function for $\FF$. 

We first show that, in some sense, $\Lambda_+$ restricts to a Lagrangian submanifold  over the interior $F_{y_0}^\circ$ of each fibre $F_y$ of $\FF$. 

\begin{lemma}\label{lem:Lag-fibre}
The Lagrangian submanifold $\Lambda_+$, restricted to ${}^\Phi \pi^{-1} F_{y_0}$, is, in a natural way, a Lagrangian submanifold $\Lambda_{y_0}$ of $T^* F_{y_0}^\circ$. 
\end{lemma}

\begin{proof} This is more or less an abstract result about Lagrangian submanifolds on spaces with fibred boundary; see for example a very similar result in \cite[Proposition 4.3]{Hassell-Wunsch}.

We view $\Lambda_+$ here as a submanifold of the $\Phi$-cotangent bundle ${}^\Phi T^* X^2_0$. The bundle ${}^\Phi T^* X^2_0$ is the dual space of the bundle ${}^0 TX \times {}^0 TX$ lifted to $X^2_0$ via the blowdown map. This bundle is generated by vector fields that vanish at $\FL$ and $\FR$, and are tangent to the fibres of $\FF$. That is, using the coordinates $(y, z, \rho)$, they are vector fields of the form
$$
\rho \frac{\partial}{\partial y_i}, \quad \frac{\partial}{\partial z_j}, \quad \rho \frac{\partial}{\partial \rho}.
$$
Thus the $\Phi$-cotangent bundle ${}^\Phi T^* X^2_0$ is, in these coordinates, spanned by one-forms of the form
$$
\frac{dy_i}{\rho}, \quad dz_j, \quad \frac{d\rho}{\rho}.
$$
It follows that we have coordinates $(y, z, \rho; \overline{\mu}, \zeta, \omega)$ on ${}^\Phi T^* X^2_0$ over $FF^\circ$, where we write points of  ${}^\Phi T^* X^2_0$ near $\rho = 0$ in the form
$$
\sum_i \overline{\mu}_i \frac{dy_i}{\rho} + \sum_j \zeta_j  dz_j + \omega \frac{d\rho}{\rho}.
$$
In these coordinates, the symplectic form is
\begin{equation}
\sum_i \Big( d\overline{\mu}_i \wedge \frac{dy_i}{\rho} - \overline{\mu}_i   \frac{dy_i}{\rho}  \wedge \frac{d\rho}{\rho} \Big) + \sum_j d\zeta_j  \wedge  dz_j + d\omega \wedge \frac{d\rho}{\rho}
\label{sympl}\end{equation}
and this vanishes when restricted to $\Lambda_+$ as $\Lambda_+$ is Lagrangian.

Notice that at $\Sigma \cap N^* \diagz \cap \{ \rho_F = 0 \}$,
the differentials $dy'_1, \dots, dy'_n$ and $d\rho$ are linearly independent. Also, taking $\rho = x'$, we have $\mathcal{L}_{H^L} y'_i = \mathcal{L}_{H^L} \rho = 0$. It follows that $dy'_i$ and $d\rho$ pull back to themselves under the flow generated by $H^L$. Therefore, these differentials are linearly independent on the whole flowout from $\Sigma \cap N^* \diagz \cap \{ \rho_F = 0 \}$.

Using this fact, we multiply \eqref{sympl} through by $\rho^2$ and restrict to $\Lambda_+$. This is identically zero. At $\rho = 0$ the only term without a factor $\rho$ is $\sum_i \overline{\mu}_i dy_i d\rho$. Since the $dy_i$ and $d\rho$ are linearly independent on the flowout from $\Sigma \cap N^* \diagz \cap \{ \rho_F = 0 \}$, it follows that $\overline{\mu}_i$ vanishes on $\Lambda_+$ when $\rho = 0$.
We can therefore write $\overline{\mu}_i = \rho \overline{\eta_i}$ on $\Lambda_+$, where $\overline{\eta_i}$ are smooth functions of
$(y, z, \rho; \overline{\mu}, \zeta, \omega)$. In these coordinates we know
\begin{equation}
\sum_i  d\overline{\eta}_i dy_i + \sum_j d\zeta_j  dz_i + d\omega \frac{d\rho}{\rho} \text{ vanishes when restricted to the flowout at } \rho = 0 .
\label{sympl-2}\end{equation}
We now multiply \eqref{sympl-2} by $\rho$. Since $d\rho \neq 0$ at the flowout at the boundary, we see that $d\omega = 0$ there. Since $\omega=0$ at $\Sigma \cap N^* \diagz \cap \{ \rho_F = 0 \}$, we find that $\omega = 0$ on the flowout at the boundary.

Now consider the map from smooth vector fields on $X^2_0$ tangent to the fibres of $\FF$, to smooth vector fields on  $F_{y_0}^\circ$, obtained by restriction to this fibre. This induces a map from ${}^\Phi T_{F_{y_0}^\circ}X^2_0$ to $T F_{y_0}^\circ$.
There is a dual map from $T^* F_{y_0}^\circ$ to ${}^\Phi T^* X^2_0$, which in our local coordinates is given by
$$
(z, \zeta) \mapsto (y_0, z, 0; 0, \zeta, 0).
$$
We have seen that at $\rho = 0$, we have $\overline{\mu} = 0, \omega = 0$. Therefore we can pull $\Lambda_+ \cap \pi^{-1} (F_{y_0}^\circ)$ back to a subset $\Lambda_{y_0}$ of $T^* F_{y_0}^\circ$.

It follows from \eqref{sympl-2}, and the fact that $\omega = 0$ on $\Lambda_+$ when $\rho = 0$,  that $\sum_j d\zeta_j dz_j = 0$ when restricted to $\Lambda_{y_0}$.  Moreover, $\Lambda_{y_0}$  has dimension $n+1$, since $\Lambda_+$ has dimension $2(n+1)$, and we lose $n+1$ dimensions by restricting to $y = y_0, \rho = 0, \overline{\mu} = 0, \omega = 0$ (as $\overline{\mu} = 0$ and $\omega = 0$ automatically when $\rho = 0$). It follows that $\Lambda_{y_0}$ is Lagrangian.
\end{proof}

We now determine the nature of this Lagrangian $\Lambda_{y_0}$. Denote by $e$ the distinguished point on $F_{y_0}$ given by the intersection with $\diagz$. 

\begin{proposition}\label{prop:LambdaFF}
The left Hamiltonian determines a hyperbolic metric on $F_{y_0}^\circ$, and the Lagrangian $\Lambda_{y_0}$ is that generated by the graph of the differential of the function $\Psi$ given by hyperbolic distance to $e$. 
\end{proposition}

\begin{proof}
In terms of the coordinates $(y, z, \rho; \overline{\mu}, \zeta, \omega)$, where we now specify $z = (s = x/x', Y = (y'-y)/x')$, $\rho = x'$, and $\zeta = (\sigma, N)$ (so that $\zeta \cdot dz = \sigma ds + N \cdot dY$), we have 
$$
\mu = s\overline{\mu} - s N, \quad \mu' = N, \quad \lambda = s \sigma, \quad \lambda' = \omega - s \sigma - N \cdot Y. 
$$
Of course, by choosing coordinates appropriately, we can (and from now on, will) assume that $g_0^{ij}(y_0) = \delta^{ij}$. 
We now write the left  Hamiltonian vector field in these coordinates at $x' = 0$. As we saw in the proof of Lemma~\ref{lem:Lag-fibre}, at $x' = 0$, we have $\overline{\mu} = 0$ and $\omega = 0$. So we have, at $x' = 0$ on $\Lambda_+$, 
\begin{equation}
\mu = -sN, \quad, \lambda = s \sigma.
\label{mul}\end{equation}
It follows from \eqref{lHvf-ff-0} that at $x' = 0$, 
\begin{equation}\begin{gathered}
H^L = 2  s^2 \sigma \frac{\partial}{\partial s} - 2s^2 N_i\frac{\partial}{\partial Y_i} 
- 2s (\sigma^2 + |N|^2) \frac{\partial}{\partial \sigma} 
\end{gathered}\label{lHvf-ff-0-2}\end{equation}
(bearing in mind the transformation law $s \partial_s \mapsto s \partial_s - \sigma \partial_\sigma - N \cdot \partial_N$ when we move from $(s, Y, \lambda, \mu)$ coordinates to $(s, Y, \sigma, N)$ coordinates). We recognize this as the geodesic flow equations for the Hamiltonian 
%In these coordinates, the left and right Hamiltonians, $\lambda^2 + g_0^{ij} \mu_i \mu_j - 1$ and  ${\lambda'}^2 + g_0^{ij} \mu'_i \mu'_j - 1$  take the form, restricted to $\Lambda_{y_0}$ take the form 
\begin{equation}
s^2 \big( \sigma^2 + |N|^2 \big) - 1
%, \quad (s\sigma + N \cdot Y)^2 + g_0^{ij}(y_0) N_i N_j - 1.
\label{hyp-metric}\end{equation}
which is an exact hyperbolic metric on $F_{y_0}^\circ$ (in the upper half space model $(s, Y) \in \RR_+ \times \RR^n$). 
%This is not very surprising, since \eqref{hyp-metric} is precisely the left Hamiltonian restricted to $x' = 0$, in these coordinates. 

The proof is now completed by Proposition~\ref{prop:dist}. 
\end{proof}

We next determine the structure of the flowout from a neighbourhood of $\partial_{\FF}(N^* \diagz \cap \Sigma_L)$.

\begin{proposition}\label{prop:near-ff}
There is a neighbourhood $U$ of ${}^\Phi T^*_{\FF} X^2_0$
%$\partial_{\FF}(N^* \diagz \cap \Sigma_L)$ in $N^* \diagz \cap \Sigma_L$ 
such that the part of the forward flowout from $N^* \diagz \cap \Sigma_L$ by the vector field $H^{L}/\rho_L + H^R/\rho_R$ lying in $U$ is the graph of the differential of a function $\Psi$ defined on ${}^\Phi \pi(U)$ (except at $\partial_{\diagz} \Lambda_+$, where Proposition~\ref{prop:nd} applies). This function $\Psi$ is equal to the Riemannian distance function on the interior of $X^2_0$, and restricts to the hyperbolic distance on each fibre $F_{y_0}$. It has the form (where defined) 
\begin{equation}
\Psi = -\log \rho_L - \log \rho_R + \tilde \Psi, \quad \tilde \Psi \in C^\infty({}^\Phi \pi(U) \setminus \diagz). 
\end{equation}
%where $\tilde \Psi$ is $C^\infty$. 
\end{proposition}

\begin{remark}The reason we choose neither $H^L$ nor $H^R$ to generate the bicharacteristic flow-out is that $H^L$ ($H^R$) isn't transverse to $\FR$ ($\FL$). However $H^{L}/\rho_L + H^R/\rho_R$ is transverse to both $\FL$ and $\FR$, as we show below.
\end{remark}

\begin{proof}
Let us define vector fields $V^L$ and $V^R$ by $V^L = H^L/\rho_L$ and $V^R = H^R/\rho_R$. We claim that the flow corresponding to the sum of vector fields $V^L + V^R$  reaches the boundary of $\FF$ (that is, either $\rho_L = 0$ or $\rho_R = 0$ in uniformly finite time).

Recall that the flowout from $N^* \diagz \cap \Sigma_L$ is a union of two-dimensional leaves, given by the joint flowout of the commuting vector fields $H_L$ and $H_R$.  At  $\FF$, these lie over hyperbolic planes contained in a given fibre $F_{y_0}$ of $\FF$, in which $Y$ and $N$ are multiples of a fixed unit vector $N_0$. At $\rho_F = 0$, these can be parametrized explicitly.  Consider the map 
\begin{equation}\begin{gathered}
(r, r') \mapsto \big( x' = 0, \ y = y_0,  \ s = \frac{\sin r}{\sin r'}, \ Y = \frac{\cos r' - \cos r}{\sin r'} N_0,  \\
\lambda = \cos r, \ \lambda' = -\cos r', \ \mu = \sin r N_0, \ \mu' = - \sin r' N_0 
 \big). 
\end{gathered}\label{rr'map}\end{equation}
The right hand side solves the system \eqref{eqn: hamilton flow-out} at $FF$ satisfied by the bicharacteristic flow-out.
This extends from $\{ (r, r') \mid r \geq r', 0 < r, r'  < \pi \}$ to a smooth map $\kappa$  from the model leaf $L_{model}$ shown in Figure~\ref{fig:modelleaf}  (in which the corners $\{ r = r' = 0 \}$ and $\{ r = r' = \pi \}$ have been blown up) into ${}^\Phi T^*_{F_{y_0}} X^2_0$. Let $L_{y_0, N_0}$ be the image of $\kappa$ in ${}^\Phi T^* X^2_0$. Then \eqref{rr'map} extends to a diffeomorphism from $L_{model}$ onto $L_{y_0, N_0}$. Moreover, the union of all the $L_{y_0, N_0}$ gives $\partial_{\FF} \Lambda_+$ (it is not a disjoint union, however)\footnote{In fact, over a given  $F_{y_0}$, all the leaves join together at the sets $\{ (x' = 0, y = y_0, s \in \RR, Y = 0, \lambda = \pm 1, \lambda' = \mp 1, \mu = 0, \mu' = 0) \}$, which correspond to the edges AE and CD in Figure~\ref{fig:modelleaf}. A similar phenomenon happens in the asymptotically conic case; see the remark at the end of Section 11 of \cite{Hassell-Wunsch}}. 

It is easy to see that the edge DE is mapped into $\partial_{\diagz} \Lambda_+$. Moreover, plugging the right hand side of \eqref{rr'map} in \eqref{lHvf-ff-0-2}, one can see the vector field $\sin r \partial_r$ pushes forward to $H^L$, while $-\sin r' \partial_{r'}$ pushes forward to $H^R$. Under the map \eqref{rr'map}, the inverse image of $L_{y_0, N_0} \cap \{ \rho_R = 0 \}$ in $L_{model}$ is the boundary segment AB, while the inverse image of $L_{y_0, N_0} \cap \{\rho_L = 0\}$ is the boundary segment BC. 
For a vector field on $L_{model}$ that pushes forward to $H^L/\rho_L$ one can take $\sin(r/2) \cos(r'/2) \partial_r$, and for a vector field that pushes forward to $H^R/\rho_R$ one can take $-\sin(r/2) \cos(r'/2) \partial_{r'}$. It is easy to check that, on $L_{model}$, every integral curve of the vector field $\sin(r/2) \cos(r'/2) ( \partial_r - \partial_{r'})$ reaches either AB or BC in uniformly finite time. Consequently, on $L_{y_0, N_0}$, every integral curve of the vector field $V^L + V^R$ reaches the boundary in uniformly finite time. 

\begin{center}\begin{figure}
\includegraphics[width=0.7\textwidth]{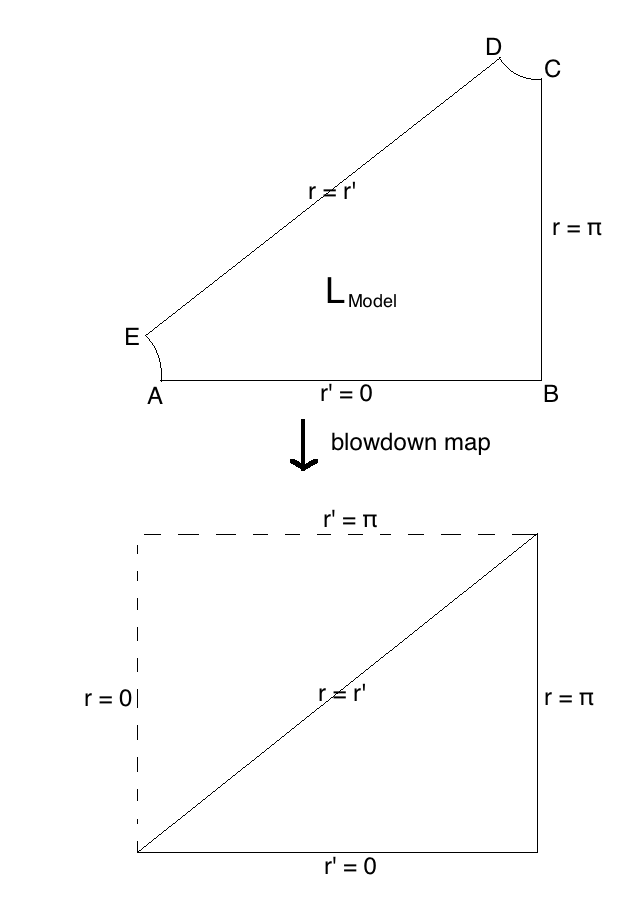}
\caption{\label{fig:modelleaf}The model leaf at the boundary}
\end{figure}
\end{center}

We now consider the flowout from a neighbourhood  of $\partial_{\FF} ( N^* \diagz \cap \Sigma_L)$ by the vector field $V^L + V^R$. By continuity, if the neighbourhood is sufficiently small, the flowout reaches the boundary in finite time. Therefore the flowout is smooth up to the boundary and is transversal to the boundary (transversality at the corner follows since the flowout is invariant under $V^L$ and $V^R$ separately). 

It is now clear that the closure $\Lambda_+$ of $\FBR$ coincides with this flowout in $U$.
Indeed, this flowout contains $\FBR \cap U$ and is closed, due to the finiteness of time in which the flowout reaches the boundary at $\rho_{\FL} = 0$ or $\rho_{\FR} = 0$. Conversely, suppose that $q \in {}^\Phi T^* X^2_0 \cap U$ is a limit point of $\FBR$. 
Let $q_n$ be a sequence in $\FBR$ converging to $q$. Each $q_n$ is associated to a geodesic $\gamma_n$ in ${}^0T^* X$. If the sequence $\gamma_n$ approaches the boundary ${}^0T^*_{\partial X} X$ uniformly, then $q$ lies over $\FF$, and is contained in the flowout from $\partial_{\FF} ( N^* \diagz \cap \Sigma_L)$ by the vector field $V^L + V^R$. If there is a subsequence $\gamma_{n_j}$ with $\sup x(\gamma_{n_j}(t)) \geq x_0 > 0$, then a further subsequence converges to an interior geodesic $\gamma^*$. Then $q$ lies on the  leaf corresponding to $\gamma^*$, which is contained in $\Lambda_+$.

We next show that $\Lambda_+ \cap U$ is the graph of a differential of a function, at least away from $N^* \diagz$. We have seen in Proposition~\ref{prop:LambdaFF} that at $\rho_F = 0$, the flowout is a graph over $\FF$, or in other words, if $w_1, \dots, w_{n+1}$ are local coordinates on $\FF$ (either in the interior or near the boundary), then $dw_i$ are linearly independent on $\partial_{\FF} \tL_+$; also, it was shown that $d\rho_F$ is also linearly independent on $\Lambda_+$ (and therefore also $\tL_+$) at $\rho_F = 0$. By continuity, this remains true for the flowout from $U$, if it is a  sufficiently small neighbourhood. 
This shows that, at least \emph{locally},  this flowout is given by the graph of the differential of a smooth function (away from $N^* \diagz \cap \Sigma_L$).  To show that this is \emph{globally} true (near $\rho_F = 0$, but globally along $\FF$), it remains to show that there is a neighbourhood $V$ of $\FF$ in $X^2_0$, such that for each point $m$ in $V$ there is exactly one point in ${}^\Phi T^*_m X^2_0$ in $\Lambda_+$ which can be reached by flowing from $N^* \diagz \cap \Sigma_L$ by the vector field $V^L + V^R$, while staying uniformly close to $\rho_F = 0$.  This is equivalent to the condition that for any point $y_0$ of $\partial X$, there is a small neighbourhood $V' \subset X$ of $(0, y_0)$ such that for any two points of $V'$,  there is a unique geodesic\footnote{Here, and in Lemma~\ref{lem:one},  we use the term `geodesic' in its traditional meaning as a curve in $X$ (rather than $T^* X$).} contained in $V'$ joining them. 
We prove this fact in Lemma~\ref{lem:one} below. 

%That there is one point follows from the existence of a short geodesic between any two points of $X^\circ$. Moreover, if the two points are $(x, y)$ and $(x', y')$ and if $x$, $x'$ and $|y-y|$ are small, then there is a unique geodesic whose projection to $X^\circ$ lies in a small neighbourhood of $(0, y)$. 

To prove the last statement, we consider the shifted Lagrangian $\tL_+$. As a consequence of the argument above, the part of $\tL_+$ that corresponds to $\Lambda_+ \cap U$  is the graph of the differential of a function $\tilde \Psi$, defined in a neighbourhood of $\FF$ in $X^2_0$, that is smooth away from $\diagz$. Therefore, the unshifted Lagrangian $\Lambda_+$ in this region is the differential of the function $\Psi = \tilde \Psi - d\rho_L/\rho_L - d\rho_R/\rho_R$. By Proposition~\ref{prop:dist}, $\Psi$ is the Riemannian distance function. This completes the proof of the proposition. 
\end{proof}

\begin{remark} The leaf $L_{y_0, N_0}$ is, not surprisingly in view of Figure~\ref{fig:modelleaf}, the limit of a sequence of interior leaves having the property that the associated sequence of bicharacteristics approach the boundary uniformly. 
\end{remark} 

\begin{lemma}\label{lem:one}
Let $\epsilon$ be sufficiently small. Then for each $y_0 \in \partial X$, 
there is a neighbourhood $V'$ of $(0, y_0) \in X$, which contains 
$\{ x < \epsilon, d(y, y_0) < \epsilon \}$ and is contained in $\{ x < 2\epsilon, d(y, y_0) < 2\epsilon \}$, such that for each pair of points $(x_1, y_1)$ and $(x_2, y_2)$ in $V'$, there is a unique geodesic  joining them that is contained in $V'$. Moreover this is the shortest geodesic joining $(x_1, y_1)$ and $(x_2, y_2)$. 
\end{lemma}

\begin{proof}
It suffices to find a neighbourhood $V'$ whose boundary is smooth and has nonnegative curvature, in the sense that for any smooth curve $c(s)$ contained in the closure of $V'$, such that $c(s_0)$ is in the boundary of $V'$ and $c'(s_0)$ is tangent to the boundary of $V'$, then the acceleration vector points inwards. Then we can run the argument of \cite[Theorem 6.11.3]{Jost} to show that there is a unique geodesic joining the two points. This argument uses the gradient flow for the energy functional on curves, and the positive curvature condition on $V'$ ensures that if a curve starts in $V'$, then the gradient flow remains in $V'$. 

We give an explicit example of such a $V'$. Let $y$ be normal coordinates for the metric $h |_{x=0}$ on $\partial X$ centred at $y_0$, extended into the interior such that the metric takes the form \eqref{g-normalform}. Then we let 
$$
V'  =  \{ (x, y) \mid |y-y_0|^2 + (x - \epsilon/2)^2 \leq 2\epsilon^2 \}.
$$
We now explain why  $V'$ has the positive curvature property. On hyperbolic space, with metric $x^{-2}(dx^2 + dy^2)$, the set $\{ (x, y) \mid |y-y_0|^2 + (x - \epsilon/2)^2 = 2\epsilon^2\}$ is a hypersphere, with constant positive curvature. A scaling and perturbation argument shows that is also has positive curvature for a general asymptotically hyperbolic metric (in $y$-normal coordinates) for sufficiently small $\epsilon$. 

The claim that this geodesic is the shortest between $(x_1, y_1)$ and $(x_2, y_2)$ follows from a homotopy argument and a simple length comparison. Any other geodesic between these two points must either be non-homotopic to the geodesic within $V'$, or reach a region with some nonnegative sectional curvature, or otherwise the homotopy argument of \cite[Theorem 6.11.3]{Jost} applies. But
if $\epsilon$ is sufficiently small, this means that the geodesic must reach the region
$$
\{ (x, y) \mid |y - y_0| \geq 1000 \epsilon \text{ or } x \geq 1000\epsilon \}.
$$
If this is so, trivial length estimates show that it must have greater length than the geodesic constructed within $V'$.  \end{proof}

%%%%%%%%%%%%%%%%%%%%%%%%%%%%%%%%

\subsection{Structure of $\Lambda_+$ away from $\partial_{\diagz} \Lambda_+$ and $\partial_{\FF} \Lambda_+$}\label{subsec:Lambdastar}

We now divide $\Lambda_+$ into a union of two open sets. The first, denoted $\Lambda_+^{nd}$ (for `near-diagonal'), will be a neighbourhood of $\partial_{\diagz} \Lambda_+$ in which Proposition~\ref{prop:nd} applies, together with a neighbourhood of $\partial_{\FF} \Lambda_+$ for which Proposition~\ref{prop:near-ff} applies. The second, which we will denote $\Lambda_+^*$, will be an open set disjoint from $\partial_{\diagz} \Lambda_+$ and $\partial_{\FF} \Lambda_+$. We similarly divide $\tL_+$ into two open sets $\tL_+^{nd}$ and $\tL_+^*$. 
It remains to analyze $\Lambda_+^*$, or $\tL_+^*$. 

To do this, we shall view $\Lambda_+ \setminus \partial_{\FF} \Lambda_+$ as a union of (interior) forward leaves. Each leaf corresponds to an interior bicharacteristic (geodesic, in our case) of ${}^0T^* X$. We note that some of those geodesics stay uniformly close to the boundary of $X$ --- say, where $x \leq \epsilon$ along the entire geodesic. The forward leaves corresponding to such geodesics will lie wholly within the part of $\Lambda_+$ treated in Proposition~\ref{prop:near-ff}, for sufficiently small $\epsilon$. Thus, it remains to consider forward leaves for which the underlying geodesic reaches the region $x > \epsilon$. 

It turns out that it is undesirable to view such leaves as living over the blown-up space $X^2_0$. This is because there may well be geodesics in which the limiting forward direction $y_\infty \in \partial X$ is equal to the limiting backward direction $y_{-\infty}$. In this case, the `antidiagonal' corners of the leaf would return to the front face $\FF$ of $X^2_0$. However, there seems no reason to suppose that, at such points, the set $\Lambda_+$ (or $\tL_+ \subset T^* X^2_0$) would have  a nice structure  (such as being a submanifold).

To avoid such difficulties, we simply observe that the blowup of the boundary of the diagonal plays no useful role at the `antidiagonal' corners of the forward leaves, and should be avoided. Therefore, we will view this part $\Lambda_+^*$ of $\Lambda_+$ as living, not on ${}^\Phi T^* X^2_0$, but on ${}^0 T^* X \times {}^0 T^* X$. Moreover, we shall immediately pass to the shifted leaves, for which we have Lemma~\ref{lem:int-shifted-leaves} (note that the corresponding result is not true for the unshifted leaves --- see Remark~\ref{rem:conicsing}). We recall here that the shift on ${}^0 T^* X \times {}^0 T^* X$ takes the form 
\begin{equation}
q \mapsto q + dx/x + dx'/x'. 
\label{shiftXX}\end{equation}

To be more precise, we consider an shifted interior leaf (for brevity, we omit the `forward' in `forward leaf' from now on). It has five boundary hypersurfaces, labelled AB, BC, CD, DE and EA in Figure~\ref{fig:intleaf}.  Here, DE is the intersection of the leaf with $N^* \diagz$, AB is the intersection with ${}^\Phi\pi^{-1} \FR$,  BC is the intersection with ${}^\Phi\pi^{-1} \FL$, and CD, EA are the intersections with ${}^\Phi\pi^{-1} \FF$, which has two components labelled by the value of $\lambda \in \{ -1, 1 \}$. 

\begin{center}\begin{figure}
\includegraphics[width=0.8\textwidth]{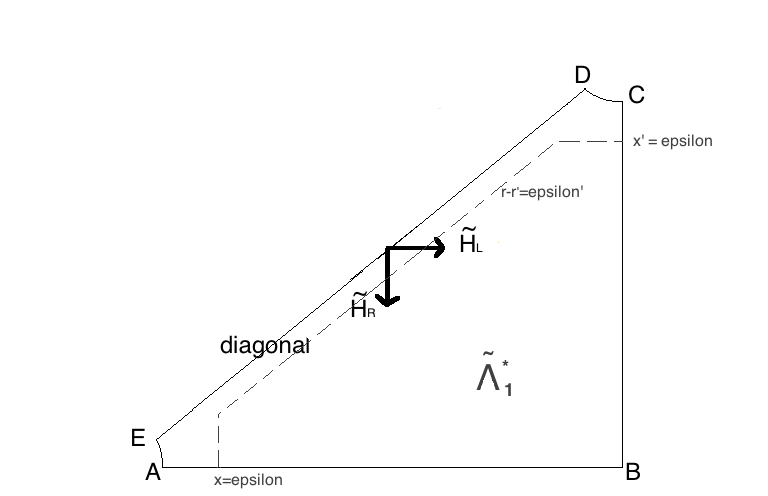}
\caption{\label{fig:intleaf}An shifted interior forward leaf. The part contained in $\tL_+^*$ is that part to the right and down from the dotted line.}
\end{figure}
\end{center}

%Consider the intersection of this leaf with $\Lambda_+^*$. Since $\Lambda_+^*$ is disjoint from $\partial_{\diagz} \Lambda_+$ and $\partial_{\FF} \Lambda_+$, the part of the leaf contained in $\Lambda_+^*$ is disjoint from $B_{\diagz}, B_{F,1}$ and $B_{F,-1}$. Because of that, we can dispense with the blowup of the leaf and think of this part of the leaf as a subset of $\gamma^{2, \prime}$. Thus, $\Lambda_+^*$ is (or can be viewed as)  an open subset of the  union of leaves $\gamma^{2, \prime}$ over all $\gamma$ which are not just interior leaves, but have the property that $x > \epsilon$ at some point along $\gamma$. This is clearly a smooth submanifold of ${}^0 T^* X \times {}^0 T^* X$. 
%
%Moreover, we consider the shifted Lagrangian $\tL_+^*$. As we view $\tL_+^*$ as living over $X^2$, the shift takes the form 
%\begin{equation}
%q \mapsto q + dx/x + dx'/x'. 
%\label{shiftXX}\end{equation}
%Then, under this transformation,  the leaf $\gamma^{2, \prime}$ is mapped to $\tilde \gamma^{2, \prime}$ which is a smooth submanifold when viewed as a subset of $T^* X \times T^* X$. Therefore, $\tL_+^*$ is  smooth  as a submanifold of $T^* X \times T^* X$.

Assume that the bicharacteristic contains a point where $x \geq \epsilon$ (otherwise, for $\epsilon$ sufficiently small, we may assume that the leaf is entirely contained in $\Lambda_+^{nd}$). 
We note that along each bicharacteristic, $\ddot x \leq 0$ for small $x$, say $x \geq \epsilon/2$. It follows that the part of each bicharacteristic where $x \leq \epsilon/2$ consists of two intervals, each containing one endpoint. 

We define $\tL_+^*$ to be the union, over all shifted leaves, of that part of the leaf to the right of the line $x \leq \epsilon/2$, below the line $x' \leq \epsilon/2$ and to the right of the line $r - r' = \epsilon'$ as indicated in the figure. By choosing $\epsilon, \epsilon'$ sufficiently small,  we arrange that $\tL_+^*$ contains $\tL_+ \setminus \tL_+^{nd}$. 

Notice that, in Figure~\ref{fig:intleaf}, horizontal motion represents motion in the left variables, and vertical motion represents motion in the right variables. Thus, the left Hamilton vector field $H^L$ restricted to this leaf is a horizontal vector field, pointing to the right, and the right Hamilton vector field $H^R$ restricted to this leaf is a vertical vector field,  pointing downward. 
(Notice that this means that $H^L - H^R$ is tangent to the diagonal, in accordance with Lemma~\ref{lem:commute}.) 
Because of this, our choice of $\tL_+^*$ also has the following property:
\begin{equation}
\text{Once a flow line of either $\tHL$ or $\tHR$ enters $\tL_+^*$, it stays in $\tL_+^*$ thereafter.}
\label{flowlineproperty}\end{equation}
This is useful in our analytic construction of the next section. The point is to make solving transport equations on $\Lambda_+$ as easy as possible. The initial condition for the transport equation on $\Lambda_+$ is at $\partial_{\diagz} \Lambda_+$, that is, in $\Lambda_+^{nd}$. As we move along the bicharacteristic, we may leave $\Lambda_+^{nd}$ and enter $\Lambda_+^*$. The procedure we follow is to cut off the solution in the overlap region $\Lambda_+^{nd} \cap \Lambda_+^*$ and complete the construction within $\Lambda_+^*$. Thus the condition \eqref{flowlineproperty} for each flow line in $\Lambda_+^*$ to stays in $\Lambda_+^*$ thereafter is to ensure that we do not have to repeatedly cut off and pass back to $\Lambda_+^{nd}$.

%To see that we can do this, we consider a leaf $\gamma^{2, \prime}$ which we represent as a square. See figure, in which the horizontal motion represents motion in the left variables, and vertical motion represents motion in the right variables. Thus, the left Hamilton vector field $H^L$ restricted to this leaf is a horizontal vector field, which we take to be pointing to the right, and the right Hamilton vector field $H^R$ restricted to this leaf is a vertical vector field, which we take to be pointing downward. 
%Notice that this means that $H^L - H^R$ is tangent to the diagonal, in accordance with Lemma~\ref{lem:commute}. The diagonal is indicated with a dotted line. 
%
%
%We can choose $\Lambda_+^*$ so that the intersection with each leaf is the region
%to the right of the line $x \leq \epsilon/2$, below the line $x' \leq \epsilon/2$ and to the right of the line $r - r' = \epsilon'$ as indicated in the figure. Then it is certainly the case that once a flow line of either $H^L$ or $H^R$ enters $\Lambda_+^*$, it stays in $\Lambda_+^*$ thereafter. Also, by choosing $\epsilon, \epsilon'$ sufficiently small,  we arrange that $\Lambda_+^*$ contains $\Lambda_+ \setminus \Lambda_+^{nd}$. 

%\begin{center}
%\includegraphics[width=0.8\textwidth]{bicharacteristic.png}
%\end{center} 

To summarize the results of this section, we have proved

\begin{proposition}\label{prop:flowout-structure}
The forward bicharacteristic relation \eqref{Lambda1def} 
can be expressed as the union of two relatively open subsets $\FBR^{nd} \cup \FBR^*$ , having the following properties. 
\begin{itemize}
\item Let $\Lambda_+^{nd}$ denote the lift of $\FBR^{nd}$ to ${}^\Phi T^* X^2_0$, together with its limit points lying over $\FF$, $\FL$ and $\FR$. 
%in ${}^\Phi \pi^{-1} \FF$, ${}^\Phi \pi^{-1} \FL$ and ${}^\Phi \pi^{-1} \FR$. 
Then this is a manifold $\Lambda_+^{nd}$ with codimension three corners, having the properties listed in Proposition~\ref{prop:nd} and Proposition~\ref{prop:near-ff}. The boundary hypersurfaces are  $\partial_{\diagz} \Lambda_+^{nd} = \Lambda_+^{nd} \cap N^* \diagz$ lying over the diagonal, and $\partial_{\FF}\Lambda_+^{nd}$, lying over $\FF$,  $\partial_{\FL}\Lambda_+^{nd}$, lying over $\FL$ and $\partial_{\FR}\Lambda_+^{nd}$, lying over $\FR$. 

\item The image $\tL_+^{nd}$ of $\Lambda_+^{nd}$ under the shift \eqref{tL1}  is a Lagrangian submanifold of $T^* X^2_0$ with codimension three corners, having the properties listed in Proposition~\ref{prop:nd} and Proposition~\ref{prop:near-ff}. The boundary hypersurfaces are $\partial_{\diagz} \tL_+^{nd} = \tL_+^{nd} \cap N^* \diagz$ lying over the diagonal, and $\partial_{\FF}\tL_+^{nd}$, lying over $\FF$,  $\partial_{\FL}\tL_+^{nd}$, lying over $\FL$ and $\partial_{\FR}\tL_+^{nd}$, lying over $\FR$. 

\item Let $\widetilde{\FBR^*} \subset T^* X^2$ denote the image of $\FBR^*$ under the shift \eqref{shiftXX}, and let $\tL_+^*$ denote $\widetilde{\FBR^*}$ together with its limit points in $T^* X^2$ lying over the left and right boundaries. Then $\tL_+^*$ 
 is a Lagrangian submanifold of $T^* X^2$  with codimension two corners. Moreover, flow lines of the vector fields $\tHL$ and $\tHR$ that enter $\tL_+^*$ remain in $\tL_+^*$.  
\end{itemize}

\end{proposition}
 
\begin{remark}\label{rem:comparison} In the case studied by Melrose, S\'{a} Barreto and Vasy \cite{Melrose-Sa Barreto-Vasy}, a small perturbation of the standard metric on hyperbolic space, the situation where the initial and final directions of an interior geodesic coincide cannot occur. Thus, in their case, they do not need to decompose the Lagrangian $\Lambda_+$ into pieces. In fact, in that case, the Lagrangian $\Lambda_+$ is globally given by the graph of the differential of the distance function (except at $\partial_{\diagz} \Lambda_+$ where the structure is given by Proposition~\ref{prop:nd}). If we make the assumption here that this situation does not occur, then it is not necessary to decompose $\tL_+$, which will be a manifold with corners of codimension 3 in $T^* X^2_0$. 
\end{remark}

\begin{remark}\label{rem:conicsing} The statement of Lemma~\ref{lem:int-shifted-leaves} does not hold for the unshifted leaves. In fact, all the bicharacteristics with a given final direction $y_\infty \in \partial X$ meet at the point $(x=0, y_\infty, \lambda = -1, \mu = 0)$. Consequently, the leaves are not disjoint, and there may be a conic singularity at the site where different leaves, viewed in ${}^0 T^* X \times {}^0 T^* X$,  intersect. This conic singularity is eliminated by blowup of the sets $\{ x = 0, \lambda = \pm 1, \mu = 0\}$ and $\{ x' = 0, \lambda' = \pm 1, \mu' = 0\}$. By passing to the shifted leaves and viewing them on the standard cotangent bundle, we are implicitly performing such a blowup. Passing to the b-cotangent bundle, as was done in \cite{Melrose-Sa Barreto-Vasy}, amounts to blowing up the larger sets $\{ x = 0,  \mu = 0\}$ and $\{ x' = 0,  \mu' = 0\}$ which also resolves the conic singularity. 

Similar conic singularities appear in the construction of the resolvent on asymptotically conic spaces at high energy \cite{Hassell-Wunsch}. Indeed, there they play a much greater role, and cannot be avoided just by passing to a different cotangent bundle as we do here. 

We also remark that the shift followed by passing to a different cotangent bundle is very strongly analogous to the procedure in \cite[Lemma 3.3]{Hassell-Wunsch-2005},
where one passes from the \emph{quadratic} scattering cotangent bundle to the usual  scattering cotangent bundle. 
\end{remark}

%%%%%%%%%%%%%%%%%%%%%%%%%%%%%%%%%
%%%%%%%%%%%%%%%%%%%%%%%%%%%%%%%%%
%%%%%%%%%%%%%%%%%%%%%%%%%%%%%%%%%

\section{Full parametrix and resolvent at high energy}\label{sec:parametrix}
We now construct a parametrix for the resolvent kernel, or more precisely the limit of the resolvent on the spectrum, as we approach it from above or below. We proceed in stages. The general idea, as outlined in Section~\ref{subsec:outline}, is to solve the PDE
\begin{equation}
P_h G_h = \delta,
\label{PGd}\end{equation}
where $\delta$ is the kernel\footnote{In this section, `kernel' always means `Schwartz kernel', not `nullspace'.} of the identity operator, that is, the delta function supported at the diagonal $\diag \subset X^2$, times the Riemannian half-density in both factors, which for brevity we write  $|dg dg'|^{1/2}$. (Note that we will take our operators to act on half-densities, which we can always identify with functions via the Riemannian half-density.) 

There is more than one solution to \eqref{PGd}. To specify a unique solution, we impose a microlocal condition. Notice that $\delta$ is a Lagrangian distribution associated to $N^* \diagz$ (when viewed as living on the $0$-double space $X^2_0$), while $P$, viewed as acting in the left variables, is a real principal type operator whose (semiclassical) characteristic variety is $\Sigma_L$, that intersects 
$N^* \diagz$. General microlocal theory (H\"ormander's propagation of singularities theorem, in its semiclassical version --- see \cite[Section 12.3]{zworski})  says that any microlocal solution of \eqref{PGd} in a microlocal neighbourhood of $N^* \diagz$ will have (semiclassical) wavefront set contained in $N^* \diagz$ together with the bicharacteristic flowout from $N^* \diagz \cap \Sigma_L$ (either forward or backward in time, or possibly both, perhaps depending on the location within $N^* \diagz \cap \Sigma_L$). We impose the microlocal condition that the semiclassical wavefront set is contained only in the \emph{forward}  flowout from $N^* \diagz \cap \Sigma_L$, which gives a unique solution. At the end of the construction, by comparing our result to the construction of Mazzeo-Melrose, we shall see that we have constructed the outgoing resolvent $R(n/2 - i/h) = (h^2 \Delta - h^2 n^2/4 - (1 - i0))^{-1}$. In this way, we avoid a priori considerations concerning which solution we should choose to obtain the outgoing resolvent. 

In our construction, the parametrix $G$ will be built up in several stages. These intermediate operators will be denoted $G_1, G_2, \dots$ and the corresponding error terms $P G_i - \delta$ will be denoted $E_i$. Our goal is to construct $E_i$ that is as `small' as possible: in particular, we want $\Id + E_i$ (thinking of $E_i$ as the kernel of an operator) to be invertible, and to have a good understanding of the kernel of the inverse.

\subsection{Elliptic construction}

In the first stage, we use a pseudodifferential construction to solve away the symbol of $\delta$ on $N^* \diagz$ away from the propagating region (that is, away from $\Sigma_L$). This will solve away the singularities of $\delta$ completely for $h > 0$, but there will still be a compact region of semiclassical wavefront set to be solved away, which will happen in the second stage. 

The operator $P = P_h$ is a semiclassical 0-differential operator with symbol $g^{ij} \zeta_i \zeta_j - 1$ (in local coordinates in the interior) or $\lambda^2 + g_0^{ij} \mu_i \mu_j - 1$ (near the boundary). Thus the symbol is elliptic for $|\zeta|^2_g \geq 2$, or $\lambda^2 + |\mu|_{g_0}^2 \geq 2$. We find an elliptic parametrix for $P_h$ in this region: that is, a semiclassical 0-pseudodifferential operator $G_1$, say with symbol supported where $|\zeta|^2_g \geq 3/2$, or $\lambda^2 + |\mu|_{g_0}^2 \geq 3/2$ such that the symbol of $E_1 = P G_1 - \Id$ is order $-\infty$ for $|\zeta|^2_g \geq 2$ or $\lambda^2 + |\mu|_{g_0}^2 \geq 2$ (see \cite[Section 4.7]{zworski}). We may assume that the kernel of $E_1$ is supported in any given neighbourhood of $\diagz$. We can view $E_1$ as a pseudodifferential operator, or alternatively, as a Lagrangian distribution in $I^{0, -\infty}(N^* \diagz)$ with semiclassical order $0$ and differential order $-\infty$; indeed, it has compact microsupport. Thus the kernel is smooth for $h > 0$, but not uniformly as $h \to 0$.

\subsection{Intersecting Lagrangian construction}

We next look for a $G_2'$ that solves away the error at $N^* \diagz$ completely. To do this, we use the Melrose-Uhlmann construction (or rather its semiclassical version --- see Appendix~\ref{sec:appB}). Thus we look for an intersecting Lagrangian distribution $G_2'$ associated to the pair of Lagrangian submanifolds $(N^* \diagz, \Lambda_+^{nd})$ (here we work microlocally near $N^* \diagz$, and will cut off the symbol on $\Lambda_+$ so that it is supported in $\Lambda_+^{nd}$). 

We note that the geometric conditions for $(N^* \diagz, \Lambda_+^{nd})$ together with the Hamilton vector field $H^L$ in Theorem~\ref{principal symbol semiclassical intersecting} are satisfied. Namely, the submanifolds $(N^* \diagz, \Lambda_+^{nd})$ meet cleanly, and the Hamilton vector field is transverse to $N^* \diagz$ and tangent to $\Lambda_+^{nd}$. Of course, we defined $\Lambda_+$ as the bicharacteristic flowout from $N^* \diagz \cap \Sigma_L$ precisely so that this is true. The key point is the non-tangency of $H^L$ to $N^* \diagz$ at its intersection with $\Sigma_L$, which was checked in Lemma~\ref{lem:commute}.

The Melrose-Uhlmann construction (or its semiclassical version described in Appendix~\ref{sec:appB}) shows that there is $G_2' \in I^{1/2}(N^* \diagz, \Lambda_+^{nd})$ such that  $E_2 = P G_2' + E_1$ is microsupported away from $N^* \diagz$. In fact we will have $E_2 \in I^{-1/2}(\Lambda_+^{nd})$ where this error comes from cutting off the symbol outside a neighbourhood of $N^* \diagz$. 

Defining $G_2 = G_1 + G_2'$, we have $P G_2 - \delta = E_2$, where $E_2$ is as just described. 

\subsection{Solving away errors on $\Lambda_+$}
The third stage is to solve away the errors on $\Lambda_+$ completely. The error term is a Lagrangian distribution on $\Lambda_+$, microsupported near, but not at,  $N^* \diagz$. 

We solve this error away by iteratively solving transport equations along $\Lambda_+$, using \eqref{vanishing-symbol-calculus}. There is no difficulty in doing this on the interior of $\Lambda_+$. Indeed, the nontrapping assumption on $(M, g)$ is equivalent to pseudoconvexity of the operator in the sense of \cite[Section 6.5]{Duistermaat-Hormander-Acta-1972}, guaranteeing that we can find global parametrices. We use the symbol calculus for Lagrangian distributions. To solve away the error $E_2 \in I^{-1/2}(\Lambda_+)$, at least microlocally, we find $G_{3, 1}^{nd} \in I^{1/2}(\Lambda_+^{nd})$ satisfying the transport equation\footnote{Note that the subprincipal symbol of the Laplacian is zero, so the zeroth order term in \eqref{vanishing-symbol-calculus} vanishes.}
\begin{equation*}
\mathcal{L}_{H^L} g_{3,1}^{nd}  = \sigma^{-1/2}(E_2),
\end{equation*}
which we solve `forward' along the bicharacteristics (so that the support of $g_{3,1}^{nd}$ is forward along the bicharacteristic relative to the symbol of $E_2$).  We then cut off this symbol within $\Lambda_+^{nd} \cap \Lambda_+^*$, and then find $G_{3,1}^*$ with symbol $g_{3,1}^*$ solving away the resulting error. (Here we make use of \eqref{flowlineproperty}, guaranteeing that when we solve this error away, we remain in $\Lambda_+^*$ microlocally.) Let $G_{3,1} = G_{3,1}^{nd} + G_{3,1}^*$. 
This reduces the order of the error at $\Lambda_+$ to $-3/2$. Inductively, given an error term $E_{2, k}$ in $I^{-k-1/2}$, we can solve this away with a term $G_{3, k} \in I^{-k+1/2}(\Lambda_+^{nd}) + I^{-k+1/2}(\Lambda_+^*)$ satisfying the transport equation
\begin{equation}
\mathcal{L}_{H^L} g_{3,1}  = \sigma^{-k-1/2}(E_{2,k}),
\label{trl}\end{equation}
reducing the error term to $E_{2, k+1} \in I^{-(k+1) - 1/2}(\Lambda_+)$. 

This works perfectly on compact subsets of the interior of $\Lambda_+$, but we must address the regularity of the symbol at the boundary. To this end, let $G_3'$ be an asymptotic sum of the $G_{3, j}$.
At least microlocally on compact subsets of the interior of $\Lambda_+$, this solves away the error term $E_2$. 
We define the kernel  $K$ by 
\begin{equation}
K = \rho_L^{1/2} \rho_R^{1/2}  e^{i \log \rho_L/h } e^{i \log \rho_R/h } G_3'.
\label{Kdefn}\end{equation}

 %We shall show

\begin{proposition}\label{prop:K} The kernel $K$ is $\rho_F^{-(n+1)/2}$ times a smooth Lagrangian distribution associated to $\tL_+$, in the sense that its symbol is $\rho_F^{-(n+1)/2}$ times a smooth function on $\tL_+$ times a smooth half-density on $\tL_+$. That is, 
$$
\rho_F^{(n+1)/2} K \in  I^{1/2}(X^2_0, \tL_+^{nd}; \Omega^{1/2}) + I^{1/2}(X^2, \tL_+^*; \Omega^{1/2}). 
$$
\end{proposition}

\begin{definition}
We define the space $I^{k}(X^2_0, \Lambda_+^{nd}; \Omegazh)$ by
$$
I^{k}(X^2_0, \Lambda_+^{nd}; \Omegazh) = \rho_F^{-(n+1)/2} (\rho_L \rho_R)^{-(n+1)/2-i/h} I^{k}(X^2_0, \tL_+^{nd}; \Omega^{1/2}).
$$
%We also define 
%$$
%I^{k}(X^2_0, \Lambda_+; \Omegazh) = \rho_F^{-(n+1)/2} (\rho_L \rho_R)^{-(n+1)/2-i/h} I^{k}(X^2_0, \tL_+^{nd}; \Omega^{1/2}) + (xx')^{-(n+1)/2-i/h} I^{k}(X^2, \tL_+^{*}; \Omega^{1/2}) .
%$$
\end{definition}

\begin{remark} This definition takes account of the relation between the $0$-density bundle and the standard density bundle in \eqref{Omega}. 
\end{remark}

We can thus write the conclusion of Proposition~\ref{prop:K} in terms of $G$ and the geometrically more natural 0-half-densities:  
%Writing in terms of $G$, and switching to the (more geometrically natural) 0-half densities, we have 
\begin{equation}
G_3' \in (\rho_L \rho_R)^{n/2-i/h} I^{1/2}(X^2_0, \tL_+^{nd}; \Omegazh) + (xx')^{n/2 - i/h} I^{1/2}(X^2, \tL_+^*; \Omegazh).
\end{equation}

%Writing in terms of $G$, and switching to the (more geometrically natural) 0-half densities, we have 
%\begin{equation}
%G_3' \in (\rho_L \rho_R)^{n/2} I^{1/2}(X^2_0, \Lambda_+^{nd}; \Omegazh) + (xx')^{n/2 - i/h} I^{1/2}(X^2, \tL_+^*; \Omegazh).
%\end{equation}
%\end{proposition}

The proof of this proposition will occupy the rest of this subsection. 
To prove the smoothness statement in the proposition we will  need to use the right as well as the left transport equation, so we first show

\begin{lemma}\label{lem:symm}
Consider the sum $S_{3, N}$ of the first $N$ terms of the $G_{3, k}$ defined above. Then the kernel of $S_{3, N}$ is microlocally symmetric, in the sense if $S^t_{3, N}$ denotes the transpose of $S_{3, N}$, we have, over the interior $X^\circ$ of $X$, 
$$
S_{3, N} - S^t_{3, N} \in I^{1/2 - N}(N^* \diagz, \Lambda_+).
$$
\end{lemma}

\begin{proof}
We have, with $P_L$ indicating the operator $P$ in the left variables, acting on $X^2_0$, 
$$
P_L S_{3, N} - \delta = E^{3, N+1} \in I^{-N-3/2}(\Lambda_+).
$$
We apply $P_R$, the same operator in the right variables:
$$
P_R P_L S_{3, N} - P_R \delta = P_R (E^{3, N+1}) \in I^{-N-3/2}(\Lambda_+).
$$
We know that $P_R$ and $P_L$ commute, as differential operators on $X^2_0$. Also, $P_R \delta$ is the Schwartz kernel of $P^* = P$, where $P^*$ is the formal adjoint. Also, $P_L \delta$ is the Schwartz kernel of $P$. So $P_R \delta = P_L \delta$. We deduce that 
\begin{equation}
P_L (P_R S_{3, N} - \delta) = P_R (E^{3, N+1}) \in I^{-N-3/2}(\Lambda_+).
\label{PlPr}\end{equation}
Since $S_{3, N}$ is an element of $I^{1/2}(N^* \diagz, \Lambda_+)$, so also $P_R S^t_{3, N} - \delta$ is an element of $I^{1/2}(N^* \diagz, \Lambda_+)$. Equation \eqref{PlPr} says that after applying $P_L$, the order at $\Lambda_+$ is reduced to $-N-3/2$. Therefore, the leading symbol $\sigma(P_R S_{3, N} - \delta)$ of $P_R S_{3, N} - \delta$ on $\Lambda_+$,  of order $1/2$, must satisfy the transport equation 
\begin{equation}
\mathcal{L}_{H^L} \sigma(P_R S_{3, N} - \delta)  = 0 .
\label{tr-eqn-sigma}\end{equation}
Suppose that $\sigma(P_R S_{3, N} - \delta)$ were nonzero at some point. 
Then, as it solves the homogeneous equation \eqref{tr-eqn-sigma}, it would be nonzero along the whole bicharacteristic through that point. This bicharacteristic passes into the \emph{backward} flowout from $N^* \diagz \cap \Sigma_L$, which would mean that there are points in the backward flowout from $N^* \diagz \cap \Sigma_L$ in the semiclassical wavefronts set of $P_R S_{3, N} - \delta$. This is impossible: $S_{3, N}$ and $\delta$ are microsupported at $N^* \diagz$ together with the \emph{forward} flowout from\footnote{The backward flowout does not meet the forward flowout: if it did, there would be a periodic geodesic, contradicting the nontrapping hypothesis} $N^* \diagz \cap \Sigma_L$. So the leading symbol of $P_R S_{3, N} - \delta$ is zero. Inductively, we see that the symbols on $\Lambda_+$ of every order down to $-N+1/2$ vanish. 

In addition, the ellipticity of $P_L$ on $N^* \diagz \setminus \Sigma_L$ shows that the symbol of $P_R S_{3, N} - \delta$ at $N^* \diagz$ vanishes to all orders larger than $-N-1$. It follows that $P_R S_{3, N} - \delta$ is an element of $I^{-1/2-N}(N^* \diagz, \Lambda_+)$.

Now we take the transpose, obtaining $P_L S^t_{3, N} - \delta$ is an element of $I^{-1/2-N}(N^* \diagz, \Lambda_+)$. However, ellipticity (at $N^* \diagz$) and the transport equation (at $\Lambda_+$) show that there is a \emph{unique} solution $u$ in the space $I^{1/2}(N^* \diagz, \Lambda_+)$  to the equation 
$$P_L u -  \delta \in I^{-1/2-N}(N^* \diagz, \Lambda_+),
$$
modulo $I^{1/2-N}(N^* \diagz, \Lambda_+)$. We conclude that $S_{3, N}$ is equal to  $S^t_{3, N}$ modulo $I^{1/2 - N}(N^* \diagz, \Lambda_+)$, as claimed. 

\end{proof}

\begin{proof}[Proof of Proposition~\ref{prop:K}]
The only issue here is the boundary regularity of the symbol of $K$. We investigate this in turn in the regions listed in Lemma~\ref{lem:shifted-vf}. 

To do this, we transform the operator $P_L$ to a more convenient operator. Notice that $P_L$ is formally self-adjoint on $L^2(X^2_0; dg dg')$.  For the purposes of symbol calculus, it is convenient to work on an $L^2$ space with respect to Lebesgue measure $d\mu$ in local coordinates. For example, in region 1 the Lebesgue measure takes the form $dz dz'$, in region 2, it takes the form $dx dy dz'$, and so on.  We may write $d\mu = (\rho_L \rho_R \rho_F)^{n+1} a^2 dg dg'$, where $a$ is smooth and nonvanishing. Then multiplication by $(\rho_L \rho_R \rho_F)^{(n+1)/2} a$ is a unitary transformation from $L^2(d\mu)$ to $L^2(dg dg')$. We conclude that the operator
\begin{equation}
a^{-1} (\rho_L \rho_R \rho_F)^{-(n+1)/2} P_L (\rho_L \rho_R \rho_F)^{(n+1)/2} a
\label{Plconj}\end{equation}
is formally self-adjoint on $L^2(d\mu)$. This remains true if we multiply by $(\rho_L \rho_R)^{-1/2}$ on each side. In addition, we can conjugate by a complex function of norm one, as this is also a unitary transformation. Thus we define the formally self-adjoint operator  $Q_L$ (on $L^2(d\mu)$) by 
\begin{equation}
Q_L = e^{i\log (\rho_L \rho_R)/h} a^{-1} (\rho_L \rho_R)^{-(n+2)/2}  \rho_F^{-(n+1)/2} P_L (\rho_L \rho_R)^{n/2}  \rho_F^{(n+1)/2} a \, e^{-i\log (\rho_L \rho_R)/h}.
\label{QlPl}\end{equation}

We now interpret the operator $Q_L$ in terms of  half-densities. We regard $P_L$ as acting on half-densities by letting
$$
P_L^g ( f |dgdg'|^{1/2}) = (P_L f) |dgdg'|^{1/2},
$$
where here $P_L$ on the RHS operates on functions on $X^2_0$, and $P_L^g$ on the LHS acts on half-densities. In other words, we define $P_L^g$ on half-densities via the flat connection on the half-density bundle that annihilates the Riemannian half-density $|dg|^{1/2}$. We define $Q_L^g$ the same way. However, it is more convenient for the purposes of calculations to write the operators with respect to the connection that annihilates the coordinate half-density $|d\mu|^{1/2}$ (this is implicitly done in \cite{Duistermaat-Hormander-Acta-1972}). If we switch to this connection, then we obtain operators $P_L^\mu$, $Q_L^\mu$ defined by 
$$
P_L^\mu ( f |d\mu|^{1/2}) = (P_L f) |d\mu|^{1/2}, \quad Q_L^\mu ( f |d\mu|^{1/2}) = (Q_L f) |d\mu|^{1/2}. 
$$
These are clearly related by conjugation with $(\rho_L \rho_R \rho_F)^{(n+1)/2} a$: 
$$
P_L^g = (\rho_L \rho_R \rho_F)^{-(n+1)/2} a^{-1} P_L^\mu (\rho_L \rho_R \rho_F)^{(n+1)/2} a .
$$
That is, \eqref{Plconj} is the correct expression for our operator $x' P_L^g$,  where we use the connection that annihilates the coordinate half-density $|d\mu|^{1/2}$ for computational convenience. We usually denote $P_L^g$ by $P_L$ below (unless emphasis is required); hopefully it will be clear from context  whether we are thinking of $P_L$ as acting on functions or half-densities. 

Clearly, \eqref{QlPl} implies that
$$
Q_L^\mu = e^{i\log (\rho_L \rho_R)/h} a^{-1} (\rho_L \rho_R)^{-(n+2)/2}  \rho_F^{-(n+1)/2} P_L^\mu (\rho_L \rho_R)^{n/2}  \rho_F^{(n+1)/2} a \, e^{-i\log (\rho_L \rho_R)/h}.
$$
Thus, we have 
\begin{equation}
Q_L^\mu =  e^{i\log (\rho_L \rho_R)/h} (\rho_L \rho_R)^{-1/2} P_L^g (\rho_L \rho_R)^{-1/2} e^{-i\log (\rho_L \rho_R)/h}.
\label{QmuPg}\end{equation}

Notice that $Q_L$ is constructed so that 
\begin{equation}
Q_L^\mu K = O(h^\infty),
\end{equation}
which follows immediately from $P_L G_3' = O(h^\infty)$ and from \eqref{Kdefn}, \eqref{QmuPg}. 
Moreover, the operator $Q_L$ is directly related to the calculations in Lemma~\ref{lem:shifted-vf}. Multiplication by $e^{i\log (\rho_L \rho_R)/h}$ has the effect of shifting the Lagrangian submanifold associated to $K$ from $\Lambda_+$ to $\tL_+$. (One can think of multiplication by $e^{i\log (\rho_L \rho_R)/h}$ as an FIO associated to the shifting transformation \eqref{tL1}.) Then the symbol of $\rho_R Q_L$ is equal to the symbol of $P_L$, pulled back by $T^{-1}$, and then divided by $\rho_L$; that is, the symbol of $\rho_R Q_L$ is $\tpL/\rho_L$. 

We now compute the explicit form of $Q_L$ in the various regions. 
The crucial point in each case is that $\rho_R Q_L$ is a differential operator with smooth coefficients, despite the division by a power of $\rho_L$.

$\bullet$ In region 2a, the operator $P_L$ takes the form (writing $D_x = - i \partial/\partial x$, etc)
\begin{equation}
 \big(hx D_x\big)^2 +  in h \big(hx D_x \big) - \sum_{i, j = 1}^n g_0^{i j}\big(h x D_{y_i}\big)\big(h x D_{y_j}\big) + hx \bigg( \sum_{k = 1}^n b_k hx \frac{\partial}{\partial y_k} + b_0 hx \frac{\partial}{\partial x} \bigg) - \frac{h^2 n^2}{4} - 1
\label{Pl2a}\end{equation}
for some real coefficient functions $b_j$. 
%This is formally self-adjoint on $L^2(X; dg)$. For the purposes of symbol calculus, it is convenient to work on an $L^2$ space with respect to Lebesgue measure $d\mu = dx dy dz'$ in local coordinates. We write $d\mu = x^{n+1} a^2 dg$, where $a$ is smooth and nonvanishing. Then multiplication by $x^{(n+1)/2} a$ is a unitary transformation from $L^2(dg)$ to $L^2(d\mu)$. We conclude that the operator
%\begin{equation}
%a^{-1} x^{-(n+1)/2} P_L x^{(n+1)/2} a
%\label{Plconj}\end{equation}
%is formally self-adjoint on $L^2(d\mu)$. This remains true if we multiply by $x^{-1/2}$ on each side.  In addition, we can conjugate by a complex function of norm one, as this is also a unitary transformation. Thus we define the formally self-adjoint operator  $Q_L$ (on $L^2(d\mu)$) by 
In this region, we may take $\rho_L = x, \rho_R =  \rho_F = 1$. Thus, $Q_L$ is given by 
$$
Q_L = e^{i\log x/h} a^{-1} x^{-(n+2)/2} P_L x^{n/2} a e^{-i\log x/h}.
$$
It is crucial that this operator $Q_L$ is smooth --- that is, involves no negative powers of $x$ --- so we provide full details in the following  calculation. 

\begin{lemma}\label{lem:Ql} The differential operator $Q_L$ is given in region 2a by
\begin{equation}
Q_L = 
 (hD_x) x (hD_x) - 2 h D_x + 
 x\sum_{i, j = 1}^n (h  D_{y_i}) g_0^{ij} (h  D_{y_j})  + f
\label{Ql}\end{equation}
for some $C^\infty$ real function $f$. 
\end{lemma}

\begin{proof} We first note that conjugation of $P_L$ by the $a$ factor does not change the form of \eqref{Pl2a}; it only changes the coefficients $b_0$ and $b_j$. So without loss of generality, we set $a = 1$. Writing $e^{i\log x/h} = x^{i/h}$, we compute 
\begin{equation}\begin{gathered}
x^{-c+i/h}P_L x^{c - i/h} =  \Bigg[ 
 \bigg(hx D_x -1 -ihc \bigg)^2 +  n h \bigg(ihx D_x -i + hc \bigg) + \sum_{i, j = 1}^n g_0^{i j}(h x D_{y_i})(h x D_{y_j}) \\ + hx \bigg( \sum_{k = 1}^n b_k hx \partial_{y_k} + b_0 \Big( hx \partial_x - i + hc \Big)  \bigg)  - \frac{h^2 n^2}{4} - 1 \Bigg] \\
= (hxD_x)^2 - 2(1 + ihc)(hxD_x) + (1 + ihc)^2 + inh (hx D_x) + nh(-i + hc) + \sum_{i, j = 1}^n g_0^{i j}(h x D_{y_i})(h x D_{y_j}) \\ 
+ hx \bigg( \sum_{k = 1}^n b_k hx \partial_{y_k} + b_0 \Big( hx \partial_x - i + hc \Big)  \bigg)  - \frac{h^2 n^2}{4} - 1  \\
= (hxD_x)^2 - 2 hx D_x + ih (n - 2c) hxD_x + ih(2c - n) - h^2 \big( c- \frac{ n}{2} \big)^2 \\
+ \sum_{i, j = 1}^n g_0^{i j}(h x D_{y_i})(h x D_{y_j}) +
hx \bigg( \sum_{k = 1}^n b_k hx \partial_{y_k} + b_0 \Big( hx \partial_x - i + hc \Big)  \bigg).
\end{gathered}\label{Pl2a-2}\end{equation}
We see from this that three cancellations occur when $c = n/2$, and we can then divide by a factor of $x$, obtaining 
\begin{equation}\begin{gathered}
Q_L = x^{-1} x^{-n/2+i/h}P_L x^{n/2 - i/h} =  
 (hD_x) x (hD_x) - 2 h D_x + 
 x\sum_{i, j = 1}^n (h  D_{y_i}) g_0^{ij} (h  D_{y_j}) \\ +
h \bigg( \sum_{k = 1}^n b_k hx \partial_{y_k} + b_0 \Big( hx \partial_x - i + hc \Big)  \bigg).
\end{gathered}\label{Pl2a-3}\end{equation}
Note that we wrote the term with second $y$-derivatives in divergence form above (which requires a further redefinition of the $b_k$ coefficients). Since $Q_L$ is self-adjoint, and the (redefined) $b_k$ coefficients are real, they must vanish. We see that 
\eqref{Ql} holds 
for some real function $f$. 
\end{proof}

\begin{remark} This is very similar to Vasy's algebraic manipulations of an asymptotically hyperbolic Laplacian in \cite{Vasy}. The difference is that Vasy assumes that the metric is even in $x$. This allows one to divide $Q_L$ by a further factor in $x$ and express in terms of the function $\mu = x^2$, taken to be the boundary defining function for a new differentiable structure on $X$. By analyzing the Hamiltonian flow for this operator on an extension of $X$ into $\mu < 0$, and in particular the flow near the radial sets, Vasy shows analytic continuation of the resolvent without the need of a parametrix construction. However, for our applications in Theorems~\ref{thm:spec-mult} and \ref{thm:Strichartz}, the parametrix construction is indispensable. 
\end{remark}

We have shown that $Q_L$ is a differential operator with smooth coefficients. Thus,  $Q_L$ can be extended as a smooth differential operator across the boundary at $x=0$. Indeed,  all we need to do is to extend the functions $g_0^{ij}$ and the function $f$ in \eqref{Ql} in some smooth manner. 
Moreover, the Hamilton vector field is transverse to $\{ x = 0\}$, as shown in Lemma~\ref{lem:shifted-vf} (of course, this is due to the $-2h D_x$ term in \eqref{Ql}). It follows that $K$ can be extended to a smooth Lagrangian solution to $Q_L K = O(h^\infty)$ through the boundary. In particular, its symbol is a smooth half-density on $\tL_+$. This completes the proof of Proposition~\ref{prop:K} in region 2a. 

In the remaining regions, the structure of the proof is exactly the same, and is related to the calculations of Lemma~\ref{lem:shifted-vf} in exactly the same way.  So we only give brief details in the remaining regions. 

$\bullet$ Region 2b. This works in exactly the same way as Region 2a, using the right Hamilton vector field instead of the left (taking advantage of Lemma~\ref{lem:symm}). 

$\bullet$ Region 3. We first note that this case applies not only to the neighbourhood of a point $q$ lying over a point in $\FL \cap \FR$ and away from $\FF$ in $X^2_0$, but also (in the case where we are working on $\tL_+^*$ --- see Section~\ref{subsec:Lambdastar}) to the neighbourhood of a point $q \in \tL_+^*$ lying over any point in $\FL \cap \FR$  in $X^2$.

In this region, we take $\rho_L = x, \rho_R = x', \rho_F = 1$. So we have to conjugate simultaneously in the left and right variables. That is, we conjugate the operator by $x^{(n+1)/2} {x'}^{(n+1)/2} a a'$. This gives us two operators $Q_L$ and $Q_R$, such that 
$$
Q_L K = O(h^\infty), \quad Q_R K = O(h^\infty).
$$
Both $x' Q_L$ and $x Q_R$ have the form given in Lemma~\ref{lem:Ql}. So they can be extended smoothly to a local extension of $X^2_0$ or ($X^2$) across the boundary near the corner $x = x' = 0$. As we have seen, the Hamilton vector field of $Q_L$ is transverse to $x=0$ and the Hamilton vector field of $Q_R$ is transverse to $x'=0$. Since we have smoothness of the symbol for $x, x' > 0$, this shows smoothness across the boundary. In particular, the symbol of $K$ is a smooth half-density on $\tL_+$.  

$\bullet$ Region 4a. In this region, using coordinates as in Lemma~\ref{lem:shifted-vf} (so $\rho_L = s$, $\rho_R = 1$, $\rho_F = x'$),  we compare the Riemannian density to the coordinate density $d\mu = |ds dx' dy dY|$. Clearly, $dg = x^{(n+1)/2} a d\mu$ for some smooth positive $a$, as in Region 2a. Thus, $P_L$ is given by the same formula \eqref{Plconj}. Of course, conjugating $P_L$ by $x^{(n+1)/2}$ is the same as conjugating by $s^{(n+1)/2}$, since $P_L$ commutes with multiplication by $x'$. Therefore, in this region, 
$$
Q_L = e^{i\log s/h} a^{-1} s^{-(n+2)/2} P_L s^{n/2} a e^{-i\log s/h},
$$
which is self-adjoint with respect to the coordinate Lebesgue measure. Since $P_L$ has the form 
\begin{equation}
- \bigg(hs \frac{\partial}{\partial s}\bigg)^2 +  n h \bigg(hs \frac{\partial}{\partial s}\bigg) - \sum_{i, j = 1}^n g_0^{i j}\bigg(h s \frac{\partial}{\partial y_i}\bigg)\bigg(h s \frac{\partial}{\partial y_j}\bigg)  - \frac{h^2 n^2}{4} - 1 + O(x's),
\label{Pl2a-4}\end{equation}
the calculation looks identical to that in region 2a, with $s$ replacing $x$ (up to an error $O(x's)$). We see that 
\begin{equation}
Q_L = 
 (hD_s) s (hD_s) - 2 h D_s + 
 s\sum_{i, j = 1}^n (h  D_{y_i}) g_0^{ij} (h  D_{y_j})  + f + O(x'). 
\label{Ql-4a}\end{equation}
The rest of the argument proceeds as in region 2a: we have 
$$
Q_L K = - e^{i\log s/h} a^{-1} s^{-(n+2)/2} E_2 + O(h^\infty),
$$
and $Q_L$ extends across the boundary. Since $E_2$ is in $I^{-1/2}(X^2_0, \Lambda_+; \Omegazh)$, supported in a deleted neighbourhood of $\partial_{\diagz} \Lambda_+$, $E_2$ is ${x'}^{-(n+1)/2} |d\mu|^{1/2}$ times a Legendre distribution of order $-1/2$ associated to $\tL_+$ (the factor ${x'}^{-(n+1)/2}$ adjusting for the ratio between the Riemannian half-density and $|d\mu|^{1/2}$). It follows from the standard theory of Lagrangian distributions that  there is a solution $K \in {x'}^{-(n+1)/2} I^{1/2}(X^2_0, \tL_+; \Omega^{1/2})$ that extends across the boundary. It follows that the symbol of $K$ is ${x'}^{-(n+1)/2}$ times a $C^\infty$ half-density on $\tL_+$. 

$\bullet$ Region 4b. This works just as for region 4a, using the right operators $P_R, Q_R$ instead of the left operators. 

$\bullet$ Region 5. In this region we set $d\mu$ to be the coordinate density $d\mu = ds_1 ds_2 dt dZ dy$. It then follows that
$$
dg = \frac{a}{(s_1 s_2 t)^{n+1}} d\mu,
$$
where $a$ is smooth in local coordinates. We therefore define $Q_L$ in this region (acting on functions) by 
$$
Q_L = e^{i\log s_1s_2/h} a^{-1} s_1^{-(n+2)/2} s_2^{-(n+2)/2} t^{-(n+1)/2} P_L^\mu s_1^{n/2} s_2^{n/2} t^{(n+1)/2} a e^{-i\log s_1s_2/h}.
$$
Since $s_1 = x/t$, $s_2 = x'/t$, this is the same as $(s_1 s_2)^{-1}$ times 
$$
  e^{i\log x/h} e^{-2i\log(y'_1 - y_1)/h}  a^{-1} x^{-n/2} (y_1' - y_1)^{(n-1)/2} P_L x^{n/2}  (y'_1 - y_1)^{-(n-1)/2} a e^{-i\log x/h} e^{2i\log(y'_1 - y_1)/h}  .
$$
For the same reason as above, we can neglect the $a$ term. Conjugation by the factor 
$$ 
x^{n/2}  (y'_1 - y_1)^{-(n-1)/2}  e^{-i\log x/h} e^{2i\log(y'_1 - y_1)/h}$$ 
has the effect
$$
hxD_x \mapsto hxDx -1 -\frac{ihn}{2}, \quad hx D_{y_1} \mapsto hxD_{y_1} -2s_1 -\frac{ihs_1(n-1)}{2}.
$$
Then changing to coordinates $(s_1, s_2, t, Z, y)$ has the effect
$$
h x D_x \mapsto h s_1 D_{s_1}, \quad h x D_{y_1} \mapsto h \Big( s_1 D_{s_1} +  s_2 D_{s_2} - t D_t  + Z \cdot D_Z + t D_{y_1}\Big).
$$
It follows that, in the new coordinates,
\begin{equation}\begin{gathered}
s_2 Q_L = s_1^{-1} \bigg[  \bigg(hs_1 D_{s_1} -1 -\frac{ihn}{2} \bigg)^2 +  n h \bigg(ihs_1 D_{s_1} -i - \frac{hn}{2} \bigg) \\ + 
g_0^{11} s_1^2 \Big( h s_1 D_{s_1} +  hs_2 D_{s_2} - ht D_t -2 + hZ \cdot D_Z + ht D_{y_1} - \frac{ih(n-1)}{2} \Big)^2 \\
+2 s_1^2 \sum_{j=2}^n g_0^{1j}(y) \Big(h s_1 D_{s_1} +  hs_2 D_{s_2} - ht D_t -2 + hZ \cdot D_Z + ht D_{y_1}  - \frac{ih(n-1)}{2} \Big)D_{Z_j} \\
+ s_1^2 \sum_{i, j = 2}^n g_0^{ij}(y)  h^2D_{Z_i} D_{Z_j} - \frac{h^2 n^2}{4} - 1 +  O(s_1 t)  \Bigg]  \\
= h D_{s_1} s_1 (h D_{s_1}) - 2 h D_{s_1} + s_1 \Big(h s_1 D_{s_1} +  hs_2 D_{s_2} - ht D_t -2 + hZ \cdot D_Z + ht D_{y_1}  - \frac{ih(n-1)}{2} \Big)^2  \\
+ 2 s_1 \sum_{j=2}^n g_0^{1j}(y) \Big( h s_1 D_{s_1} +  hs_2 D_{s_2} - ht D_t -2 + hZ \cdot D_Z + ht D_{y_1}  - \frac{ih(n-1)}{2} \Big)^2  D_{Z_j}  \\
+ s_1 \sum_{i, j = 2}^n g_0^{ij}(y) h^2D_{Z_i} D_{Z_j} 
+ O(t) + O(h). 
\end{gathered}\label{Ql-5}\end{equation}
In particular, $s_2 Q_L$ is an operator with smooth coefficients. Also, the principal symbol of $s_2 Q_L$ is $\tpL/s_1$, by the same argument as in the other regions above (or just by comparing \eqref{Ql-5} with the calculation in Lemma~\ref{lem:shifted-vf}). Thus, as shown in Lemma~\ref{lem:shifted-vf}, the Hamilton vector field for $s_2 Q_L$ is transverse to $s_1 = 0$ (this is clear from the above form, due to the term $2 h D_{s_1}$). Similarly, the right operator $Q_R$ is such that $s_1 Q_R$  has a smooth symbol and a Hamilton vector field transverse to $s_2 = 0$. We have arranged that $s_2 Q_L K = O(h^\infty)\rho_F^{-(n+1)/2}$ and $s_1 Q_R K = O(h^\infty)\rho_F^{-(n+1)/2}$. Extending $s_2 Q_L$ and $s_1 Q_R$ across the boundaries at $s_1 = 0$ and $s_2 = 0$, we see that $K$ can be extended as an $O(h^\infty)\rho_F^{-(n+1)/2}$ solution. Therefore the symbol of $K$ is $\rho_F^{-(n+1)/2}$ times a Lagrangian distribution on $\tL_+$. This completes the proof of Proposition~\ref{prop:K}. 
\end{proof}

%We now let 
%$$
%G_3' = e^{i \log (\rho_{FL} \rho_{FR})/h} (\rho_{FL}\rho_{FR})^{-1/2} K. 
%$$
%
%\begin{corollary} We have
%\begin{equation}
%G_3' \in (\rho_L \rho_R)^{n/2} I^{1/2}(\Lambda_+; \Omegazh).
%\label{G3'}\end{equation}
%\end{corollary}
%
%\begin{proof} We showed that $K$ was $\rho_F^{-(n+1)/2}$ times an element of $I^{1/2}(X^2_0, \tL_+; \Omega^{1/2})$. Therefore, $G_3'$ is $\rho_F^{-(n+1)/2} (\rho_L\rho_R)^{-1/2}$ times an element of $I^{1/2}(X^2_0, \Lambda_+; \Omega^{1/2})$. Switching back to the Riemannian half-density $|dg dg'|^{1/2}$, which is $(\rho_F\rho_L\rho_R)^{-(n+1)/2}$ times a coordinate half-density on $X^2_0$, we obtain \eqref{G3'}.
%\end{proof}

%%%%%%%%%%%%%%%%%%%%%%%%%%%%%%%%%

\subsection{Infinite decay at front face}\label{subsec:ff}

The idea to solve the front face error away is to reduce to iterative normal operator equations for fixed $h$ and retain the $O(h^\infty)$ vanishing property.
Essentially speaking, we need only borrow Mazzeo and Melrose's normal operator arguments in \cite{Mazzeo-Melrose}. For the sake of completeness, we outline the proof.

So far, we have constructed $G_3 = G_1 + G_2' + G_3'$ such that the error term $E_3$ has no microlocal singularities on $N^* \diagz$ or $\Lambda_+$. We have 
\begin{equation}\begin{gathered}
P_L^g G_3'  =  \bigg( (\rho_L \rho_R)^{1/2} e^{-i\log(\rho_L \rho_R)/h} Q_L^\mu (\rho_L \rho_R)^{-1/2} e^{i\log(\rho_L \rho_R)/h} \bigg) \bigg( (\rho_L \rho_R)^{1/2} e^{-i\log(\rho_L \rho_R)/h} K \bigg) \\
= \rho_L^{1/2} \rho_R^{-1/2} e^{-i\log(\rho_L \rho_R)/h} (\rho_L Q_L^\mu) K \\
= -E_2 + E_3, 
%\ E_3 \in   \rho_L^{1/2-i/h} \rho_R^{-1/2-i/h} \rho_F^{-(n+1)/2} h^\infty C^\infty(X^2_0; \Omega^{1/2}) + x^{1/2-i/h} {x'}^{-1/2-i/h} h^\infty C^\infty(X^2; \Omega^{1/2}) \\
%= -E_2 + E_3, \ E_3 \in   \rho_L^{n/2+1-i/h} \rho_R^{n/2-i/h}h^\infty C^\infty(X^2_0; \Omegazh) + x^{n/2+1-i/h} {x'}^{n/2-i/h} h^\infty C^\infty(X^2; \Omegazh) . 
\end{gathered}\end{equation}
where $E_3$ is in the space 
\begin{equation}\begin{gathered}
 \rho_L^{1/2-i/h} \rho_R^{-1/2-i/h} \rho_F^{-(n+1)/2} h^\infty C^\infty(X^2_0; \Omega^{1/2}) + x^{1/2-i/h} {x'}^{-1/2-i/h} h^\infty C^\infty(X^2; \Omega^{1/2}) \\
=   \rho_L^{n/2+1-i/h} \rho_R^{n/2-i/h}h^\infty C^\infty(X^2_0; \Omegazh) + x^{n/2+1-i/h} {x'}^{n/2-i/h} h^\infty C^\infty(X^2; \Omegazh) . 
\end{gathered}\end{equation}

Our next task is to remove errors at the front face (in the first term in the expression for $E_3$ above), up to $O(\rho_F^\infty)$ errors. To do this, we look for a correction term $G_4'$ of the form 
$$
G_4' = \sum_{j=0}^\infty \rho_F^j G_{4, j},
$$
that solves away the error at $\FF$ order by order. We want $G_4'$ to be $O(h^\infty)$, so as not to disturb the fact that we already have an error term that is $O(h^\infty)$. 

The first term, $G_{4, 0}$, must satisfy
\begin{equation}
N_{y_0}(P) N_{y_0}(G_{4, 0}) = -N_{y_0} (E_3) \in \rho_{L}^{n/2 + 1 -i/h} \rho_R^{n/2 -i/h} h^\infty C^\infty(F_{y_0}; \Omega_0).
\end{equation}

Using Proposition \ref{normal operator}, we see that there is a solution with 
$$
G_{4, 0} \in \rho_{L}^{n/2 + 1 -i/h} \rho_R^{n/2 -i/h} h^\infty C^\infty(X^2_0; \Omegazh).
$$
The error term is now reduced to 
$$
E_{3, 1} \in \rho_F \rho_{L}^{n/2 + 1 -i/h} \rho_R^{n/2 -i/h}  h^\infty C^\infty(X^2_0; \Omegazh) = x' \rho_{L}^{n/2 + 1 -i/h} \rho_R^{n/2-1 -i/h}  h^\infty C^\infty(X^2_0; \Omegazh) .
$$
Notice that $x'$ commutes with $P_L$. So to solve this term away to leading order at $\FF$, we take $G_{4, 1}$ such that 
\begin{equation}
N_{y_0}(P) N_{y_0}(G_{4, 0}) = -N_{y_0} (E_{3, 1}/x') \in \rho_{L}^{n/2 + 1 -i/h} \rho_R^{n/2 -1 -i/h} h^\infty C^\infty(F_{y_0}; \Omega_0).
\end{equation}
Proposition \ref{normal operator} guarantees that there is a solution with 
$$
G_{4, 1} \in \rho_{L}^{n/2 + 1 -i/h} \rho_R^{n/2-1 -i/h} h^\infty C^\infty(X^2_0; \Omegazh).
$$
This reduces the error to 
$$
E_{3, 2} \in x' \rho_F \rho_{L}^{n/2 + 1 -i/h} \rho_R^{n/2-1 -i/h}  h^\infty C^\infty(X^2_0; \Omegazh),
$$
which can either be viewed as an element of 
$${x'}^2 \rho_{L}^{n/2 + 1 -i/h} \rho_R^{n/2-2 -i/h}  h^\infty C^\infty(X^2_0; \Omegazh) .
$$
or  
$$
\rho_F^2 \rho_{L}^{n/2 + 1 -i/h} \rho_R^{n/2 -i/h}  h^\infty C^\infty(X^2_0; \Omegazh). 
$$
We proceed in this way, and take an asymptotic summation to obtain the desired correction term $G_4'$.  Setting $G_4 = G_3 + G_4'$, we find that the new error term
$
E_4 := P_L G_4 - \delta
$
satisfies
$$
E_4 \in \rho_{L}^{n/2 + 1 -i/h} \rho_F^{\infty} \rho_R^{n/2 -i/h} h^\infty C^\infty(X^2_0; \Omegazh) .
$$

%%%%%%%%%%%%%%%%%%%%%%%%%%%%%%%%%%%%%%%

\subsection{Left boundary behaviour}

Having removed the error term (up to $O(\rho_F^\infty)$) at the front fact, we may view $E_4$ as living on $X^2$ instead of $X^2_0$; then we have
$$
E_4 \in \rho_{L}^{n/2 + 1 -i/h}  \rho_R^{n/2 -i/h} h^\infty C^\infty(X^2_0; \Omegazh) .
$$
We next remove the error terms at $\FL$. This is a straightforward Taylor series calculation at the left boundary. 

%We write 
%$$
%E_4 = \sum_{j=0}^\infty x^{n/2 + 1 -i/h + j} E_{4, j}. 
%$$
We wish to solve this error away, up to $O(x^\infty)$, at $x=0$, with a correction term of the form 
$$
G_5 = \sum_{j=0}^\infty  x^{n/2 +1  -i/h + j} G_{5, j}(y, z') |dg dg'|^{1/2}.
$$
The left operator takes the form (reverting to the connection annihilating $|dg dg'|^{1/2}$) 
$$
h^{-2} P_L = (x D_x)^2 + n x \partial_x + x^2 R - \big(h^{-2} + \frac{n^2}{4} \big),
$$
where $R$ is a b-differential operator, that is, a combination of $x D_x$ and $D_y$ with smooth coefficients. Let the leading part of $E_4$ at $x=0$ be $x^{n/2 + 1 -i/h} E_{4, 0}(y, z')$. To solve this away modulo $O(x^{n/2 + 2 -i/h})$, we require that
$$
- \big(\frac{n}{2} + 1 + \frac{i}{h} \big)^2 + n \big(\frac{n}{2} + 1 + \frac{i}{h} \big) - \big(h^{-2} + \frac{n^2}{4} \big) G_{5, 0} = E_{4, 0}.
$$
This simply requires that 
$$
G_{5, 0} = -(1 - \frac{2i}{h})^{-1} E_{4, 0}.
$$
Notice that $E_{4, 0} = O(h^\infty)$ implies that also $G_{5, 0} = O(h^\infty)$. 

Inductively, suppose that the error at $x=0$ has been reduced to $E_{4, j} \in x^{n/2 + 1 -i/h + j} C^\infty$. Then we choose $G_{5, j}$ such that 
$$
- \big(\frac{n}{2} + 1 + j + \frac{i}{h} \big)^2 + n \big(\frac{n}{2} + 1 +j + \frac{i}{h} \big) - \big(h^{-2} + \frac{n^2}{4} \big) G_{5, j} = E_{4, j},
$$
that is, 
$$
G_{5, 0} = -\Big((1+j)^2  - \frac{2i}{h}(1+j)\Big)^{-1} E_{4, 0}.
$$
Again, $E_{4, j} = O(h^\infty)$ implies that also $G_{5, j} = O(h^\infty)$. Let $G_5'$ be an  asymptotic sum  of the $G_{5, j}$. Then $G_5 = G_4 + G_5'$ solves the equation up to an error that is in $x^\infty h^\infty {x'}^{n/2-i/h} C^\infty(X^2; \Omegazh)$.

\subsection{Resolvent from parametrix}

We have found a left parametrix $G_5$ so that the error term, $E_5 = P G_5 - \Id$, is a kernel in $x^N h^N {x'}^{n/2-i/h} C^\infty(X^2, \Omegazh)$ for every $N$. We want to invert $\Id + E_5$; if this is possible, then we have constructed the resolvent kernel in the form $G_5(\Id + E_5)^{-1}$. 

We first show that for small $h$, $\Id + E_5$ is invertible on a weighted $L^2$ space; equivalently, we can conjugate $E_5$ by a power of $x$ so that it becomes invertible on $L^2$ (with respect to the Riemannian measure $dg$, of course). Let $E_5^c$ denote $x^{-1}  E_5  x$ (we write this in operator notation; in kernel notation, it is $x^{-1} E_5 x'$), where the $c$ indicates conjugation. This kernel is in $x^N h^N {x'}^{n/2+1-i/h} C^\infty(X^2, \Omegazh)$ for every $N$, and is therefore in $L^2(X^2)$, since $x^{n/2+1}$ is square-integrable with respect to the Riemannian density which is a smooth multiple of $x^{-(n+1)} dx dy$ near $x=0$. Moreover, its $L^2$ norm is $O(h^N)$ for every $N$. Therefore, $E_5^c$ is a Hilbert-Schmidt kernel with small Hilbert-Schmidt norm for small $h$. Therefore $\Id + E_5^c$ is invertible for sufficiently small $h$; moreover, the inverse can be written in the form $\Id + S^c$, where $S^c$ is Hilbert-Schmidt   with Hilbert-Schmidt norm $O(h^N)$, as follows from expressing $S^c$ as a  Neumann series. 

Using a standard argument, we show that in fact, $S^c$ has the form $x^N h^N {x'}^{n/2+1 -i/h} C^\infty(X^2, \Omegazh)$ for every $N$. To see this, we use the equations 
$$
(\Id + E_5^c)(\Id + S^c) = \Id = (\Id + S^c)(\Id + E_5^c)
$$
to write
$$
S^c = - E_5^c - E_5^cS^c = - E_5^c - E_5^c(- E_5^c - S^cE_5^c) = -E_5^c + (E_5^c)^2 + E_5^cS^cE_5^c. 
$$
It is easy to check that $ -E_5^c + (E_5^c)^2$ has the claimed form. As for the final term $E_5^cS^cE_5^c$, we express this as an integral:
$$
E_5^cS^cE_5^c(z, z') = \int_X \int_X E_5^c(z, z'') S^c(z'', z''') E_5^c(z''', z') \; dg(z'') \, dg(z'''). 
$$
We write the left factor of $E_5^c = x^N h^N A(z, z'')$ where $A(z, z'')$ is $C^\infty$ in $z$ with values in $L^2(X_{z''})$, and the right factor of $E_5^c$ in the form ${x'}^{n/2 + 1 -i/h} B(z''', z')$ where $B(z''', z')$ is an $L^2$ function of $z'''$ with values in $C^\infty(X_{z'})$. Then it is evident that the integral expression for $E_5^cS^cE_5^c$ has the claimed form. 

We finally observe that $(\Id + E_5)^{-1} = \Id + x S^c x^{-1}$, so we can write the inverse in the form $\Id + S$ where $S = x S^c x^{-1}$ has the form $x^N h^N {x'}^{n/2 -i/h} C^\infty(X^2, \Omegazh)$ for every $N$. The final step is to express the true resolvent by
$$
R_h = G_5 + G_5 S.
$$
Thus it remains to determine the nature of the kernel $G_5 S$.

\begin{lemma}
The kernel $G_5 S$ is in the space
$$
x^{n/2-i/h} {x'}^{n/2-i/h} h^\infty C^\infty(X^2 \times [0, h_0]) |dg dg'|^{1/2}.
$$
\end{lemma}

\begin{proof} 
To prove this, we can view the composition as a pushforward. More precisely, we consider the map $\Upsilon: X^2_0 \times X \times [0, h_0]\to X^2\times [0, h_0]$, which is the composition of a blowdown map $X^2_0(z, z'') \times X(z') \times [0, h_0]\to X^3(z, z'', z')\times [0, h_0]$ followed by projection $X^3(z, z'', z')\times [0, h_0] \to X^2(z, z')\times [0, h_0]$ (we indicate the coordinate variables, valid at least in the interiors of these spaces, to indicate how the maps operator on the various factors of $X$). Then the kernel of the composition $G_5 S$ can be realized by 
\begin{itemize}
\item 
lifting $G_5$ to $X^2_0(z, z'') \times X(z') \times [0, h_0]$ via the projection to the left factor, 
\item 
lifting $S$ to $X^2_0(z, z'') \times X(z')\times [0, h_0]$ by first lifting to $X^3\times [0, h_0]$ by the projection $X^3(z, z'', z') \to X^2(z'', z')$ and then lifting to $X^2_0 \times X \times [0, h_0]$ by the blowdown map to $X^2_0 \times X \to X^2 \times X \equiv X^3$,
\item 
multiplying these kernels together, and
\item
pushing forward by $\Upsilon$. 
\end{itemize}
We remark that the product of the two half-density factors gives a full density in the $z''$ variable, which can be pushed forward invariantly by the map $\Upsilon$. 

The space $X^2_0(z, z'') \times X(z') \times [0, h_0]$ has five boundary hypersurfaces, which we will denote $\FL, \FR$ and $\FF$, arising from the $X^2_0$ factor, $\mathrm{FX}$, arising from the $X(z')$ factor, and $\mathrm{FH}$, at $\{ h = 0 \}$. Notice that the lift of $S$ vanishes to infinite order at $\FF$, $\FR$ and $\mathrm{FH}$. It follows that the product $G_5 S$ has the form $x^{n/2-i/h} {x'}^{n/2-i/h} h^\infty \rho_F^\infty \rho_R^\infty B \, |dg|^{1/2} |dg''| |dg'|^{1/2}$, where $B$ is smooth on $X^2_0\times X$. 
Due to the rapid vanishing at $\FF$ and $\FR$, the pushforward is well-defined (in the sense that the integral converges) and the result has the form $x^{n/2-i/h} {x'}^{n/2-i/h} h^\infty b  \,  |dg|^{1/2}  |dg'|^{1/2}$, where $b$ is smooth on $X^2$. This completes the proof.
\end{proof}

\subsection{Summary}
We have shown that, for sufficiently small $h$,  the outgoing resolvent $(h^2 \Delta - h^2n^2/4 - (1 + i0))^{-1}$ can be expressed as $G_1 + G'_2 + G'_3 + G'_4 + G'_5 + G_5 S$, where
\begin{itemize}
\item $G_1 \in {}^0 \Psi^{-2, 0}(X, \Omegazh)$, that is, a semiclassical 0-pseudodifferential operator of semiclassical order $0$ and differential order $-2$;
\item $G_2' \in (\rho_L \rho_R)^{n/2}I^{1/2}(N^* \diagz, \Lambda_+^{nd}; \Omegazh)$, that is, a semiclassical Lagrangian distribution of order $1/2$ associated to $N^* \diagz$ and to the Lagrangian submanifold $\Lambda_+^{nd}$;
\item $G'_3 \in (\rho_L \rho_R)^{n/2} I^{1/2}(X^2_0, \Lambda_+^{nd}; \Omegazh) + (xx')^{n/2 -i/h} I^{1/2}(X^2, \tL_+^*; \Omegazh)$, that is, a semiclassical Lagrangian distribution of order $1/2$ associated to $\Lambda_+$;
\item $G_4' \in  \rho_{L}^{n/2 + 1 -i/h} \rho_R^{n/2 -i/h} h^\infty C^\infty(X^2_0\times [0, h_0]; \Omegazh)$, 
\item $G_5' \in x^{n/2+1 -i/h} {x'}^{n/2 -i/h} h^\infty C^\infty(X^2\times [0, h_0]; \Omegazh)$, and
\item $G_5 S \in  x^{n/2-i/h} {x'}^{n/2-i/h} h^\infty C^\infty(X^2 \times [0, h_0]; \Omegazh)$.
\end{itemize}
We notice that the $G_4$ term (which can be taken to have support in a small neighbourhood of $\FF$) can be regarded as an element of $(\rho_L \rho_R)^{n/2} I^{-\infty}(X^2_0, \Lambda_+^{nd}; \Omegazh)$. Also, the $G_5'$ and $G_5 S$ terms can be combined as an  element of $x^{n/2-i/h} {x'}^{n/2-i/h} h^\infty C^\infty(X^2 \times [0, h_0]; \Omegazh)$. 
Moreover, the summand of $G_3'$ lying in $(\rho_L \rho_R)^{n/2} I^{1/2}(X^2_0, \Lambda_+^{nd}; \Omegazh)$ may be regarded as an element of $(\rho_L \rho_R)^{n/2}I^{1/2}(N^* \diagz, \Lambda_+^{nd}; \Omegazh)$. Collecting terms in this way, we can express our result as follows.

\begin{theorem}\label{asymptotically hyperbolic resolvent} The semiclassical operator $P_h = h^2 \Delta - h^2 n^2 /4 - 1$ is inverted by an operator that is the sum of the following terms:
\begin{itemize}
\item an element of ${}^0\Psi^{-2, 0}_h(X)$, that is, a semiclassical $0$-pseudodifferential operator of differential order $-2$ and semiclassical order $0$; 
\item A semiclassical intersecting Lagrangian distribution 
in  $(\rho_L \rho_R)^{n/2} I^{1/2}(N^* \diagz, \Lambda_+^{nd}; \Omegazh)$, where $N^* \diagz$ is the conormal bundle of the $0$-diagonal in $X^2_0$, and $\Lambda_+^{nd}$ is a subset of the Lagrangian $\Lambda_+$, generated by bicharacteristic flowout from the intersection of $N^* \diagz$ and the zero set of the symbol of $P_h$; 
\item a kernel lying in $(x x')^{n/2-i/h} I^{1/2}(X^2, \tL_+^*, \Omegazh)$, also associated to the bicharacteristic flowout, as above, but living on $X^2$ rather than $X^2_0$;
\item an element of 
$(xx')^{n/2-i/h} h^\infty C^\infty(X^2\times [0, h_0]; \Omegazh)$.
\end{itemize}
 \end{theorem} 

We finally justify our claim that we have constructed $R(n/2 - i/h)$. Denote the operator constructed above by $R_h$. Consider the difference, $R(n/2 - i/h) -  R_h$. This satisfies the homogeneous equation $P_h(R(n/2 - i/h) - R_h) = 0$. Moreover, the regularity properties of $R(n/2 - i/h)$ (see \eqref{MM})  and $R_h$ show that, for each fixed $z' \in X^\circ$, the Schwartz kernel of the difference has an expansion 
$$
(R(n/2 - i/h) - R_h)(z, z') =  x^{n/2 - i/h} f_{z'}(z), \quad f_{z'} \in  C^\infty(X).
$$
Using \cite[Proposition 3.2]{Graham-Zworski}, we see that the restriction of $f_{z'}$ to $\partial X$ vanishes. Then, following the argument in \cite[proof of Proposition 3.4]{Graham-Zworski}, we see that $f_{z'}$ vanishes identically. Indeed, a Taylor series analysis of $f_{z'}$ at $\partial X$ shows that $f_{z'}$ vanishes to all orders there. But if $f_{z'}$ did not vanish identically, then $x^{n/2 - i/h} f_{z'}$ would be an $L^2$ eigenfunction for $P_h$ with eigenvalue $n^2/4 + h^{-2}$, which is impossible \cite{Mazzeo-1991}. It follows that $x^{n/2 - i/h} f_{z'}$ vanishes identically, and therefore, $R(n/2 -i/h) = R_h$. 

%\begin{theorem}\label{asymptotically hyperbolic resolvent}
%
%
%
%On non-trapping asymptotically hyperbolic space, the semiclassical operator $P_h = h^2 \Delta - h^2 n^2 /4 - 1$ is inverted by an operator  $ A_1 + A_2,$ where
%
%\begin{enumerate}
%\item  the $A_1 \in \Psi^{- 2}_0$ is a $0$-pseudodifferential operators with semiclassical order $0$ and differential order $- 2$,
%
%\item  the $A_2 \in I_{1/2}(N^* \diagz, \Lambda_+)$ is differentially smoothing and semiclassically $1/2$-order Lagrangian distributions,
%
%\item the principal symbol of $A_2$ is of $ \rho_L^{n/2 } \rho_R^{n/2 } C^\infty $ near boundary provided  boundary defining functions $\rho_L, \rho_R$.
%
%\end{enumerate}
%
%\end{theorem}
%

\begin{appendix}

\section{Semiclassical Lagrangian distributions}\label{app:sld}

In this section, we shall investigate the basic semiclassical analysis of Fourier integral operators and Lagrangian distributions. The ideas in this appendix are (minor) variations of ideas introduced by  H\"{o}rmander \cite{Hormander-Acta-1971} and Duistermaat and H\"{o}rmander \cite{Duistermaat-Hormander-Acta-1972}. Alternatively different expositions appear in \cite{Hormander4}, \cite{Duistermaat}, \cite{sogge} and \cite{Egorov}.

\begin{remark}Following the convention of H\"{o}rmander, we denote the dimension of manifold by $d$ in the appendix, rather than $n + 1$, which denotes the dimension of asymptotically hyperbolic space.\end{remark}

	\subsection{Lagrangian distributions}
	Suppose $X$ is an $d$-dimensional manifold, and $\Lambda$, associated with a non-degenerate phase function $\phi$ defined on $X \times \mathbb{R}^N$, is a Lagrangian submanifold of $(T^\ast X, dx \wedge d\xi)$. A phase function $\phi(x, \theta)$, $x \in U \subset X$, $\theta \in \RR^N$ is a local nondegenerate parametrization of $\Lambda$ near $(x_0, \xi_0) \in \Lambda$ if 
	\begin{itemize}
	\item the differentials $d(\partial \phi/\partial \theta_i)$ are linearly independent whenever $d_\theta \phi = 0$, $i = 1 \dots N$; and 
	\item the map from $C_\phi = \{(x, \theta): \phi^\prime_\theta = 0\}$ to $T^* X$ given by 
	$$
	C \ni (x, \theta) \mapsto (x, d_x \phi) \in T^* X
	$$
	is a local diffeomorphism from $C$ to a neighbourhood of $(x_0, \xi_0)$ in $\Lambda$.
	(Notice that $C$ is a submanifold of dimension $n$ of $U \times \RR^N$ as a consequence of the first condition.) 
	\end{itemize}

	Then we define the space $I^k(X, \Lambda, \Omega^{1/2}) \subset \mathcal{D}^\prime(\mathbb{R}^d, \Omega^{1/2})$ of half-density Lagrangian distributions associated to $\Lambda$ as follows.  Let $J$ be an index set, such that for any $j \in J$ \begin{enumerate}\renewcommand{\labelenumi}{$($\roman{enumi}$)$}
	\item there is a local coordinate patch $X^\prime_j$ of $X$ with local coordinates $(x_1, \dots, x_d) \in \mathbb{R}^d$; 
	\item there is a positive integer $N_j$ and a non-degenerate phase function $\phi_j$ defined in an open subset $U_j$ of $X_j^\prime \times \mathbb{R}^N $ such that the map from $C_j \cap U_j$ to $\Lambda$, $$(x, \theta) \longmapsto (x, \phi^\prime_x)$$ is a diffeomorphism on an open subset $U_j^\Lambda$ of $\Lambda$.\end{enumerate}
	
	\begin{definition}The Lagrangian distribution $A \in I_k(X, \Lambda, \Omega^{1/2})$ is a locally finite sum of $A_j$ in $J$ with $$\langle A_j, u \rangle = (2\pi h)^{- k - (d + 2N_j)/4} \int_{\mathbb{R}^{N_j}}\int_{X_j^\prime} e^{i\phi_j(x, \theta) / h} a_j(x, \theta, h) u(x) \,dx d\theta,$$ where  $a_j \in S(U_j)$ is compactly supported in $\theta$, and where $u = u(x) |dx|^{1/2}$ is a smooth half-density. \end{definition}
	
	We remark that $a_j \in S(U_j)$ means that $a_j$ is smooth in $U_j$, with uniform bounds on all its derivatives as $h \to 0$. 
	
	%\begin{remark}This result allows us to restrict symbols to the Lagrangian.\end{remark}
	%
	
	\vskip 5pt 
	
	Every Lagrangian submanifold $\Lambda \subset T^* X$ can be locally parametrized. 
	Given local coordinates $x$ on $X$, and dual fibre coordinates $\xi$, we can decompose $x = (x', x'')$ and correspondingly $\xi = (\xi', \xi'')$ such that $(x', \xi'')$ locally furnish coordinates on $\Lambda$. Since $d(\xi \cdot dx) = 0$ on $\Lambda$,  Poincar\'{e}'s lemma gives a smooth function $f(x^\prime, \xi^{\prime\prime})$ such that $df = \xi \cdot dx$. We assert 
	
	\begin{lemma}\label{new parametrization}  The phase function $\Phi$ defined by  
	$$\Phi(x, \xi^{\prime\prime}) = \langle x^{\prime\prime}-X^{\prime\prime}(x^\prime, \xi^{\prime\prime}), \, \xi^{\prime\prime} \rangle + f(x^\prime, \xi^{\prime\prime})$$
	locally parametrizes $\Lambda$. 
	\end{lemma}
	
	\begin{proof}
	We must justify that $\Lambda$ is given locally by $\{(x, \Phi^\prime_x) : \Phi^\prime_{\xi^{\prime\prime}} = 0\}$. Restriction $\Phi^\prime_{\xi^{\prime\prime}} = 0$ amounts to \begin{equation}\label{restriction} x^{\prime\prime} - X^{\prime\prime} -  \xi^{\prime\prime} \frac{\partial X^{\prime\prime}}{\partial \xi^{\prime\prime}}  + \frac{\partial f}{\partial \xi^{\prime\prime}} = 0.\end{equation}
	
	The differential of $f$ on $\Lambda$ is \begin{eqnarray*}
	df = \Xi^\prime dx^\prime + \xi^{\prime\prime} dX^{\prime\prime} = \bigg(\Xi^\prime + \xi^{\prime\prime}  \frac{\partial X^{\prime\prime}}{\partial x^\prime}\bigg) dx^\prime + \xi^{\prime\prime} \frac{\partial X^{\prime\prime}}{\partial \xi^{\prime\prime}} d\xi^{\prime\prime},\end{eqnarray*} which gives 
	\begin{equation}\label{restriction2}\frac{\partial f}{\partial x^\prime} = \Xi^\prime + \xi^{\prime\prime}  \frac{\partial X^{\prime\prime}}{\partial x^\prime} \quad \text{and} \quad \frac{\partial f}{\partial \xi^{\prime\prime}} = \xi^{\prime\prime} \frac{\partial X^{\prime\prime}}{\partial \xi^{\prime\prime}}.
	\end{equation}
	
	Combining it with (\ref{restriction}), we get $$x^{\prime\prime} = X^{\prime\prime},\quad \quad \Phi_x^\prime = (\Phi_{x^{\prime}}^\prime, \Phi_{x^{\prime\prime}}^\prime) = \bigg(- \xi^{\prime\prime} \frac{\partial X^{\prime\prime}}{\partial x^\prime} + \frac{\partial f}{\partial x^\prime},\, \xi^{\prime\prime}\bigg) = (\Xi^\prime, \xi^{\prime\prime}) \quad\quad \text{on $\{\Phi^\prime_{\xi^{\prime\prime}} = 0\}$},$$ which completes the proof.
	\end{proof}

	\vskip 5pt 
	
	It is legitimate to expect the  semiclassical wave front set of a Lagrangian distribution is $$ \Lambda = \{(x, \phi^\prime_x ) | \phi_\theta^\prime = 0\}.$$ Indeed, the semiclassical Fourier transform of Lagrangian distribution $$(2 \pi h)^{-(d + 2N)/4} \int e^{i(\phi(x, \theta) - \langle x, \xi \rangle) / h} a(x, \theta, h) \, d\theta dx$$ is rapidly decreasing when $\xi$ is off an open neighbourhood of $\{\phi^\prime_x | \phi^\prime_\theta (x, \theta) = 0\}$, which amounts to showing that for such $\xi$, $$\int e^{i(\phi(x, \theta) - \langle x, \xi \rangle)/h} a(x, \theta, h) \, d\theta dx = O(h^\infty).$$ We work on the new phase, say $\Phi = \phi(x, \theta) - \langle x, \xi \rangle$, then $$|\nabla_{(x, \theta)} \Phi| \approx |\phi_x^\prime - \xi| + |\phi^\prime_\theta|.$$ If $|\phi_\theta^\prime| = 0$ and $\text{dist} \, (\xi, \phi^\prime_x|_{\Lambda_\phi}) > 0$, then $\Phi^\prime_x \neq 0$. Hence the desired estimate follows as $\Phi$ is a non-stationary phase.

	\subsection{Half densities}
	Given a nondegenerate phase function $\phi(x, \theta)$ locally parametrizing a Lagrangian submanifold $\Lambda$, we define a density $d_C$ on $$C = \{(x, \theta) : \phi^\prime_\theta(x, \theta) = 0\} $$
	as follows. Let $\lambda_1, \cdots, \lambda_d$ be local coordinates on $C$ extended to a neighbourhood of $C$. Then we define 
	$$
	d_C = 
	\big| d\lambda_1 \cdots d\lambda_d \big|  \bigg| \frac{\partial (\lambda, \phi^\prime_\theta)}{\partial (x, \theta)} \bigg|^{-1},$$
	which is clearly independent of the choice of $\lambda$. 
	
	We shall study the invariance of half densities under change of phase function, $$\tilde{\phi}(\tilde{x}, \tilde{\theta}) = \phi(x, \theta)$$ where we perform change of variables $x = x(\tilde{x}), \theta = \theta(\tilde{x}, \tilde{\theta}).$ To make $\langle \tilde{A}, \tilde{u} \rangle = \langle A, u \rangle$, we write $$\tilde{u}(\tilde{x}) = \bigg|\frac{\partial x}{\partial \tilde{x}}\bigg|^{1/2} u(x) \quad\mbox{and}\quad \tilde{a}(\tilde{x}, \tilde{\theta}, h) = a(x(\tilde{x}), \theta(\tilde{x}, \tilde{\theta}), h) \bigg|\frac{\partial x}{\partial \tilde{x}}\bigg|^{1/2}\bigg|\frac{\partial \theta}{\partial \tilde{\theta}}\bigg|.$$
	
	We claim that the pushforward of $a\sqrt{d_C}$ under the map from $C$ to $\Lambda$ (which we still denote $a\sqrt{d_C}$) is invariant under changes of phase function  for the same Lagrangian:
	\begin{equation}
	a \sqrt{d_C} = \tilde a \sqrt{d_{\tilde C}}.
	\label{adC}\end{equation}
	To prove this, we must show that 
	$$\bigg|\frac{\partial (\lambda, \phi^\prime_\theta)}{\partial (x, \theta)}\bigg|^{-1/2} a(x, \theta, h) = \bigg|\frac{\partial ({\lambda}, \tilde{\phi}^\prime_{\tilde{\theta}})}{\partial (\tilde{x}, \tilde{\theta})}\bigg|^{-1/2} \tilde{a}(\tilde{x}, \tilde{\theta}, h),$$ 
	
	Indeed, using facts $\tilde{\phi}^\prime_{\tilde{\theta}} = \phi^\prime_\theta \frac{\partial \theta}{\partial \tilde{\theta}} $ and $ \phi^\prime_\theta = 0 \,\mbox{on}\, C,$ we have \begin{eqnarray*}\bigg|\frac{\partial ({\lambda}, \tilde{\phi}^\prime_{\tilde{\theta}})}{\partial (\tilde{x}, \tilde{\theta})}\bigg| & = &\left|\begin{array}{cc}\frac{\partial x}{\partial \tilde{x}}&\frac{\partial \theta}{\partial \tilde{x}}\\\frac{\partial x}{\partial \tilde{\theta}}&\frac{\partial \theta}{\partial \tilde{\theta}}\end{array}\right| \left|\begin{array}{cc}\frac{\partial \lambda}{\partial x}&\frac{\partial \tilde{\phi}^\prime_{\tilde{\theta}}}{\partial x}\\\frac{\partial \lambda}{\partial \theta}&\frac{\partial \tilde{\phi}^\prime_{\tilde{\theta}}}{\partial \theta}\end{array}\right|\\& = &\left|\begin{array}{cc}\frac{\partial x}{\partial \tilde{x}}&\frac{\partial \theta}{\partial \tilde{x}}\\0&\frac{\partial \theta}{\partial \tilde{\theta}}\end{array}\right| \left|\begin{array}{cc}\frac{\partial \lambda}{\partial x}&\sum_{k = 1}^N \phi^{\prime\prime}_{\theta_kx}\frac{\partial \theta_k}{\partial \tilde{\theta}} + \phi^\prime_{\theta_k}\frac{\partial^2 \theta_k}{\partial \tilde{\theta} \partial x}\\\frac{\partial \lambda}{\partial \theta}&\sum_{k = 1}^N \phi^{\prime\prime}_{\theta_k\theta}\frac{\partial \theta_k}{\partial \tilde{\theta}} + \phi^\prime_{\theta_k}\frac{\partial^2 \theta_k}{\partial \tilde{\theta} \partial \theta}\end{array}\right|\\ & = & \bigg|\frac{\partial\theta}{\partial\tilde{\theta}}\bigg|^2 \bigg|\frac{\partial x}{\partial \tilde{x}}\bigg| \bigg|\frac{\partial (\lambda, \phi^\prime_\theta)}{\partial (x, \theta)}\bigg|,\end{eqnarray*} which  combined with the definition of $\tilde{a}$ proves the assertion.

	\subsection{Equivalence of phase functions}
	We say that two phase functions $\phi$ and $\tilde \phi$ locally parametrizing the Lagrangian submanifold $\Lambda \subset T^* X$ are (locally) equivalent if there is a change of variables  $\theta = \theta(x, \tilde \theta)$ such that, locally, 
	$$
	\tilde\phi(x, \tilde \theta) = \phi(x, \theta(x, \tilde \theta)).
	$$
	It follows then that Lagrangian distributions written with phase function $\phi$ may equally well be written with phase function $\tilde \phi$. 
	
	\begin{proposition}\label{prop:equivalence}
	Suppose that $\phi$ and $\tilde \phi$ are two phase functions, defined in a neighbourhood of $(x_0, \theta_0) \in U \times \RR^N$ and $(x_0, \tilde\theta_0) \in U \times \RR^{\tilde N}$ parametrizing the Lagrangian submanifold $\Lambda \subset T^* X$ locally near $(x_0, \xi_0) \in \Lambda$. Then they are equivalent if and only if
	
	(i) $N = \tilde N$, and
	
	(ii) $\phi''_{\theta \theta}(x_0, \theta_0)$ and $\tilde \phi''_{\tilde \theta \tilde \theta}(x_0, \tilde \theta_0)$ have the same signature.
	\end{proposition}
	
	We refer to \cite[Section 3.1]{Hormander-Acta-1971} for the proof, in which homogeneity plays no particular role. 
	
	We now want to define an invariant semiclassical principal symbol for a Lagrangian distribution $A \in I^k(X, \Lambda; \Omega^{1/2})$. To do this we need the following 
	
	\begin{lemma}\label{lem:Nrank}
	Let $\phi$ be a non-degenerate phase function in a neighbourhood of $(x_0, \theta_0)$ in $X \times \mathbb{R}^N$ with $\phi^\prime_\theta (x_0, \theta_0) = 0$ and $\xi_0 = \phi^\prime_x(x_0, \theta_0)$. Then we have $$N - \text{rank}\, \phi^{\prime\prime}_{\theta\theta}(x_0, \theta_0) = d - \text{rank}\, d\pi_\Lambda(x_0, \xi_0),$$ where $\pi_\Lambda$ is the restriction to $\Lambda$ of the projection $T^\ast X \rightarrow X$.
	\end{lemma}
	
	\begin{proof} Consider the maps 
	$$C \ni (x, \theta) \longrightarrow (x, \phi^\prime_x) \longrightarrow x \in X.$$
	The right hand side is the dimension of the kernel of the differential of the right hand arrow, where $(x, \phi^\prime_x)$ is  restricted to the $d$-dimensional manifold $\Lambda = \{(x, \phi^\prime_x) : \phi^\prime_\theta = 0\}$. On the other hand, this dimension is the same as the dimension of the kernel of the differential of the composite map, since the first arrow is a diffeomorphism. This kernel is 
	 the set of tangent vectors $v = a \cdot \partial_x + b \cdot \partial_\theta$ such that $ d\phi^\prime_\theta(v) = 0$ and $dx(v) = 0$. This requires $a = 0$ and $b$ is in the kernel of  the matrix $\phi^{\prime\prime}_{\theta\theta}$, whose dimension is the left hand side.
	\end{proof}
	
	\subsection{Principal symbol}
	In general, if we have a $k$th order Lagrangian distribution
	$$
	 A = (2\pi h)^{-k- (d + 2N)/4} \int\int e^{i \phi(x, \theta) / h} a(x, \theta, h) \, d\theta |dx|^{1/2},
	$$ 
	which can be defined by $\phi$ and $a \in S$, respectively, $\tilde{\phi}$ and $\tilde{a} \in S$, then we will show that 
	\begin{equation}
	e^{ i \pi \text{sgn}\, \phi^{\prime\prime}_{\theta\theta}/4}a(x, \theta, h) \sqrt{d_C}  - e^{ i \pi \text{sgn}\, \tilde{\phi}^{\prime\prime}_{\theta\theta}/4} \tilde{a}(x, \tilde{\theta}, h) \sqrt{d_{\tilde{C}}} \in h S(\Lambda, \Omega^{1/2}).
	\end{equation}
	This allows us to define an invariant principal symbol for $A$, which is a half-density on $\Lambda$ with values in the Maslov bundle. 
	
	To do this, we increase the number of variables in the phase functions by adding a nondegenerate quadratic form to each:
	$$
	\Phi(x, \theta, \theta') = \phi(x, \theta) + Q(\theta', \theta'), \quad \tilde\Phi(x, \tilde\theta, \tilde\theta') = \tilde\phi(x, \tilde\theta) + \tilde Q(\tilde\theta', \tilde\theta').
	$$
	We may do this in such a way that $\Phi$ and $\tilde\Phi$ have the same number of total extra variables, and the same signature (note this requires a mod 2 compatibility between the number of $\theta$ variables and the signature of the Hessian of the phase function in those variables; this is guaranteed by Lemma~\ref{lem:Nrank}). 
	
	We now invoke Proposition~\ref{prop:equivalence} and assert that the phase functions $\Phi, \tilde\Phi$ are equivalent. We write $A$ using the these functions; the amplitude  is now $a |\det Q|^{1/2} e^{-i\pi \sgn Q/4}$ for $\Phi$, respectively $\tilde a |\det \tilde Q|^{1/2} e^{-i\pi \sgn \tilde Q/4}$, for $\tilde \Phi$, to cancel the effect of the quadratic form in the phase. 
	Applying \eqref{adC}, we find that 
	$$
	a |\det Q|^{1/2} e^{-i\pi \sgn Q/4} \sqrt{d_{C_{\Phi}}} = \tilde a |\det \tilde Q|^{1/2} e^{-i\pi \sgn \tilde Q/4} \sqrt{d_{C_{\tilde \Phi}}}.
	$$
	On the other hand, it is easy to see that $\sqrt{d_{C_{\Phi}}}$ and $\sqrt{d_{C_{\phi}}}$ are related by 
	$$
	\sqrt{d_{C_{\Phi}}} = |\det Q|^{-1/2} \sqrt{d_{C_{\phi}}},
	$$
	and similarly for $\sqrt{d_{C_{\tilde\Phi}}}$ and $\sqrt{d_{C_{\tilde\phi}}}$. We deduce that 
	$$ e^{ i \pi \sigma/4} a(x, \theta, h) \sqrt{d_C} -   \tilde{a}(x, \tilde{\theta}, h) \sqrt{d_{\tilde{C}}} \in h S(\Lambda, \Omega^{1/2}),$$
	where we use the notation 
	$$\sigma = \text{sgn}\, \phi^{\prime\prime}_{\theta\theta} - \text{sgn}\, \tilde{\phi}^{\prime\prime}_{\theta\theta} = (\text{rank}\, \phi^{\prime\prime}_{\theta\theta} - \text{rank}\, \tilde{\phi}^{\prime\prime}_{\theta\theta})\mod 2,$$ which is an even number. 
	
	By  Lemma \ref{lem:Nrank} we have $\sigma = (N - \tilde{N})\mod 2$. Hence $\sigma^\prime = (\sigma - N + \tilde{N})/2 \in \mathbb{Z}.$ Then we have
	$$i^{\sigma^\prime} e^{i\pi N/4} a(x, \theta, h) \sqrt{d_C}  - e^{i\pi \tilde{N}/4} \tilde{a}(x, \tilde{\theta}, h) \sqrt{d_{\tilde{C}}}  \in hS(\Lambda, \Omega^{1/2}).$$ 
	Following \cite[Section 3]{Hormander-Acta-1971}, we interpret the discrepancy factor $i^{\sigma^\prime}$ as a transition function for a line bundle, the so-called Maslov bundle, defined over $\Lambda$. 
	We obtained the same Maslov transition functions as for classical (homogeneous) Lagrangian distributions. 
	This shows that $e^{i\pi N/4} a(x, \theta, h) \sqrt{d_C}$ has invariant meaning as a section of the Maslov bundle defined in \cite[Section 3]{Hormander-Acta-1971}. 
	
	\begin{definition}
	The principal symbol of $A$ is defined by 
	$$
	\sigma^k(A) = e^{ i \pi N/4} a(x, \theta, h) \sqrt{d_C} \in S(\Lambda, L \otimes \Omega^{1/2})/ h S(\Lambda, L \otimes \Omega^{1/2})
	$$
	where $L$ is the Maslov bundle over $\Lambda$, and $(a, C, N)$ are the data corresponding to any local oscillatory integral representation of $A$ as above. 
	\end{definition}

	\subsection{Exact sequence}
	We have the exact sequence $$0 \hookrightarrow I^{k - 1}(X, \Lambda; \Omega^{1/2}) \hookrightarrow I^k(X, \Lambda; \Omega^{1/2}) \xrightarrow{\ \ \sigma_k \ \ } S(\Lambda, L \otimes \Omega^{1/2})/ h S(\Lambda, L \otimes \Omega^{1/2}) \longrightarrow 0.$$
	This is equivalent to saying that $\sigma^k$ is surjective and its kernel is $I^{k-1}(X, \Lambda; \Omega^{1/2})$. To show surjectivity, let $s \in S(\Lambda, L \otimes \Omega^{1/2})$ with $s_j \in S (U^\Lambda_j, \Omega^{1/2})$ be given. One may pull back $s_j$ to $a_j \sqrt{d_{C_j}}$ under the map from $C_j \cap U_j$ to $U^\Lambda_j$, so $ a_j \in S(C_j)$. By taking a homogeneous $C^\infty$ retraction to $C_j$, one may extend $a_j$ to $U_j$. Using the global definition, we define a semiclassical Lagrangian distribution in $I^k$. Different choices of the extension off $C_j$ cause an error in $I^{k - 1}$. Therefore the map is well-defined.
	Injectivity is shown using oscillatory testing, as in \cite[Section 3.2]{Hormander-Acta-1971}. 
	
	\subsection{Canonical relation}
	
	To establish the calculus,  we connect semiclassical Fourier integral operators with canonical relations.
	
	A semiclassical Fourier integral operator with kernel $A \in I_k(\Lambda)$ is a map $$\mathcal{S}(Y, \Omega^{1/2}) \longrightarrow \mathcal{S}^\prime(X, \Omega^{1/2}),$$ where $X$ and $Y$ are two manifolds of dimension $n_X$ and $n_Y$ respectively, and $\Lambda$ is a Lagrangian of $(T^\ast X \times T^\ast Y , \sigma_X + \sigma_Y)$. The canonical relation $\mathcal{C}$ is a Lagrangian of $(T^\ast X \times T^\ast Y , \sigma_X - \sigma_Y)$.
	
	The Lagrangian $\Lambda$ can be parametrized by the phase function of $A$, say $\phi(x, y, \theta)$, as $$\{(x, \phi^\prime_x, y, \phi^\prime_y) | \phi^\prime_\theta = 0\}.$$ We have local coordinates on corresponding canonical relation $$\mathcal{C} = \Lambda^\prime = \{(x, \phi^\prime_x, y, -\phi^\prime_y) | \phi^\prime_\theta = 0\}.$$
	
	In particular, the canonical relation $\mathcal{C}$ from $T^\ast Y$ to $T^\ast X$ is a local canonical graph if the projection $\mathcal{C} \rightarrow T^\ast Y$, consequently $T^\ast X$, is a local diffeomorphism.

	\subsection{Composition}
	
	H\"{o}rmander's proof still holds for following theorem of semiclassical Lagrangian distributions.
	
	\begin{theorem}\label{composition-semiFIO}Let $\mathcal{C}_1$ be a canonical relation from $T^\ast (Y)$ to $T^\ast (X)$ and $\mathcal{C}_2$ another from $T^\ast (Z)$ to $T^\ast (Y)$ with three manifolds $X, Y, Z$, $\mathcal{C}_1 \times \mathcal{C}_2$ intersects the diagonal in $T^\ast (X) \times T^\ast (Y) \times T^\ast (Y) \times T^\ast (Z)$ transversally, and the projection from the intersection to $T^\ast (X) \times T^\ast (Z)$ is proper, then $\mathcal{C}_1 \circ \mathcal{C}_2$ is a homogeneous relation from $T^\ast (Z)$ to $T^\ast (X)$. Moreover, $$A_1A_2 \in I^{k_1 + k_2} (X \times Z, (\mathcal{C}_1 \circ \mathcal{C}_2)^\prime),$$ provided properly supported $A_1 \in I^{k_1}(X \times Y, \mathcal{C}_1^\prime)$ and $A_2 \in I^{k_2} (Y \times Z, \mathcal{C}_2^\prime)$.\end{theorem}

	 Standard parametrix construction by symbol calculus for elliptic operators applies to semiclassical Fourier integral operators with elliptic symbols as follows.
	
	\begin{theorem}Suppose $A \in I^k(X \times Y, \mathcal{C}^\prime)$ has an elliptic symbol, provided $\mathcal{C}$ is a canonical graph from $T^\ast(Y)$ onto $T^\ast(X)$. There is a two sided parametrix $B \in I^{- k}(Y \times X, (\mathcal{C}^{-1})^\prime)$ of $A$.\end{theorem}
	
	\subsection{Vanishing principal symbol calculus}
	
	Regarding the calculus for operators with a vanishing principal symbol, we shall use sub-principal symbol instead.
	
	Consider $0$-th order semiclassical pseudodifferential operator (expressed using left  quantization) defined on a manifold $X$ $$P_h  = \frac{1}{(2\pi h)^d} \int\int e^{i(x - y) \cdot \xi / h} P(x, \xi, h)  \, d\xi |dxdy|^{1/2},$$ with smooth amplitude $P$. One may do an oscillatory test $e^{-i \phi(x) / h} P_h (\omega e^{i \phi / h}) (x)$ given a smooth function $\phi$ and smooth half density $\omega = \omega(y) |dy|^{1/2}$ on $X$. \begin{eqnarray*}
	\lefteqn{e^{- i \phi(x) / h} P_h (\omega e^{i \phi / h})(x)}\\
	& = & \frac{e^{- i \phi(x) / h}}{(2\pi h)^d} \int\int e^{i(x - y) \cdot \xi / h} P(x, \xi, h) \omega(y) e^{i \phi(y) / h} \, d\xi dy \\ & = & \frac{1}{(2\pi h)^d} \int\int e^{\frac{i}{h} ( (x - y)\cdot\xi + \phi(y) - \phi(x) )} P(x, \xi, h) \omega(y) \, d\xi dy\\ & = & \frac{1}{(2\pi h)^d} \int\int e^{\frac{i}{h} ((x - y)\cdot(\xi - \phi^\prime(x)))} P(x, \xi, h) \omega(y) e^{\frac{i}{h}(\phi(y) - \phi(x) - (y - x)\cdot\phi^\prime(x))}\, d\xi dy.
	\end{eqnarray*} Stationary phase gives an asymptotic expansion, $$\sum_{|\alpha| \geq 0} \frac{(ih)^{|\alpha|}}{\alpha !} D_\xi^\alpha P(x, \xi, h)\bigg|_{\xi = \phi^\prime(x)} D_y^\alpha \Big(\omega(y) e^{\frac{i}{h}(\phi(y) - \phi(x) - (y - x)\phi^\prime(x))}\Big)\bigg|_{y = x}.$$ We just consider the leading and sub-leading terms $$P(x, \phi^\prime(x), h)\omega(x) + h\bigg( \sum_{j = 1}^d P^{(j)}(x, \phi^\prime(x), h) D_j \omega(x) + \frac{1}{2i} \sum_{j, l = 1}^d P^{(jl)}(x, \phi^\prime(x), h) \omega(x) \partial_{x_j x_l}^2 \phi\bigg),$$ where $P^{(j)}_{(l)}(x, \xi, h) = \partial_{x_l \xi_j}^2 P(x, \xi, h)$.
	 We write in terms of the Lie derivative $\mathcal{L}_{\nabla_{\xi} P(x, \phi^\prime(x), h)} \omega$, \begin{eqnarray*}\lefteqn{e^{-i\phi(x) / h}P_h(\omega e^{i\phi / h})(x)}\\&=& P(x, \phi^\prime_x, h) \omega(x)  + \frac{h}{2i}\bigg(\sum_{j, l = 1}^d P^{(jl)}(x, \phi^\prime(x), h) \omega(x) \partial_{x_j x_l}^2 \phi - \sum_{j = 1}^d  \partial_{x_j} P^{(j)}(x, \phi^\prime_x, h) \omega(x)\bigg) \\&& + h \sum_{j = 1}^d P^{(j)}(x, \phi^\prime(x), h) D_j \omega(x) + \frac{h}{2i} \sum_{j = 1}^d  \partial_{x_j} P^{(j)}(x, \phi^\prime_x, h) \omega(x) + o(h^2)\\&=& P(x, \phi^\prime_x, h) \omega(x)  - \frac{h}{2i} \sum_{j, l = 1}^d P^{(j)}_{(j)}(x, \phi^\prime(x), h) \omega(x)   + \frac{h}{i}\mathcal{L}_{\nabla_\xi P(x, \phi^\prime(x), h)} \omega + o(h^2).\end{eqnarray*}
	
	Therefore, for semiclassical pseudodifferential operator $P_h$, $$P(x, \xi, h)  - \frac{h}{2i} \sum_{j = 1}^d \frac{\partial P}{\partial x_j \xi_j}(x, \xi, h)  \in S / h^2 S,$$ is invariantly defined. 
	Moreover, if $P(x, \xi, h) = p(x, \xi) + O(h)$, then  we can uniquely determine a sub-principal symbol 
	\begin{equation}
	s(x, \xi, h) = P(x, \xi, h) - p(x, \xi)  - \frac{h}{2i} \sum_{j = 1}^d \frac{\partial P}{\partial x_j \xi_j}(x, \xi, h) \in hS.
	\label{subpr}\end{equation}
	We have the following calculus for vanishing symbols with this notion.
	
	\begin{theorem}\label{vanishing-symbol-calculus-1}Let $P_h \in \Psi^k(X)$ be a properly supported semiclassical pseudo-differential operator with full symbol $P(x, \xi, h)$, principal symbol $p(x, \xi)$ and sub-principal symbol $s(x, \xi, h)$, $\mathcal{C}$ be a canonical relation from $T^\ast(Y)$ to $T^\ast(X)$ such that $p$ vanishes on the projection of $\mathcal{C}$ in $T^\ast(X)$. For an $A_h \in I^{k^\prime} (X \times Y, \mathcal{C}^\prime)$ with principal symbol $a \in S(\mathcal{C}^\prime, L \otimes \Omega^{1/2})$, $P_h A_h \in I^{k + k^\prime} (X \times Y, \mathcal{C}^\prime)$ has a principal symbol \begin{equation*} \frac{h}{i} \mathcal{L}_{H_{p}} (a) + s a.\end{equation*} Here $H_{p}$ is the Hamilton field of $p$ lifted to a function on $T^\ast(X) \setminus 0 \times T^\ast(Y) \setminus 0$ via the projection onto the first factor and $\mathcal{L}$ denotes the Lie derivative of half densities.\end{theorem}

	\begin{proof} We locally parametrize $\Lambda$ using Lemma~\ref{new parametrization}. Thus we can write
	$$\begin{gathered}
	\Phi(z', z'', \xi'') = \langle z'' - Z''(z', \xi''), \xi''\rangle + f(z', \xi''), \\
	\Lambda = \{ (z, d_z \Phi) \mid  d_{\xi''} \Phi = 0 \} =  \{ (z', Z''(z', \xi''), \Xi'(z', \xi''), \xi'') \}. 
	\end{gathered}$$
	
	The principal term of the composition we are studying takes the form 
	\begin{equation}(2 \pi h)^{- k - d - (d + 2N)/4} \int e^{i((z - x) \cdot \zeta + \Phi(x, \xi^{\prime\prime}))/h} (p(z, \zeta) + h r(z, \zeta, h)) a(x^\prime, \xi^{\prime\prime}, h) \, dx d\zeta d\xi^{\prime\prime},
	\label{comp-expr}\end{equation}
	 where we can let $a$ be independent of $x^{\prime\prime}$. To keep the symplectic structure, we integrate in the $(x, \zeta)$ variables and apply the stationary phase expansion. Consequently, it will leave $\xi^{\prime\prime}$ and replace $x$ by $z$.
	
	The phase function is stationary at $\{\zeta = \Phi^\prime_x, x = z\}$, where the Hessian, respectively its inverse, is $$\left( \begin{array}{cc} \Phi^{\prime\prime}_{zz}(z, \xi^{\prime\prime}) & - I\\ -I & 0  \end{array} \right), \quad\quad \left( \begin{array}{cc} 0 & - I\\ -I & - \Phi^{\prime\prime}_{zz}(z, \xi^{\prime\prime}) \end{array} \right).$$
	
	The leading term in the stationary phase expansion of \eqref{comp-expr} is 
	\begin{equation}\label{leading integral} 
	(2 \pi h)^{- k  - (d + 2N)/4} \int e^{i\Phi(z, \xi^{\prime\prime})/h} \Big( p(z^\prime, z^{\prime\prime}, \Phi_{z^\prime}^\prime, \Phi_{z^{\prime\prime}}^\prime) + h r(z^\prime, z^{\prime\prime}, \Phi_{z^\prime}^\prime, \Phi_{z^{\prime\prime}}^\prime) \Big) a(z^\prime, \xi^{\prime\prime}, h) \,  d\xi^{\prime\prime}.
	\end{equation} 
	We note that for this phase function, $\Phi'_{z'} = \Xi'$ and $\Phi'_{z''} = \xi''$, using \eqref{restriction2}. We expand $p(z', z'', \Xi', \xi'')$ around $z'' = Z''$: 
	\begin{equation}\begin{gathered}
	p(z', z'', \Xi', \xi'') = p(z', Z'', \Xi', \xi'') + (z''_j - Z''_j) \frac{\partial p}{\partial z''_j} (z', Z'', \Xi', \xi'') \\ + \frac1{2}{(z''_j - Z''_j)(z''_k - Z''_k)} \frac{\partial^2 p}{\partial z''_j \partial z''_k} (z', Z'', \Xi', \xi'') + 
	 \, \text{third order remainder}.
	\end{gathered} \end{equation} 
	 It can be reduced to $$\Phi^\prime_{\xi^{\prime\prime}_j}  \frac{\partial p}{\partial z^{\prime\prime}_j}  + \frac{\Phi^\prime_{\xi^{\prime\prime}_j} \Phi^\prime_{\xi^{\prime\prime}_k}}{2}  \frac{\partial^2 p}{\partial z^{\prime\prime}_j \partial z^{\prime\prime}_k} + \, \text{third order remainder}$$ on $\Lambda$, since by assumption, $p$ vanishes on $\Lambda$. Returning to the leading term (\ref{leading integral}), we integrate by parts and obtain 
	 \begin{equation}\begin{gathered} \frac{h}{i}  \,   (2 \pi h)^{- k  - (d + 2N)/4} \int e^{i\Phi(z, \xi^{\prime\prime})/h}  \Bigg( i r a - \frac{\partial}{\partial \xi^{\prime\prime}_j}\bigg( \frac{\partial p}{\partial z^{\prime\prime}_j}  \big(z^\prime, Z^{\prime\prime}(z^\prime, \xi^{\prime\prime}), \Xi^\prime(z^\prime, \xi^{\prime\prime}), \xi^{\prime\prime}\big) \, a(z^\prime, \xi^{\prime\prime}, h)  \\  \quad \quad + \frac{\Phi_{\xi^{\prime\prime}_k}^\prime}{2} \frac{\partial^2 p}{\partial z_j^{\prime\prime} \partial z_k^{\prime\prime}}   \big(z^\prime, Z^{\prime\prime}(z^\prime, \xi^{\prime\prime}), \Xi^\prime(z^\prime, \xi^{\prime\prime}), \xi^{\prime\prime}\big) \, a(z^\prime, \xi^{\prime\prime}, h)\bigg) \, \Bigg) \,  d\xi^{\prime\prime},\end{gathered}\label{bbb}\end{equation}
	where the amplitude expands as 
	\begin{equation}\label{leading integral 2} - \frac{h}{i} \bigg(\frac{\partial p}{\partial z^{\prime\prime}_j} \frac{\partial a}{\partial \xi^{\prime\prime}_j} + \frac{1}{2} \frac{\partial Z^{\prime\prime}_k}{\partial \xi^{\prime\prime}_j} \frac{\partial^2 p}{\partial z^{\prime\prime}_j \partial z^{\prime\prime}_k} a  + \frac{\partial \Xi^\prime_k}{\partial \xi^{\prime\prime}_j}\frac{\partial^2 p}{\partial \xi^\prime_k \partial z^{\prime\prime}_j} a + \frac{\partial^2 p}{\partial \xi^{\prime\prime}_j \partial z^{\prime\prime}_j} a\bigg) + O(h^2).\end{equation}
	(Here, the second term with the factor $1/2$ arises both from the first and second terms of \eqref{bbb}, with coefficient $1$ from the first term and $-1/2$ from the second term, using $\Phi'_{\xi''_k} = (z''_k - Z''_k)$.)
	
	Next we treat the subleading term in the stationary phase expansion of \eqref{comp-expr}. 
	Inserting the Hessian inverse, we get  $$(2 \pi h)^{- k - (d + 2N)/4} \int e^{i \Phi(z, \xi^{\prime\prime})/h} \frac{h}{i} \bigg( \frac{\Phi^{\prime\prime}_{x_j^\prime x_k^\prime}}{2} \frac{\partial^2}{\partial \zeta_j^\prime \partial \zeta_k^\prime} + \frac{\partial^2 }{\partial x_j \partial \zeta_j}\bigg)\bigg|_{\zeta = \Phi^\prime_z, x = z} \big(p(z, \zeta)\, a(x^\prime, \xi^{\prime\prime}, h)\big) \, d\xi^{\prime\prime}.$$ The amplitude restricted to $\Lambda$ is \begin{equation}\label{subleading integral} \frac{h}{i} \bigg(\frac{1}{2} \frac{\partial \Xi^\prime_j}{\partial z^\prime_k} \frac{\partial^2 p}{\partial \xi^{\prime}_j \partial \xi^{\prime}_k} (z^\prime, Z^{\prime\prime}, \Xi^\prime, \xi^{\prime\prime}) a(z^\prime, \xi^{\prime\prime}, h) + \frac{\partial p}{\partial \xi_j^\prime}(z^\prime, Z^{\prime\prime}, \Xi^\prime, \xi^{\prime\prime})  \frac{\partial a}{\partial z_j^\prime}(z^\prime, \xi^{\prime\prime}, h) \bigg) \end{equation}
	
	Combining (\ref{leading integral 2}) and (\ref{subleading integral}), we get the principal symbol of the composition operator on $\Lambda$ as $$r a + \frac{h}{i}\bigg(\frac{1}{2} \frac{\partial \Xi^\prime_j}{\partial z^\prime_k} \frac{\partial^2 p}{\partial \xi^{\prime}_j \partial \xi^{\prime}_k} a + \frac{\partial p}{\partial \xi_j^\prime}  \frac{\partial a}{\partial z_j^\prime} - \frac{\partial p}{\partial z^{\prime\prime}_j} \frac{\partial a}{\partial \xi^{\prime\prime}_j} - \frac{1}{2} \frac{\partial Z^{\prime\prime}_k}{\partial \xi^{\prime\prime}_j} \frac{\partial^2 p}{\partial z^{\prime\prime}_j \partial z^{\prime\prime}_k} a  - \frac{\partial \Xi^\prime_k}{\partial \xi^{\prime\prime}_j}\frac{\partial^2 p}{\partial \xi^\prime_k \partial z^{\prime\prime}_j} a - \frac{\partial^2 p}{\partial \xi^{\prime\prime}_j \partial z^{\prime\prime}_j} a\bigg).$$ 
	Noting the definition \eqref{subpr} of subprincipal symbol, the principal symbol of $P_hA_h$ can be written  \begin{eqnarray}
	\lefteqn{ sa + \frac{h}{i}\bigg(\frac{\partial p}{\partial \xi_j^\prime}  \frac{\partial a}{\partial z_j^\prime} - \frac{\partial p}{\partial z^{\prime\prime}_j} \frac{\partial a}{\partial \xi^{\prime\prime}_j}\bigg)} \nonumber \\ \label{principal symbol-vanishing} && + \frac{h}{i}\bigg( \frac{1}{2} \frac{\partial^2 p}{\partial \xi^{\prime}_j \partial z^{\prime}_j} a + \frac{1}{2} \frac{\partial \Xi^\prime_j}{\partial z^\prime_k} \frac{\partial^2 p}{\partial \xi^{\prime}_j \partial \xi^{\prime}_k} a - \frac{1}{2} \frac{\partial Z^{\prime\prime}_k}{\partial \xi^{\prime\prime}_j} \frac{\partial^2 p}{\partial z^{\prime\prime}_j \partial z^{\prime\prime}_k} a  - \frac{\partial \Xi^\prime_k}{\partial \xi^{\prime\prime}_j}\frac{\partial^2 p}{\partial \xi^\prime_k \partial z^{\prime\prime}_j} a - \frac{1}{2} \frac{\partial^2 p}{\partial \xi^{\prime\prime}_j \partial z^{\prime\prime}_j} a \bigg)
	\end{eqnarray}

	On the other hand, let us look at the Lie derivative of half densities on $\Lambda$. For local coordinates $\lambda$ on $\Lambda$, if $H_p = \sum_i \kappa_i \partial_{\lambda_i}$, we have 
	$$
	\mathcal{L}_{H_p}(a |d\lambda|^{1/2}) = \Big( H_p (a) + \text{div}(H_p) a/2 \Big) |d\lambda|^{1/2}, \quad \text{div} H_p = \sum_i \frac{\partial \kappa_i}{\partial \lambda_i} . 
	$$

	We consider the Hamilton vector field term first. As we use $(z^\prime, \xi^{\prime\prime})$ on $\Lambda$, \begin{eqnarray*}\frac{\partial}{\partial z^\prime_i} &=& \frac{\partial}{\partial z_i^\prime} + \frac{\partial Z_j^{\prime\prime}}{\partial z_i^{\prime}} \frac{\partial}{\partial z_j^{\prime\prime}} + \frac{\partial \Xi^\prime_k}{\partial z_i^\prime} \frac{\partial}{\partial \xi_k^\prime}  \\
	\frac{\partial}{\partial \xi^{\prime\prime}_i} &=& \frac{\partial}{\partial \xi_i^{\prime\prime}} + \frac{\partial Z_j^{\prime\prime}}{\partial \xi_i^{\prime\prime}} \frac{\partial}{\partial z_j^{\prime\prime}} + \frac{\partial \Xi^\prime_k}{\partial \xi_j^{\prime\prime}} \frac{\partial}{\partial \xi_k^\prime}.\end{eqnarray*} The principal symbol $p$ vanishing on the Lagrangian $\Lambda$ implies more information. Differentiating identity $p(z^\prime, Z^{\prime\prime}, \Xi^\prime, \xi^{\prime\prime}) = 0$ gives \begin{eqnarray*}\frac{\partial p}{\partial z_i^\prime} + \frac{\partial Z_j^{\prime\prime}}{\partial z_i^\prime} \frac{\partial p}{\partial z_j^{\prime\prime}} + \frac{\partial \Xi^\prime_k}{\partial z_i^\prime} \frac{\partial p}{\partial \xi_k^\prime} &=& 0 \\ \frac{\partial p}{\partial \xi_i^{\prime\prime}} + \frac{\partial Z_j^{\prime\prime}}{\partial \xi_i^{\prime\prime}} \frac{\partial p}{\partial z_j^{\prime\prime}} + \frac{\partial \Xi^{\prime}_k}{\partial \xi_i^{\prime\prime}} \frac{\partial p}{\partial \xi_k^\prime} &=& 0 
	\end{eqnarray*} 
	Also as $dz^\prime_i \wedge d\Xi^\prime_i + dZ_j^{\prime\prime} \wedge d\xi_j^{\prime\prime} = 0$ on $\Lambda$, we have 
	\begin{equation}
	\frac{\partial Z^{\prime\prime}_j}{\partial \xi^{\prime\prime}_i}  =  \frac{\partial Z^{\prime\prime}_i}{\partial \xi^{\prime\prime}_j}, \quad\quad  \frac{\partial \Xi^\prime_j}{\partial z^{\prime}_i} = \frac{\partial \Xi^\prime_i}{\partial z^{\prime}_j},   \quad\quad   \frac{\partial \Xi_i^\prime}{\partial \xi_j^{\prime\prime}} + \frac{\partial Z_j^{\prime\prime}}{\partial z_i^\prime} = 0.
	\label{bbbb}\end{equation}
	 One thus can have \begin{eqnarray*}
	\lefteqn{\frac{\partial p}{\partial \xi^\prime_i} \frac{\partial}{\partial z^\prime_i} -  \frac{\partial p}{\partial z_i^{\prime\prime}} \frac{\partial}{\partial \xi^{\prime\prime}_i}}\\ &=&   \frac{\partial p}{\partial \xi^\prime_i} \bigg(\frac{\partial}{\partial z_i^\prime} + \frac{\partial Z_j^{\prime\prime}}{\partial z_i^{\prime}} \frac{\partial}{\partial z_j^{\prime\prime}} + \frac{\partial \Xi^\prime_j}{\partial z_i^\prime} \frac{\partial}{\partial \xi_j^\prime}\bigg)   -  \frac{\partial p}{\partial z_i^{\prime\prime}} \bigg(\frac{\partial}{\partial \xi_i^{\prime\prime}} + \frac{\partial Z_j^{\prime\prime}}{\partial \xi_i^{\prime\prime}} \frac{\partial}{\partial z_j^{\prime\prime}} + \frac{\partial \Xi^\prime_j}{\partial \xi_i^{\prime\prime}} \frac{\partial}{\partial \xi_j^\prime}\bigg)\\  &=&  \frac{\partial p}{\partial \xi^\prime_i} \frac{\partial}{\partial z_i^\prime} -  \frac{\partial p}{\partial z_i^{\prime\prime}} \frac{\partial}{\partial \xi_i^{\prime\prime}} + \bigg(\frac{\partial p}{\partial \xi^\prime_i} \frac{\partial \Xi^\prime_j}{\partial z_i^\prime}  - \frac{\partial p}{\partial z_i^{\prime\prime}} \frac{\partial \Xi^\prime_j}{\partial \xi_i^{\prime\prime}} \bigg) \frac{\partial}{\partial \xi_j^\prime}  + \bigg(\frac{\partial p}{\partial \xi^\prime_i} \frac{\partial Z_j^{\prime\prime}}{\partial z_i^{\prime}}  -  \frac{\partial p}{\partial z_i^{\prime\prime}}  \frac{\partial Z_j^{\prime\prime}}{\partial \xi_i^{\prime\prime}}  \bigg)\frac{\partial}{\partial z_j^{\prime\prime}}\\   &=&  \frac{\partial p}{\partial \xi^\prime_i} \frac{\partial}{\partial z_i^\prime} -  \frac{\partial p}{\partial z_i^{\prime\prime}} \frac{\partial}{\partial \xi_i^{\prime\prime}} + \bigg(\frac{\partial p}{\partial \xi^\prime_i} \frac{\partial \Xi^\prime_i}{\partial z_j^\prime}  + \frac{\partial p}{\partial z_i^{\prime\prime}} \frac{\partial Z^{\prime\prime}_i}{\partial z_j^{\prime}} \bigg) \frac{\partial}{\partial \xi_j^\prime}  - \bigg(\frac{\partial p}{\partial \xi^\prime_i} \frac{\partial \Xi_i^{\prime}}{\partial \xi_j^{\prime\prime}}  +  \frac{\partial p}{\partial z_i^{\prime\prime}}  \frac{\partial Z_i^{\prime\prime}}{\partial \xi_j^{\prime\prime}}  \bigg)\frac{\partial}{\partial z_j^{\prime\prime}}\\   &=&  \frac{\partial p}{\partial \xi^\prime_i} \frac{\partial}{\partial z_i^\prime} -  \frac{\partial p}{\partial z_i^{\prime\prime}} \frac{\partial}{\partial \xi_i^{\prime\prime}} - \frac{\partial p}{\partial z_j^\prime}\frac{\partial}{\partial \xi_j^\prime}  + \frac{\partial p}{\partial \xi^{\prime\prime}_j} \frac{\partial}{\partial z_j^{\prime\prime}}\\ &=& H_p,
	\end{eqnarray*} which shows following lemma.

	\begin{lemma}  The Hamilton vector of $p$, restricted to Lagrangian $\Lambda$, is $$\frac{\partial p}{\partial \xi^\prime_i}(z^\prime, Z^{\prime\prime}, \Xi^{\prime}, \xi^{\prime\prime}) \frac{\partial}{\partial z^\prime_i} -  \frac{\partial p}{\partial z_j^{\prime\prime}}(z^\prime, Z^{\prime\prime}, \Xi^{\prime}, \xi^{\prime\prime}) \frac{\partial}{\partial \xi^{\prime\prime}_j},$$ provided that $p$ vanishes on $\Lambda$.\end{lemma}
	
	Using above expression of Hamilton vector field, we have
	\begin{lemma} Using coordinates $(z', \xi'')$ on $\Lambda$, the divergence $\mathrm{div} H_p$ on $\Lambda$, times $(2i)^{-1} h$,  is \begin{eqnarray*}
	\frac{h}{2i} \Bigg( \lefteqn{\frac{\partial}{\partial z^\prime_i} \bigg( \frac{\partial p}{\partial \xi^\prime_i}(z^\prime, Z^{\prime\prime}, \Xi^{\prime}, \xi^{\prime\prime}) \bigg) -  \frac{\partial}{\partial \xi^{\prime\prime}_j} \bigg( \frac{\partial p}{\partial z_j^{\prime\prime}}(z^\prime, Z^{\prime\prime}, \Xi^{\prime}, \xi^{\prime\prime}) \bigg)\Bigg) } \\ 
	& = & \frac{h}{2i} \Bigg( \frac{\partial^2 p}{\partial z_i^\prime \partial \xi_i^\prime}  +  \frac{\partial Z^{\prime\prime}_j}{\partial z^\prime_i} \frac{\partial^2 p}{\partial z_j^{\prime\prime} \partial \xi_i^\prime} + \frac{\partial \Xi^\prime_j}{\partial z^\prime_i} \frac{\partial^2 p}{\partial \xi_j^\prime \partial \xi_i^\prime} -   \frac{\partial Z^{\prime\prime}_i}{\partial \xi^{\prime\prime}_j} \frac{\partial^2 p}{\partial z_i^{\prime\prime} \partial z_j^{\prime\prime}} - \frac{\partial \Xi^{\prime}_i}{\partial \xi^{\prime\prime}_j} \frac{\partial^2 p}{\partial \xi^\prime_i \partial z_j^{\prime\prime}} - \frac{\partial^2 p}{\partial \xi_j^{\prime\prime} \partial z_j^{\prime\prime}} \Bigg) .\end{eqnarray*}
	\end{lemma}
	Here, the second and fifth terms on the RHS are equal, using the last identity of \eqref{bbbb}. Using these lemmas and comparing with 
	(\ref{principal symbol-vanishing}) proves the theorem.
	\end{proof}

	\section{Semiclassical intersecting Lagrangian distribution}\label{sec:appB}
	
	In this section we shall adapt the work of intersecting Lagrangian distributions, due to Melrose and Uhlmann \cite{Melrose-Uhlmann-CPAM-1979}, and develop analogous semiclassical analysis for the use near the diagonal.
	
	\subsection{Lagrangian intersection}
	
	To construct symbolic global parametrices for semiclassical pseudodifferential operators of real principal type, we shall have a quick review of Lagrangian intersection introduced by Melrose and Uhlmann \cite{Melrose-Uhlmann-CPAM-1979}. Suppose $X$ is a $C^\infty$ manifold of dimension $d \geq 2$. A pair of Lagrangian manifolds $(\Lambda_0, \Lambda_1)$ of $T^\ast X$ , where $\Lambda_1$ has a boundary, is said to be an intersecting pair of Lagrangian manifolds, if $$\Lambda_0 \cap \Lambda_1 = \partial \Lambda_1 \quad \mbox{and} \quad T_\lambda(\Lambda_0) \cap T_\lambda(\Lambda_1) = T_\lambda(\partial \Lambda_1), \, \mbox{for any}\, \lambda \in \partial \Lambda_1.$$
	
	Consider Lagrangian manifolds $$\tilde{\Lambda}_0 = T^\ast_0 \mathbb{R}^d  \quad \mbox{and} \quad \tilde{\Lambda}_1 = \{ (x, \xi) \in T^\ast \mathbb{R}^d : (x_2, \dots, x_d) = 0, \xi_1 = 0, x_1 \geq 0 \}.$$ We introduce space $I^k(\mathbb{R}^d; \tilde{\Lambda}_0, \tilde{\Lambda}_1, \Omega^{1/2}) \subset \mathcal{S}^\prime(\mathbb{R}^d; \Omega^{1/2})$ consisting of distributions $A = A_1 + A_2$, provided $A_2 \in C_0^\infty(\mathbb{R}^d)$ and $$A_1(x, h) = (2\pi h)^{-k - 3d/4 - 1/2} \int_0^\infty \int_{\mathbb{R}^d} e^{i((x_1 - r)\xi_1 + x^\ast \cdot \xi^\ast)/h} a(r, x, \xi, h) \, d\xi \, dr  \ |dx|^{1/2}, $$ where $x^\ast = (x_2, \dots, x_d)$ respectively $\xi^\ast = (\xi_2, \dots, \xi_d)$, and $a \in S_0$ is a compactly supported amplitude.

	Let us consider the semiclassical wavefront set and rewrite $A$ as the pushforward of the product of Heaviside function $H(r)$ and  $$\tilde{A}(r, x) = (2\pi h)^{-k - 3d/4 - 1/2} \int_{\mathbb{R}^d} e^{i((x_1 - r) \xi_1 + x^\ast \cdot \xi^\ast)/h} a(r, x, \xi, h) \,dr$$ under the projection $(r, x) \longmapsto x$. We shall use the results on semiclassical wave front sets of products and pushforwards. First of all, $$\text{WF}\, (\tilde{A}(s, x)) = \{(x_1, x_1, 0, -\xi_1, \xi_1, \xi^\ast)\} \quad \mbox{and} \quad \text{WF}\, (\tilde{H}(s)) = \{(0, x_1, x^\ast, \sigma, 0, 0)\}.$$ Then the semiclassical wave front set of the product is $$\{(0, x_1, x^\ast, \sigma, 0, 0)\} \cup \{(x_1, x_1, 0, -\xi_1, \xi_1, \xi^\ast) | x_1 \geq 0\} \cup \{(0, 0, 0, \sigma, \xi_1, \xi^\ast)\}.$$ Pushforward theorem of semiclassical wave front sets gives the semiclassical wave front set of $A$ is in $$\{(x_1, 0, 0, \xi^\ast) | x_1 \geq 0\} \cup \{(0, 0, \xi_1, \xi^\ast)\}.$$

	\subsection{Intersecting Lagrangian distributions}
	Following Melrose and Uhlmann \cite{Melrose-Uhlmann-CPAM-1979} we make the 
	\begin{definition}Given an intersecting pair $(\Lambda_0, \Lambda_1)$ of Lagrangian manifolds in a smooth $n$-manifold $X$, the space $I^k(X, \Lambda_0, \Lambda_1, \Omega^{1/2})$ of $k$-th order semiclassical intersecting Lagrangian distributions  consists of distributional half-densities $A$, written in the form of a locally finite sum $$A_0 + A_1 + \sum_j F_j A_j$$modulo a smooth function, provided $$A_0 \in I^{k - 1/2}(X, \Lambda_0, \Omega^{1/2}), \ A_1 \in I^{k}(X, \Lambda_1 \setminus \partial \Lambda_1, \Omega^{1/2}), \ A_j \in I^k(\mathbb{R}^d, \tilde{\Lambda}_0, \tilde{\Lambda}_1, \Omega^{1/2}),$$ where $\{F_j\}$ are zero-th order semiclassical Fourier integral operators mapping $(\tilde{\Lambda}_0, \tilde{\Lambda}_1)$ to $(\Lambda_0, \Lambda_1)$.\end{definition}
	
	The Lagrangian distribution $A$ can be written with respect to any covering of $\Lambda$ with corresponding local parametrizations and elliptic Fourier integral operators $F_j$. To show the independence of choice of $F_j$, we must show that if $F \in I^0(\mathbb{R}^d, \Gamma^\prime)$ with transversal compositions $\Gamma \circ \tilde{\Lambda}_0 \subset \tilde{\Lambda}_0$ and $\Gamma \circ \tilde{\Lambda}_1 \subset \tilde{\Lambda}_1$, then \begin{equation}\label{independence-F} F: I^k(\mathbb{R}^d, \tilde{\Lambda}_0, \tilde{\Lambda}_1) \longrightarrow I^k(\mathbb{R}^d, \tilde{\Lambda}_0, \tilde{\Lambda}_1).
	\end{equation}
	
	Indeed, let $A \in I^k(\mathbb{R}^d, \tilde{\Lambda}_0, \tilde{\Lambda}_1)$, say $$A = (2\pi h)^{- k - 3d/4 - 1/2}\int_0^\infty \int e^{i(y_1 - r) \eta_1 + i y^\ast \cdot \eta^\ast} a(r, y, \eta, h) \,d\eta dr,$$ then the composition with $F$ would be $$FA = (2\pi h)^{- k - 7d/4 - 1/2} \int_0^\infty \int \bigg( \int e^{i(\phi(x, y, \theta) + (y_1 - r) \eta_1 +  y^\ast \cdot \eta^\ast)/h} b(x, y, \theta, h)   a(r, y, \eta, h) \,d\eta dy \bigg) d\theta dr,$$ where $\phi(x, y, \theta) \in C^\infty(\mathbb{R}^{3d})$ is a non-degenerate phase function defining $\Gamma^\prime$. We may integrate out the $y, \eta$ variables via stationary phase. Since the Hessian with respect to $(y, \eta)$ at the critical point $(r, 0, - d_y \phi)$ is non-degenerate, we get an expression
	$$FA = (2\pi h)^{- k - 3d/4 - 1/2} \int_0^\infty \int e^{i \phi(x, (r, 0), \theta)/h} c(r, x, \theta, h) d\theta dr.$$ 
	Borrowing the arguments due to Melrose and Uhlmann \cite{Melrose-Uhlmann-CPAM-1979}, we have the equivalence of phase functions, $\phi(x, (r, 0), \theta)$ and $\theta \cdot x - r \theta_1$, namely, this expression may be written with respect to the standard phase function $\theta \cdot x - r \theta_1$.  The independence of choice of $F_j$ follows.

	\subsection{Principal symbol}
	Let $A \in I^k(X, \Lambda_0, \Lambda_1, \Omega^{1/2})$. We define the principal symbol of $A$ initially by defining them on each Lagrangian separately, away from the  intersection, and examining their behaviour as we approach the intersection. 
	
	Let $B$ be a zeroth order semiclassical pseudodifferential operator. 
	If the operator wavefront set of $B$ is supported away from $\tilde{\Lambda}_0$, say $\text{supp}\, b(a) \in \{|x| \geq \epsilon\}$, we pick a cutoff function such that $$\mu(r) = 1 \quad \mbox{if $r \geq \epsilon/2$} \quad \mbox{and} \quad \mu(r) = 0 \quad \mbox{if $r < \epsilon/4$},$$ and write $BA = U_1 + U_2$ with $$U_1 = (2\pi h)^{-k - 3d/4 - 1/2} \int_0^\infty \int_{\mathbb{R}^d} e^{i((x_1 - r)\xi_1 + x^\ast \cdot \xi^\ast)/h} \mu(r)  b(a)(r, x, \xi, h)\, d\xi dr.$$ Since $B$ is of zero-th order, namely $b(a) \in S_0(\tilde{\Lambda}_1)$, then $U_1 \in I^{k}(\mathbb{R}^d, \tilde{\Lambda}_1)$. Since the phase is not stationary on the set $\{|x| \geq \epsilon\}\cap \{r < \epsilon / 4\}$, $U_2$ is a semiclassically smoothing operator, which gives $BA \in I^{k}(\mathbb{R}^d, \tilde{\Lambda}_1)$.
	
	If the operator wavefront set of $B$ is supported away from $\tilde{\Lambda}_1$, one may assume $\text{supp} \, b(a) \subset \{|x^\prime|^2 + \xi_1^2 / |\xi|^2 < \epsilon^2, x_1 > - \epsilon\}$. Thus the first order semiclassical differential operator $$L = \frac{-ih}{|x^\prime|^2 + \xi_1^2 / |\xi|^2} \bigg(x^\prime \cdot \frac{\partial}{\partial \xi^\prime} - \frac{\xi_1}{|\xi|^2}\frac{\partial}{\partial r}\bigg)$$ is smooth on $\text{supp} \, b(a)$. Noting $$BA = (2\pi h)^{-k - 3d/4 - 1/2} \int_0^\infty \int_{\mathbb{R}^n} L\big(e^{i((x_1 - r)\xi_1 + x^\ast \cdot \xi^\ast)/h }\big) b(a)(r, x, \xi, h) \, d\xi dr,$$ one can take integration by parts and have \begin{eqnarray*}\lefteqn{(2\pi h)^{k + 3d/4 + 1/2} BA} \\&=& \int_0^\infty \int_{\mathbb{R}^n} e^{i((x_1 - r)\xi_1 + x^\ast \cdot \xi^\ast)} L^t(b(a)) \,d\xi dr + \int_{\mathbb{R}^n} e^{i x \cdot \xi / h} \frac{ih\xi_1 b(a)(0, x, \xi, h)}{|\xi_1|^2 + |x^\ast|^2|\xi|^2} \,d\xi\\&=&  \int_0^\infty \int_{\mathbb{R}^d} e^{i((x_1 - r)\xi_1 + x^\ast \cdot \xi^\ast)} (L^t)^m(b(a)) \,d\xi dr + \sum_{j = 1}^m\int_{\mathbb{R}^d} e^{i x \cdot \xi / h} \frac{ih\xi_1 (L^t)^{j - 1}(b(a))(0, x, \xi, h)}{|\xi_1|^2 + |x^\ast|^2|\xi|^2} \,d\xi.\end{eqnarray*} Then $BA \in I^{k - 1/2}(\mathbb{R}^d, \tilde{\Lambda}_0)$, because the first term is semiclassically smoothing as $m$ goes to infinity.
	
	In general, we have
	
	\begin{proposition}\label{symboltheorem}If $B$ is a zeroth order semiclassical pseudodifferential operator with operator wavefront set away from $\Lambda_0$ respectively away from $\Lambda_1$, then $BA \in I^{k}(X, \Lambda_1)$ respectively $BA \in I^{k - 1/2}(X, \Lambda_0)$, provided $A \in I^{k}(X; \Lambda_0, \Lambda_1)$. \end{proposition}
	
	Applying $B$ converts intersecting Lagrangian distributions to usual ones and thus defines principal symbols on $\Lambda_0$ and $\Lambda_1$ restrictively. More precisely, we have local principal symbols 
	\begin{eqnarray*}
	a_1|_{\Lambda_1 \setminus \partial \Lambda_1} = \sigma^{(1)}(A) = \sigma^{k}(BA) / \sigma(B)  \in S(\Lambda_1 \setminus \partial \Lambda_1, \Omega^{1/2} \otimes L_1) 
	\end{eqnarray*} 
	if $\text{WF}\,(B) \cap \Lambda_0 = \emptyset$ and
	\begin{eqnarray*} a_0 = \sigma^{(0)}(A) = \sigma^{k-1/2}(BA) / \sigma(B) \in S(\Lambda_0 \setminus \partial \Lambda_1, \Omega^{1/2} \otimes L_0)
	\end{eqnarray*} 
	if $\text{WF}\,(B) \cap \Lambda_1 = \emptyset$. It is easy to see that these definitions are independent of $B$.

	In addition, an examination of the model situation shows that the symbol at $\Lambda_1$ extends smoothly to $\partial \Lambda_1$, while the symbol $a_0$ at $\Lambda_0$ has the property that $g a_0$ extends smoothly across $\Lambda_0 \cap \Lambda_1$, for any smooth function $g$ on $\Lambda_0$ vanishing at $\Lambda_0 \cap \Lambda_1$. In other words, $a_0$ blows up at $\Lambda_0 \cap \Lambda_1$ at most to first order. Following \cite{Melrose-Uhlmann-CPAM-1979}, we define an invariant map $R$, defined on such symbols on $\Lambda_0 \setminus \Lambda_1$. We choose a function $g \in C^\infty(\Lambda_0)$ vanishing at $\Lambda_0 \cap \Lambda_1$, and with nonvanishing differential there; similarly, we choose a function $f \in C^\infty(\Lambda_1)$ vanishing at $\Lambda_0 \cap \Lambda_1 = \partial \Lambda_1$, and with nonvanishing differential there, and such that $\{ f, g \} < 0$ (this Poisson bracket is automatically nonzero, so this is just a choice of signs for $f$ and $g$). Write $a_0 = g^{-1} r |dh_1 \dots dh_{n-1} dg|^{1/2}$, where $r$ is a smooth section of $L_0$ and $h_1, \dots h_{d-1}$ are functions on $T^* X$ with independent differentials when restricted to $\Lambda_0 \cap \Lambda_1$. We map this to 
	$$
	Ra_0 := r |dh_1 \dots dh_{d-1} df|^{1/2} \{ g, f \}^{-1/2} |_{\Lambda_0 \cap \Lambda_1}.
	$$
	This is well-defined independent of choices. Moreover, it is shown in \cite{Melrose-Uhlmann-CPAM-1979} that $L_0$ and $L_1$ are canonically isomorphic over $\Lambda_0 \cap \Lambda_1$, so $Ra_0$ can be regarded as a section of $C^\infty(\Lambda_1, \Omega^{1/2}, L_1)$ restricted to $\partial \Lambda_1$. It is shown in \cite{Melrose-Uhlmann-CPAM-1979} in the homogeneous case (and the semiclassical case works just the same) that 
	\begin{equation}\label{0to1} a_1 |_{\partial \Lambda_1} = R a_0. \end{equation} Therefore, one may invariantly define $$S(\Lambda, \Omega^{1/2} \otimes L) \subset S(\Lambda_1 \setminus \partial \Lambda_1, \Omega^{1/2} \otimes L_1) \otimes S(\Lambda_0 \setminus \partial \Lambda_1, \Omega^{1/2} \otimes L_0)$$ consisting of the pairs of local principal symbols satisfying (\ref{0to1}), where $L$ denotes the global Maslov bundle well-defined on $\Lambda$ by $L_0$ and $L_1$.
	
	This principal symbol map gives rise to an exact sequence
	\begin{equation}\label{exactseq1} 0 \hookrightarrow I^{k - 1}(X, \Lambda_0, \Lambda_1) \hookrightarrow I^k(X, \Lambda_0, \Lambda_1) \longrightarrow S(\Lambda, \Omega^{1/2} \otimes L)/hS(\Lambda, \Omega^{1/2} \otimes L) \longrightarrow 0.
	\end{equation} 
	Likewise, the fact $I^{k - 1/2}(X, \Lambda_0) \subset I^k(X, \Lambda_0, \Lambda_1)$ gives another exact sequence \begin{equation}\label{exactseq2}0 \hookrightarrow I^{k - 1/2}(X, \Lambda_0) + I^{k - 1}(X, \Lambda_0, \Lambda_1) \hookrightarrow I^k(X, \Lambda_0, \Lambda_1) \longrightarrow S(\Lambda_1, \Omega^{1/2} \otimes L)/h S(\Lambda_1, \Omega^{1/2} \otimes L)\longrightarrow 0.\end{equation}

	\subsection{Symbol calculus}
	
	Noting (\ref{independence-F}), we apply Theorem \ref{composition-semiFIO} to semiclassical intersecting Lagrangian distribution $A \in I^{k_1}(X, \Lambda_0, \Lambda_1)$ and semiclassical Fourier integral operator $F \in I^{k_2}(X, Y, \mathcal{C})$ with $\mathcal{C}$ transverse to $\Lambda_0$ and $\Lambda_1$. Then \begin{equation}\label{composition-intersecting}F \circ A \in I^{k_1 + k_2}(Y, \mathcal{C} \circ \Lambda_0, \mathcal{C} \circ \Lambda_1).
	\end{equation}
	
	In particular, we are interested in the calculus of intersecting Lagrangian distributions with pseudodifferential operators of real principal type. More precisely,
	
	\begin{theorem}\label{principal symbol semiclassical intersecting}Let $P \in \Psi^{k}(X)$ be a semiclassical pseudodifferential operator of real principal type with sub-principal symbol $s$ and principal symbol $p$ vanishing on $\Lambda_1$ and $A \in I^{k^\prime}(X, \Lambda_0, \Lambda_1, \Omega^{1/2})$. 
	%Assume that $H_p$, the Hamilton vector field of $p$, is transverse to $\Lambda_0$ at $\Lambda_0 \cap \Lambda_1$. 
	Then $PA$ can be written as a sum of $F \in I^{k + k^\prime - 1/2}(X, \Lambda_0)$ and $G \in I^{k + k^\prime - 1}(X, \Lambda_0, \Lambda_1)$ with a half density principal symbol on $\Lambda_1$ given by 
	\begin{equation}\label{vanishing-symbol-calculus} 
	\sigma^{k+k'-1}(G) = \frac{h}{i} \mathcal{L}_{H_{p}} (a) + s a.
	\end{equation} 
	Here $H_{p}$ is the Hamilton field of $p$ and $\mathcal{L}$ denotes the Lie derivative of half densities.
	\end{theorem}  
	
	Indeed, the theorem is an immediate result of  (\ref{exactseq2}) and \eqref{composition-intersecting}.
	
	We now construct a global parametrix for a semiclassical pseudodifferential operator $P \in \Psi^{k}(X)$ of real principal type on a manifold $X$. Let $(\Lambda_0, \Lambda_1)$ be an embedded intersecting Lagrangian pair of $T^\ast X$, and assume that 
	\begin{itemize}
	\item  $H_p$ is nowhere tangent to $\Lambda_0$ on $\Lambda_0 \cap \Sigma (P)$ with $\Sigma (P) = \{(x, \xi) \in T^\ast (X) \setminus 0 : p(x, \xi) = 0\}$, 
	\item $\Lambda_1$ is the forward flowout from $\Lambda_0 \cap \Sigma(P)$ by $H_p$, and 
	\item no complete bicharacteristic of $P$ lying in $\Lambda_1$ remains over a compact set in $X$. 
	\end{itemize}
	%
	% $\Lambda_1$ satisfies a non-return condition. We further assume the union of the forward flow-outs of $(\partial \Lambda_1)_j$ under $e_j H_p$ is an embedded Lagrangian submanifold with boundary, $\Lambda^e_1 \subset T^\ast \setminus 0,$ where the superscript $e$ denotes the collective orientations.
	
	\begin{theorem}\label{parametrix construction intersecting}Let a real principal type  pseudodifferential operator $P \in \Psi^k_h(X)$ and a pair of intersecting Lagrangian 
	submanifolds $(\Lambda_0, \Lambda_1)$ be given as above. For any Lagrangian distribution $F \in I^{k^\prime}(X, \Lambda_0, \Omega^{1/2})$, there is a solution $U \in I^{k^\prime - k + 1/2} (X, \Lambda_0, \Lambda_1,  \Omega^{1/2})$ to $P \circ U = F$ modulo $O(h^\infty)$.\end{theorem}
	
	\begin{proof} we shall solve $P \circ U = F$ symbolically.
	
	We seek $U_0 \in I^{k^\prime - k + 1/2} (X, \Lambda_0, \Lambda_1)$ such that $F - P U_0 = F_1 + G_1$ with  $F_1 \in I^{k^\prime - 1}(X, \Lambda_0)$   and $G_1 \in I^{k^\prime - 1 - 1/2}(X, \Lambda_0, \Lambda_1)$. Because the principal symbol of $F$ doesn't support on $\Lambda_1$, the second exact sequence (\ref{exactseq2}), symbol calculus (\ref{composition-intersecting}) and (\ref{vanishing-symbol-calculus}) require that 
	\begin{eqnarray*}&\sigma(U_0) = p^{-1} \sigma(F) &\quad\mbox{on $\Lambda_0$}\\ &0 = \frac{h}{i} \mathcal{L}_{H_{p}} (\sigma(U_0)) + s \sigma(U_0)  &\quad\mbox{on $\Lambda_1$}.
	\end{eqnarray*}  
	The relationship of symbols on an intersecting Lagrangian pair, (\ref{0to1}), requires that  the symbol at the boundary satisfies $$\sigma(U_0)|_{\partial \Lambda_1} = R(p^{-1} \sigma(F)) \quad\mbox{at $\partial \Lambda_1$}.$$
	This is a first order linear ODE on $\Lambda_1$ with initial condition at $\partial \Lambda_1$. 
	The geometric conditions listed above guarantee there is a unique solution to this ODE. Then there is a $U_0  \in I^{k^\prime - k + 1/2} (X, \Lambda_0, \Lambda_1)$ such that $$F - PU_0 = F_1 + G_1 \in I^{k^\prime - 1}(X, \Lambda_0) + I^{k^\prime - 1 - 1/2} (X, \Lambda_0, \Lambda_1).$$
	
	To gain $U_1 \in I^{k^\prime - k - 1 + 1/2} (X, \Lambda_0, \Lambda_1)$ such that $F_1 + G_1 - P U_1 = F_2 + G_2$ provided  $F_2 \in I^{k^\prime - 2}(X, \Lambda_0)$  and $G_2 \in I^{k^\prime - 2 - 1/2}(X, \Lambda_0, \Lambda_1)$,  we apply the first exact sequence (\ref{exactseq1}) and symbol calculus to get \begin{eqnarray*}&\sigma(U_1) = p^{-1} \sigma(F_1) &\quad\mbox{on $\Lambda_0$}\\ &\sigma(G_1) = \frac{h}{i} \mathcal{L}_{H_{p}} (\sigma(U_1)) + s \sigma(U_1)  &\quad\mbox{on $\Lambda_1$}.\end{eqnarray*} The boundary condition also holds $$\sigma(U_1)|_{\partial \Lambda_1} = R(p^{-1} \sigma(F_1)) \quad\mbox{at $\partial \Lambda_1$}.$$The geometric conditions again guarantee there is a unique solution to this system. Then there is a $U_1  \in I^{k^\prime - k - 1 + 1/2} (X, \Lambda_0, \Lambda_1)$ such that $$ F_1 + G_1 - PU_1= F_2 + G_2 \in I^{k^\prime - 2}(X, \Lambda_0) + I^{k^\prime - 2 - 1/2} (X, \Lambda_0, \Lambda_1).$$
	
	Repeating this procedure inductively, we have $$F - PU_0 = F_j + G_j \in I^{k^\prime - j}(X, \Lambda_0) + I^{k^\prime - j - 1/2} (X, \Lambda_0, \Lambda_1)$$ such that \begin{eqnarray*}&\sigma(U_j) = p^{-1} \sigma(F_j) &\quad \mbox{on $\Lambda_0$},\\ &\sigma(G_j) = \frac{h}{i} \mathcal{L}_{H_{p}} (\sigma(U_j)) + s \sigma(U_j) & \quad\mbox{on $\Lambda_1$}, \\ &\sigma(U_j) = R(p^{-1} \sigma(F_j)) &\quad\mbox{at $\partial \Lambda_1$}.
	\end{eqnarray*} Consequently, the parametrix $U$ constructed by asymptotically summing up $\{U_j\}$ satisfies  $$f - P\bigg(\sum_{j = 0}^N U_j\bigg) \in I^{k^\prime - N - 1}(X, \Lambda_0) + I^{k^\prime - N - 1 - 1/2} (X, \Lambda_0, \Lambda_1).$$
	We thus have $f - PU = O(h^\infty)$, as required. 
	\end{proof}

\end{appendix}

\begin{flushleft}
\vspace{0.3cm}\textsc{Xi Chen\\Mathematical Sciences
Institute\\Australian National University\\Canberra 0200 Australia}

\emph{E-mail address}: \textsf{xi.chen@anu.edu.au}

\end{flushleft}

\begin{flushleft}
\vspace{0.3cm}\textsc{Andrew Hassell\\Mathematical Sciences
Institute\\Australian National University\\Canberra 0200 Australia}

\emph{E-mail address}: \textsf{andrew.hassell@anu.edu.au}

\end{flushleft}


\begin{thebibliography}{99}

\bibitem{Arnold} V. I. Arnold, \emph{Mathematical methods of classical mechanics}, 2nd edition, translated by K. Vogtmann and A. Weinstein, Springer, 1989. 

%\bibitem{Burq-Guillarmou-Hassell}N. Burq, C. Guillarmou and A. Hassell, \emph{Strichartz estimates without loss on manifolds with hyperbolic trapped geodesics}, Geom. Funct. Anal. \textbf{20}(2010), 627-656.

%\bibitem{C-G-T}J. Cheeger, M. Gromov and M. Taylor, \emph{Finite propagation speed, kernel estimates for functions of the Laplace operator, and the geometry of complete Riemannian manifolds}, J. Diff. Geom. \textbf{17}(1982), 15-53.

\bibitem{Chen-Hassell3} X. Chen,  \emph{Resolvent and spectral measure on non-trapping asymptotically hyperbolic manifolds III:  Global-in-Time Strichartz Estimates without Loss}, preprint.

\bibitem{Chen-Hassell2} X. Chen and A. Hassell, \emph{Resolvent and spectral measure on non-trapping asymptotically hyperbolic manifolds II: Spectral Measure, Restriction Theorem, Spectral Multiplier}, preprint.



%\bibitem{Cheng-Li-Yau}S. Y. Cheng, P. Li and S. T. Yau, \emph{On the upper estimate of the heat kernel of a complete Riemannian manifold}, Amer. J. Math. \textbf{103}(1981), 1021-1063.

%\bibitem{C-S}J. L. Clerc and E. M. Stein, \emph{$L^p$-multipliers for noncompact symmetric spaces}, Proc. Nat. Acad. Sci. USA \textbf{71}(1974), No. \textbf{10}, 3911-3912.

\bibitem{Duistermaat}J. J. Duistermaat, \emph{Fourier Integral Operators}, Progress in Math. \textbf{130}. Birkhäuser Boston, 1996.

\bibitem{Duistermaat-Hormander-Acta-1972}J. J. Duistermaat and L. H\"{o}rmander, \emph{Fourier Integral Operators II}, Acta Math. \textbf{128}(1972), 183-269.

\bibitem{Egorov}Y. V. Egorov, \emph{Microlocal Analysis}. Partial differential equations IV: Microlocal Analysis and Hyperbolic Equations (Translated from Russia by P. C. Sinha), 1–147, Springer-Verlag, 1993.


%\bibitem{fourier analysis}L. Grafakos, \emph{Classical and Modern Fourier Analysis}, Prentice Hall, New Jersey, 2004.

\bibitem{Graham-Zworski}C. R. Graham and M. Zworski, \emph{Scattering matrix in conformal geometry}, Invent. Math. \textbf{152}(2003), 89-118.

\bibitem{Guillarmou}C. Guillarmou, \emph{Meromorphic properties of the resolvent on asymptotically hyperbolic manifolds}, Duke Math. J. \textbf{129}(2005), No. \textbf{1}, 1-37.

%\bibitem{Guillarmou-Hassell-Sikora}C. Guillarmou, A. Hassell and A. Sikora, \emph{Restriction and spectral multiplier theorems on asymptotically conic manifolds}, Anal. PDE \textbf{6}(2013), No.\textbf{4}, 893–950.

%\bibitem{Guillarmou-Naud}C. Guillarmou and F. Naud, \emph{Wave decay on convex co-compact hyperbolic manifolds}, Comm. Math. Phys. \textbf{287}(2009), 489-511.

%\bibitem{Guillarmou-Qing}C. Guillarmou and J. Qing, \emph{Spectral characterization of Poincar\'{e}-Einstein manifolds with infinity of positive Yamabe type}, Inter. Math. Res. Not. \textbf{2010}, No. \textbf{9}, 1720-1740.

%\bibitem{Hassell-Mazzeo-Melrose}A. Hassell, R. Mazzeo and R. B. Melrose, \emph{Analytic surgery and the accumulation of eigenvalues}, Comm. Anal. Geom. \textbf{3}(1995), No.\textbf{1}, 115-222.



%\bibitem{Hassell-Vasy}A. Hassell and A. Vasy, \emph{The spectral projections and the resolvent for scattering metrics}, J. Anal. Math. \textbf{79}(1999), 241-298.


\bibitem{Hassell-Vasy-FourierAnn}A. Hassell and A. Vasy, \emph{The resolvent for Laplace-type operators on asymptotically conic spaces}, Ann. Inst. Fourier (Grenoble) \textbf{51}(2001), No. \textbf{5}, 1229-1346.


\bibitem{Hassell-Wunsch-2005}A. Hassell and J. Wunsch, \emph{The Schr\"odinger  propagator for scattering metrics}, Ann. Math. \textbf{162}(2005), 487-523.

\bibitem{Hassell-Wunsch}A. Hassell and J. Wunsch, \emph{The semiclassical resolvent and the propagator for non-trapping scattering metrics}, Adv. Math. \textbf{217}(2008), 586-682.

%\bibitem{Hassell-Zhang}A. Hassell and J. Zhang, \emph{Global-in-time Strichartz estimates on non-trapping asymptotically conic manifolds}, preprint.

\bibitem{Hormander1}L. H\"{o}rmander, \emph{The Analysis of Linear Partial Differential Operators I: Distribution Theory and Fourier Analysis (Second Edition)}, Springer-Verlag, 1990.

\bibitem{Hormander3}L. H\"{o}rmander, \emph{The Analysis of Linear Partial Differential Operators III: Pseudo-Differential Operators}, Springer-Verlag, 1985.


\bibitem{Hormander4}L. H\"{o}rmander, \emph{The Analysis of Linear Partial Differential Operators IV: Fourier Integral Operators}, Springer-Verlag, 1994.

\bibitem{Hormander-Acta-1971}L. H\"{o}rmander, \emph{Fourier integral operators I}, Acta Math. \textbf{127}(1971), 79-183.

%\bibitem{Joshi-Sa Barreto}M. Joshi and A. S\'{a} Barreto, \emph{Inverse scattering on asymptotically hyperbolic manifolds}, Acta Math. \textbf{184}(2000), 41-86.

\bibitem{Jost}J. Jost, \emph{Riemannian Geometry and Geometric Analysis (Second Edition)}, Springer-Verlag, 1998.


\bibitem{Martinez}A. Martinez, \emph{An Introduction to Semiclassical and Microlocal Analysis}, Springer-Verlag, 2002.


\bibitem{Mazzeo-JDG-1988}R. Mazzeo, \emph{The Hodge cohomology of a conformally compact metric},  J. Diff. Geom. \textbf{28}(1988), 309-339.

\bibitem{Mazzeo-1991}R. Mazzeo, \emph{Unique continuation at infinity and embedded eigenvalues for asymptotically hyperbolic manifolds},  
Amer. J. Math. \textbf{113} (1991), no. 1, 25--45. 

\bibitem{Mazzeo-Melrose}R. Mazzeo and R. B. Melrose, \emph{Meromorphic extention of the resolvent on complete spaces with asymptotically constant negative curvature}, J. Func. Anal. \textbf{75}(1987), 260-310.

\bibitem{DiffAnal}R. B. Melrose, \emph{Differential Analysis on Manifolds with Corners}, manuscript, http://www-math.mit.edu/~rbm/book.html. 

\bibitem{Melrose-book}R. B. Melrose, \emph{Geometric Scattering Theory}, Cambridge Univ. Press, 1995.


%\bibitem{Melrose-ACTA1981}R. B. Melrose, \emph{Transformation of boundary problems}, Acta Math. \textbf{147}(1981), 149-236.

%\bibitem{Melrose-ICM1990}R. B. Melrose, \emph{Pseudodifferential operators, corners and singular limits}, Proc. the Internat. Congress of Math. Vol. I, II (Kyoto, 1990), 217-234, Math. Soc. Japan, Tokyo, 1991.

%\bibitem{Melrose-IMRN1992}R. B. Melrose, \emph{Calculus of conormal distributions on manifolds with corners}, Internat. Math. Res. Notices 1992, No.\textbf{3}, 51-61.

\bibitem{Melrose-Sa Barreto-Vasy}R. B. Melrose, A. S\'{a} Barreto and A. Vasy, \emph{Analytic continuation and semiclassical resolvent estimates on asymptotically hyperbolic spaces,}  Comm. Part. Diff. Equa. \textbf{39}(2014), 452-511.

\bibitem{Melrose-Uhlmann-CPAM-1979}R. B. Melrose and G. A. Uhlmann, \emph{Lagrangian intersection and the Cauchy problem}, Comm. Pure Appl. Math. \textbf{32}(1979), 483-519.

%\bibitem{Melrose-Zworski}R. B. Melrose and M. Zworski, \emph{Scattering metrics and geodesic flow at infinity}, Invent. Math. \textbf{124}(1996), 389-436.

%\bibitem{R-S}M. Reed and B. Simon, \emph{Methods of Modern Mathematical Physics:} \rm{I} \emph{Functional Analysis. (Revised and Enlarged Edition)}, Academic Press, 1980, New York.

\bibitem{SBY} A. Sa Barreto, Y. Wang, \emph{The Scattering Relation on Asymptotically Hyperbolic Manifolds}, 
arXiv:1410.6936. 

\bibitem{sogge}C. D. Sogge, \emph{Fourier Integrals in Classical Analysis}, Cambridge, 1993.

%\bibitem{interpolation}E. M. Stein, \emph{Interpolation of linear operators}, Trans. Amer. Math. Soc. \textbf{83}(1956), 482-492.



%\bibitem{Beijing lecture}E. M. Stein, \emph{Oscillatory integrals in Fourier analysis}, Beijing Lectures in Harmonic Analysis, Ann. of Math. Stud. \textbf{112}, 307-355, Princeton Univ. Press, Princeton, NJ, 1986.

%\bibitem{T1}M. Taylor, \emph{$L^p$-estimates on functions of the Laplace operator}, Duke Math. J. \textbf{58}(1989), No. \textbf{3}, 773-793.


\bibitem{Taylor2}M. Taylor, \emph{Partial Differential Equations II: Qualitative Studies of Linear Equations}, Springer-Verlag, 1996.

%\bibitem{Tomas}P. Tomas, \emph{A restriction theorem for the Fourier transform}, Bull. Amer. Math. Soc. \textbf{81}(1975), 477-478.

\bibitem{Vasy} A. Vasy, \emph{Microlocal analysis of asymptotically hyperbolic spaces and high energy resolvent estimates}, in  `Inverse problems and applications. Inside Out II', Gunther Uhlmann (ed.), Cambridge University Press, MSRI Publications, no. 60 (2012).

\bibitem{Vasy-invent-2013} A. Vasy, \emph{Microlocal analysis of asymptotically hyperbolic and Kerr-de Sitter spaces (with an appendix by Semyon Dyatlov)} \textbf{194}(2013), 381-513.

\bibitem{Wang}Y. Wang, \emph{Resolvent and radiation fields on asymptotically hyperbolic manifolds}, arXiv:1410.6936. 

\bibitem{zworski}M. Zworski, \emph{Semiclassical Analysis}, Grad. Stud. Math.\textbf{138}, Amer. Math. Soc., Providence, RI, 2012.







\end{thebibliography}
\end{document}